\documentclass[12pt,english,reqno]{amsart}
\usepackage{amscd,amssymb,amsmath,amstext,amsthm,amsfonts,bbold,wasysym}
\usepackage[dvips]{graphicx}
\usepackage{mathrsfs}
\usepackage{float}%Para as figuras ficarem fixadas [H]

\usepackage{stmaryrd}
\usepackage{a4wide}
\usepackage[pdftex]{hyperref}%%%PARA FAZER LINK NO PDF, mas para MiKTeKs antigos devemos desmarcar para o DVI%%%
\usepackage[usenames]{color}
\usepackage[ansinew]{inputenc}

\newtheorem{Theorem}{Theorem}
\newtheorem*{theorem}{Theorem}

\newtheorem{maintheorem}{Theorem}

\newtheorem{T}{Theorem}
\newtheorem{Corollary}[T]{Corollary}
\newtheorem{Proposition}[T]{Proposition}
\newtheorem{Lemma}[T]{Lemma}

\newtheorem{Definition}[T]{Definition}
\newtheorem*{definition}{Definition}
\newtheorem{Example}[T]{Example}
\newtheorem*{claim}{Claim}
\newtheorem*{subclaim}{Subclaim}
\newtheorem{Claim}{Claim}

\def \BB {{\mathbb B}}

\def \RR {{\mathbb R}}

\def \NN {{\mathbb N}}

\def \XX {{\mathbb X}}

\def \RE {{\mathrm{Re}}}

\def \cf {\mathcal{F}}

\def \cp {\mathcal{P}}
\def \cc {\mathcal{C}}

\def \cu {\mathcal{U}}
\def \co {\mathcal{O}}

\def \cv {\mathcal{V}}
\def \cw {\mathcal{W}}

\newcommand{\Agemo}{\operatorname{\text{\agemO}}}

\newcommand{\leb}{\operatorname{Leb}}

\newcommand{\per}{\operatorname{Per}}

\newcommand{\interior}{\operatorname{Interior}}
\newcommand{\fix}{\operatorname{Fix}}

\begin{document}

\author{P. Brand\~ao}
\address{P. Brand\~ao\\
Impa, Estrada Dona Castorina, 110, 22460-320 Rio de Janeiro, Brazil.} \email{paulo@impa.br}

\author{J. Palis}
\address{J. Palis\\
Impa, Estrada Dona Castorina, 110, 22460-320 Rio de Janeiro, Brazil.} \email{jpalis@impa.br}

\author{V. Pinheiro}
\address{V. Pinheiro\\
Departamento de Matem\'atica, Universidade Federal da Bahia\\
Av. Ademar de Barros s/n, 40170-110 Salvador, Brazil.}
\email{viltonj@ufba.br}

\date{\today}

%\thanks{Work carried out at IMPA. Partially supported by CNPq, IMPA and PROCAD/CAPES}

\title{On the Finiteness of Attractors for One-Dimensional Maps with Discontinuities}

\maketitle
\begin{abstract}

In the present paper we show that  piecewise $C^3$ maps of the interval with negative Schwarzian derivative display only a finite number of attractors, extending results in \cite{BL89eg, BL91, Ly91, vSV, StP}. We also give a more precise description of the attractors for contracting Lorenz maps.

\end{abstract}

\tableofcontents

\section{Introduction}

Attractors play a fundamental role in the study of dynamical systems for the understanding of future evolution of initial states. Along this line, stressing the importance of attractors, it was conjectured by Palis in 1995 (see \cite{Pa00,Pa2005}) that, in compact smooth manifolds, there is a dense set $D$ of differentiable dynamics such that, among other properties, any element of $D$ display finitely many attractors whose union of basins of attraction has total probability in the ambient manifold. The conjecture was built in such a way that, if proved to be true, one can then concentrate the attention on the description of the properties of these finitely many attractors and their basins of attraction to have an understanding on the whole system.

This kind of scenario appeared earlier in the works of Andronov-Pontrjagin \cite{AP} for flows on 2-dimensional discs transversal to the boundary, Peixoto \cite{Pe} for smooth flows on compact orientable surfaces and Palis-Smale \cite{PaSm} for smooth gradient flows in higher dimensions. Besides these results, the corresponding framework for flows in dimension three or more is wide open. For diffeomorphisms  in two or higher dimensions the question is in general equally wide open, except for $C^1$ diffeomorphisms on compact boundaryless 3-dimensional manifolds, see Crovisier-Pujals \cite{CP14}. On the way to prove the above conjecture, one has to face the existence of open sets in the space of diffeomorphisms containing residual sets of dynamics displaying infinitely many attractors, as shown by Newhouse \cite{Ne}. This fact does not contradict the existence of a dense set of dynamics having only a finite number of attractors, but certainly shows how delicate such a question is, see \cite{Ka, NP, PT, PaYo}.

On the other hand, for $C^r$ maps of the interval, $r\ge1$, the above conjecture has been proved by Kozlovski-Shen-van Strien in \cite{KSS}, as they have shown that the hyperbolic maps are dense.
%A $C^r$ map $f$ of the interval $[0,1]$, is called {\em hyperbolic} if it has only a finite number of hyperbolic periodic attractors and $[0,1]\setminus\BB^*$ is an uniformly expanding set, where $\BB^*$ is the union of the basins of attraction of the hyperbolic periodic attractors. As every uniformly expanding set of the interval has zero  measure, Lebesgue almost every point is attracted by one of the hyperbolic periodic attractors. Therefore, for maps of the interval, $D$ in the main conjecture above can be chosen as the set of hyperbolic maps.
The denseness of hyperbolic maps had been previously established for the real quadratic case by Lyubich in \cite{lyubich:1997ha} and for the parametrized logistic family $f_t(x)=4 t x(1-x)$ by Graczyk-Swiatek in \cite{GS97}. We observe that quadratic maps, in particular the logistic one, have negative Schwarzian derivative, as defined in Section \ref{SubsectionStaofMTh}.

%In the case of maps with negative Schwarzian derivative we definitively can go much further: the finiteness of attractors is valid for all of them.
%
%
%{\color{red}
%
%In the context of one-dimensional dynamics, negative Schwarzian condition is a suitable condition to avoid technical complications and so, to focus on main ideas of the proofs. Indeed, several important theorems in this field have been proved firstly under the assumption of negative Schwarzian derivative.
%
%
%
%}%color

%For maps with negative Schwarzian derivative, we definitively can go much further. Indeed, 

For one-dimensional dynamics, due to the simplicity of the topology, one may ask if the finiteness of the number of attractors is a condition fulfilled by all maps, instead of only a dense subset of dynamics as conjectured for higher dimensions. The question has a negative answer, since maps that do not have negative Schwarzian derivative may display an infinite number of the so-called inessential attracting periodic orbits. So, the problem here is if we always have a finite number of attractors for maps with negative Schwarzian derivative, even when we allow for discontinuities. 

A $C^2$ map of the interval is said to be {\em non-flat} if for each critical point $c$ of the map, there is a local diffeomorphism $\Phi_c \in C^2$, and $\alpha >1$ such that $\Phi_c^{-1} \circ f \circ \Phi_c(x)=\pm |x|^\alpha$ for all $x$ close to $c$.

We point out that maps of the interval with discontinuities naturally arise from  vector fields. Indeed, piecewise $C^r$ maps, $r \ge 1$, that are piecewise monotone and non-flat can be obtained as the quotient by stable manifolds of a Poincar\'e map of some $C^r$ dissipative flow. In Figure \ref{fluxoDuasSing.png} we have sketched a flow giving rise to a piecewise $C^r$ one-dimensional map with two discontinuities, see \cite{Gambaudo:1986p2422,GW,Rov93}.

%It can, for instance, appear as the quotient by stable manifolds of Poincaré maps of $C^r$ dissipative flows with singularities induced by contracting Lorenz flows, $r \ge 1$, as in the Figure \ref{fluxoDuasSing.png}, that induces in a suitably chosen square $Q$, as shown in Figure \ref{unidimensional.png}, a map as the ones being studied by us (see \cite{Gambaudo:1986p2422,Rov93}). 

The finiteness of the number of attractors for $C^3$ non-flat maps of the interval with negative Schwarzian derivative began to be established by the pioneering work of Blokh-Lyubich in \cite{BL89eg, BL91} and Lyubich \cite{Ly89}.
They have shown that  if $f:[0,1]\to[0,1]$ is a  $C^3$ non-flat map
%(near any critical point $f$ looks like a polynomial map, see \cite{MvS})
with a single critical point and negative Schwarzian derivative, then there is an attractor whose basin of attraction contains Lebesgue almost every point of the interval $[0,1]$.
Main contributions were also due to de Melo-van Strien \cite{MS89}, Guckenheimer-Johnson \cite{GJ}, Keller \cite{Ke}, Graczyk-Sands-Swiatek \cite{SGS} and others.

%In \cite{vSV}, van Strien and Vargas proved that every $C^3$ non-flat map $f$ with a finite number of critical points has a finite number of non-periodic attractors. Thus, if $f$ has also negative Schwarzian derivative, one can conclude the finiteness of attractors for the map (see also \cite{Ly91}).

For smooth non-flat maps of the interval, Lyubich \cite{Ly91} proved the finiteness of non-periodic attractors. More recently, van Strien-Vargas \cite{vSV} sharpened the classification of these attractors.

In contrast with the smooth case, the problem of finiteness of the number of attractors remained open for maps with discontinuities or other kind of lack of regularity, and this is exactly what is treated in the present paper. Former partial answers to this question include Lorenz maps: Guckenheimer-Williams solved the expanding case in \cite{GW} and, later, St. Pierre in \cite{StP}, the contracting one.

The main result in this paper is that every piecewise $C^3$ map with negative Schwarzian derivative can only display a finite number of attractors. Furthermore, the union of the basins of attraction contains Lebesgue almost every point. We do so without having to decompose the domain of the map into ergodic components, and this is a key fact in our work. Former proofs of the finiteness of the number of attractors usually went through showing the non-existence of wandering intervals and the decomposition of the domain into ergodic components. However, for maps with discontinuities that are at the same time critical points, the existence of wandering intervals is in general an open question. And, although wandering intervals are an obstruction to the decomposition into ergodic components, the methods presented here still allow us to bound the number of attractors.

%Notice that the negative Schwarzian derivative hypothesis is important to avoid trivial examples of maps with infinitely many periodic attractors. In fact, even $C^{\infty}$ diffeomorphisms of the interval not fulfilling this condition can have infinitely many attracting fixed points.

%\subsection{Historical context}

%%%%%%%%%%%%%%%%%%%%%%%%%%%%%%%%%%%%%%%%%%%%%%%
%\begin{figure}
%\begin{center}\includegraphics[scale=.3]{Criticos-singular.pdf}
%\caption{Some examples of points belonging to the exceptional set: $c_1$, $c_2$ and $c_5$ are discontinuities, $c_3$ and $c_4$ are non-regular turning points, $c_6$ is a non-regular point, $c_7$ is a turning point and $c_8$ is a cubic-like critical point}\label{Criticos-singular.pdf}
%\end{center}
%\end{figure}
%%%%%%%%%%%%%%%%%%%%%%%%%%%%%%%%%%%%%%%%%%%%%%%

%%%%%%%%%%%%%%%%%%%%%%%%%%%%%%%%%%%%%%%%%%%%%%
\begin{figure}
\begin{center}\includegraphics[scale=.3]{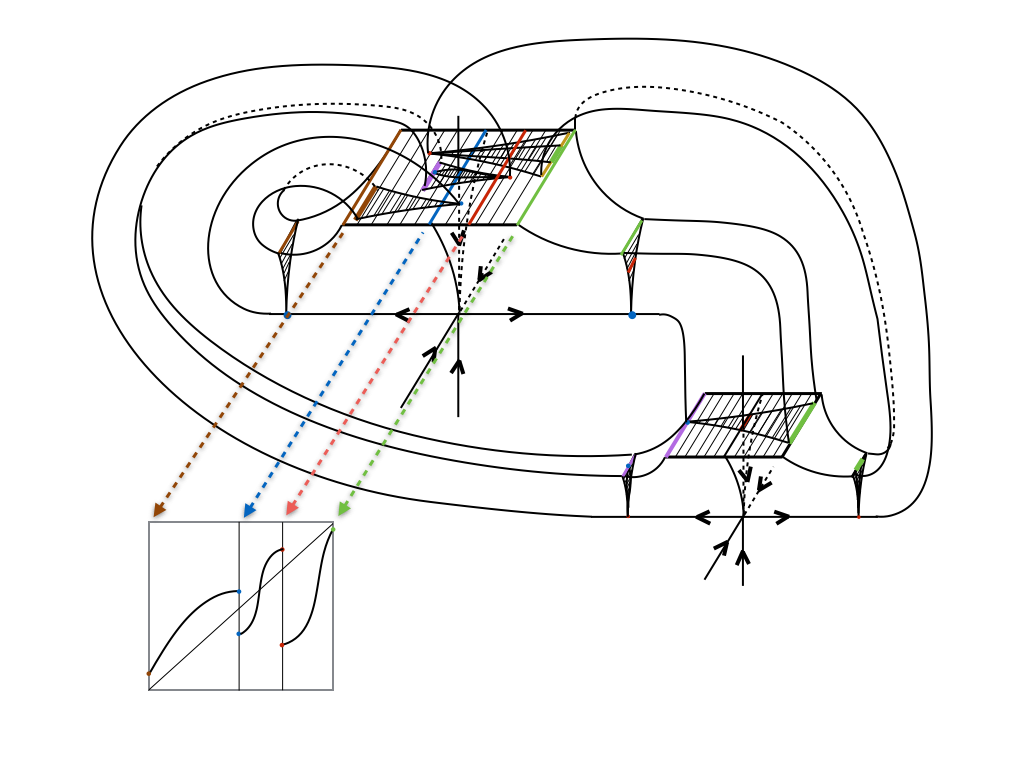}
\caption{An example of a flow displaying two coupled singularities and inducing an one-dimensional map with two discontinuities.}\label{fluxoDuasSing.png}
\end{center}
\end{figure}
%%%%%%%%%%%%%%%%%%%%%%%%%%%%%%%%%%%%%%%%%%%%%%

\subsection{Statement of the main results}\label{SubsectionStaofMTh}

Given a smooth compact manifold $M$, possibly with boundary, we say, following Milnor \cite{Milnor:1985ut}, that a compact subset $A\subset M$ is a {\em metrical attractor} for a map $f: M \to M$ if its basin of attraction $\beta_{f}(A):=\{x\,;\,\omega_f(x)\subset A\}$ has positive Lebesgue measure and there is no strictly smaller closed set $A' \subset A$ so that $\beta_{f}(A')$ is the same as $\beta_{f}(A)$ up to a measure zero set. Here, $\omega_f(x)$ is the positive limit set of $x$, that is, the set of accumulating points of the forward orbit of $x$.

An attractor is called {\em minimal} if $\leb(\beta_f(A'))=0$ for every compact set $A'\subsetneqq A$. If $A$ is a minimal attractor, then $\omega_f(x)=A\text{ for almost all }x\in\beta_f(A)$.
%Another relevant kind of attractor is the so called {\em topological attractor}~:\, the basin of attraction contains  a residual subset of an open set of $M$.

%A map $f:[0,1]\to[0,1]$ is piecewise $C^3$ with negative Schwarzian derivative if there exists a finite set $\{p_1,...,p_s\}\subset[0,1]$ such that $f$ is $C^3$ in $[0,1]\setminus\{p_1,...,p_s\}$ with $f'(x)\ne0$ and $Sf(x)<0$ for every $x\in [0,1]\setminus\{p_1,...,p_s\}$, where the Schwarzian derivative $Sf(x)$ of  $f$ at $x$ is defined  by $$Sf(x)=\frac{f'''(x)}{f'(x)}-\frac{3}{2}\bigg(\frac{f''(x)}{f'(x)}\bigg)^2.$$

In this work, we deal with piecewise $C^3$ maps of the interval into itself with negative Schwarzian derivative. More precisely, we consider maps $f:[0,1] \to [0,1]$ that are $C^3$ local diffeomorphisms with negative Schwarzian derivative, i.e., $Sf(x)=\frac{f'''(x)}{f'(x)}-\frac{3}{2}\big(\frac{f''(x)}{f'(x)}\big)^2<0$, in the whole interval, except for a finite set of points $\cc_f\subset (0,1)$. This {\em exceptional set} contains all the critical points of $f$, which may or may not be flat, as well as discontinuities of the map.
%and other non-regular points.
%Any such map is called {\em S-multimodal}.
%In Figure \ref{Criticos-singular.pdf} we indicate several kind of points belonging to the exceptional set. 

%The main purpose of this article is to prove  Theorems A, B and C below.

% {\em super-attractor} or a ,The name super-attractor comes from the fact that if $x\in\beta_f(A)$, $A$ being the super-attractor, then its Lyapunov exponent is minus infinity, i.e., $\lim\frac{1}{n}\log|(f^n)'(x)|=-\infty$.

%%%%%%%%%%%%%%%%%%%%%%%%%%%%%%%%%%%%%%%%%%%%%%
%\begin{figure}
%\begin{center}\includegraphics[scale=.2]{LorenzMap.pdf}
%\caption{A contracting Lorenz map}\label{LorenzMap.pdf}
%\end{center}
%\end{figure}
%%%%%%%%%%%%%%%%%%%%%%%%%%%%%%%%%%%%%%%%%%%%%%

\begin{maintheorem}\label{mainTheoremMTheoFofA}
Let $f:[0,1]\to[0,1]$, be a $C^3$ local diffeomorphism with negative Schwarzian derivative in the whole interval, except for a finite set $\cc_f\subset(0,1)$. Then, there is a finite collection of  attractors $A_1,\cdots,A_n$, such that $$\leb(\beta_f(A_1)\cup\cdots\cup\beta_f(A_n))=1.$$
Furthermore, for almost all points $x$, we have $\omega_f(x)=A_j$ for some $j=1,\cdots,n$.
\end{maintheorem}

%It is to be noted that the theorem above completes the program set up by the Palis Conjecture for maps of the interval with a finite number of discontinuities under  the assumption of negative Schwarzian derivative.

Notice that we are not assuming the critical points to be non-flat, as previously assumed in  \cite{Ly91,vSV}. In particular, we are not requiring the map to have any symmetry around any given critical point: after a change of coordinates, the degree of the map at the critical point does not have to be the same at the left and right sides.

To state Theorem~\ref{mainTheoremMTheoLORB} below, we introduce the notion of cycle of intervals, namely a transitive finite union of non-trivial closed intervals. It is a common type of attractor for maps of the interval associated to the existence of an absolutely continuous invariant measure.

%%%%%%%%%%%%%%%%%%%%%%%%%%%%%%%%%%%%%%%%%%%%%%%
%\begin{figure}
%\begin{center}\includegraphics[scale=.2]{FlatUnimodal.pdf}\\
%\caption{Let $0<t\le1$ and $f_t:[0,1]\to[0,1]$ be given by $f_t(1/2)=t$ and $f_t(x)=t(1-e^{2-1/|x-1/2|})$ if $x\ne1/2$. The only critical point of $f_t$ is $c=1/2$. In the picture, we draw the graphic of $f_t$ with $t=0.9$. Notice that $f_t$ is $C^{\infty}$, $f_t^{(n)}(1/2)=0$ $\forall\,n\ge1$ and $Sf(x)=-(8/(1 - 2 x)^4)$. Thus, $f_t$ is a family of flat $S$-unimodal maps.}\label{FlatUnimodal.pdf}
%\end{center}
%\end{figure}
%%%%%%%%%%%%%%%%%%%%%%%%%%%%%%%%%%%%%%%%%%%%%%%

A point $p\in[0,1]$ is called {\em right-periodic} with period $n$ for $f$ if $n$ is the smallest integer $\ell\ge1$ such that $p=\lim_{0<\varepsilon\to0}\sup\{f^{\ell}((p-\varepsilon,p))\cap[0,p)\}$.
Similarly, we define {\em left-periodic} points and a point $p$ is {\em periodic-like} if it is left or right-periodic.

We shall deal with two types of {\em finite minimal attractors}: ordinary {\em attracting periodic orbits}, and the ones that contain at least one point of the exceptional set $\cc_f$: the {\em attracting periodic-like orbits}.

%Thus, in Figure~\ref{PeriodicLikeAttractor.pdf}, we have that $A=\{c\}$ is a periodic-like attractor with $\beta_f(A)=(0,1)$. Furthermore, as $c=\lim_{0<\varepsilon\to0}$ $\sup$ $\{f((c-\varepsilon,c))\cap[0,c)\}$, it follows that $c$ is a fixed-like attracting periodic point.

%%%%%%%%%%%%%%%%%%%%%%%%%%%%%%%%%%%%%%%%%%%%%%%
%\begin{figure}
%\begin{center}\includegraphics[scale=.2]{PeriodicLikeAttractor.png}\\
%\caption{An example of an attracting periodic-like orbit (indeed, a fixed-like point) which is not an attracting periodic orbit.}\label{PeriodicLikeAttractor.pdf}
%\end{center}
%\end{figure}
%%%%%%%%%%%%%%%%%%%%%%%%%%%%%%%%%%%%%%%%%%%%%%%

When we modify the eigenvalues $\lambda_2<\lambda_3<0<\lambda_1$ of the geometric Lorenz attractor \cite{ACT}, replacing the usual expanding condition $\lambda_3+\lambda_1>0$ by a contracting condition $\lambda_3+\lambda_1<0$, it gives rise to a flow that induces a one-dimensional map called contracting Lorenz map.

We then define a {\em contracting Lorenz map} as an orientation preserving $C^2$ local diffeomorphism $f:[0,1]\setminus\{c\}\to[0,1]$, $0< c< 1$, such that $0$ and $1$ are fixed points and $f$ has no repelling fixed points in $(0,1)$.
If such a map is $C^3$ with negative Schwarzian derivative and has no periodic-like attractors, we show that it has only a single attractor whose basin of attraction has full Lebesgue measure.

Notice that our result extends the one in \cite{StP}, as we do not require the non-flatness condition of the critical point, and such a condition is implicitly assumed in Lemma 3.36  in that work.

\begin{maintheorem}
\label{mainTheoremMTheoLORB}
Let $f$ be a $C^3$ contracting Lorenz map $f:[0,1]\setminus\{c\}\to[0,1]$, $c\in(0,1)$, with negative Schwarzian derivative. If $f$ does not have periodic-like attractors, then $f$ has an  attractor $A$ such that $\omega_f(x)=A$ for almost every $x\in[0,1]$. In particular, $\leb(\beta_{f}(A))=1$.

Furthermore, $f$ can have at most two periodic-like attractors. If $f$ has a single periodic attractor, its basin of attraction has full Lebesgue measure. The case that $f$ has two periodic-like attractors, the union of their basins of attraction has full Lebesgue measure.

If $f$ does not have periodic-like attractors, then $A$ is either a cycle of intervals or a transitive Cantor set.

If $A$ is a Cantor set, then $A=\omega_f(c_-)$ or $\omega_f(c_+)$. Moreover, if $f$ is non-flat (see Definition~\ref{Non-flat} in the Appendix) and $A$ is a Cantor set, then $c_-$ and $c_+\in A=\omega_f(c_-)=\omega_f(c_+)$.
\end{maintheorem}

\subsection{An outline of the paper} 

In Section \ref{SubsectionStaofMTh} we have stated the main results of the paper, namely Theorems A and B. In Section \ref{Sectionsettingsandpre}, we present some  basic results in one-dimensional dynamics %like Koebe's Lemma, Singer's and Mañe's Theorems, as well as the Homterval Lemma.
and introduce some notation to deal with lateral limits and lateral periodic orbits. We also present some facts on cycles of intervals.

In Section \ref{Sectionmarkovmaps}, we prove the Interval Dichotomy Lemma (Lemma \ref{tudoounadadershchw}), which is one of the new main ingredients of the paper. Firstly, we make a brief study of the ergodicity with respect to Lebesgue measure of complete Markov maps adapted to our context (Lemmas~\ref{pagina1} and \ref{Lemma375945194558}) and use it to prove Lemma \ref{tudoounadadershchw}. As a consequence, the $\omega$-limit set of almost every point $x\in[0,1]$ outside the basins of attraction of the periodic-like attracting orbits and the cycles of intervals, is contained in the closure of the union of the future orbits of the exceptional set (Corollary~\ref{Cor09876111}). This fact leads to a restriction on the locus of the $\omega$-limit of almost every point and this  information is crucial in the subsequent sections.

In Section~\ref{SecMainProof}, we make use of a sort of ``parallax argument'', that is, we evaluate the position of an object by comparing two different viewpoints of this object. In our context, the object to be  considered is the $\omega$-limit set of typical orbits, indeed some conveniently chosen subset  $V$ of it. The first viewpoint is given by the original map $f$. To get the second one, we construct a suitable auxiliary map $g$ that has a distinct exceptional set $\cc_g\ne\cc_f$, but keeps unchanged the $\omega$-limit set of the points in $V$.
 Comparing the $\omega$-limit set of the points of $V$ with respect to $f$ and $g$, we are able to make precise the $\omega$-limit set of the points that are not attracted by periodic-like attracting orbits or cycles of interval (Theorem~\ref{Theorem1aa}). With that, we prove Theorem~\ref{mainTheoremMTheoFofA}. 

%{\color{red}We make use in some proofs in this section of a sort of ``parallax argument''. Parallax is a technique to make precise the position of an object by comparing the projections of this object on the background from two different viewpoints. In our context, the object to be  considered is the $\omega$-limit set of typical orbits, indeed some conveniently chosen subset  $V$ of it, the background is the interval $[0,1]$ and the projection is given by Corollary~\ref{Cor09876111}. The first viewpoint is given by the original map $f$. To get the second one, we construct a suitable auxiliary map $g$ that has a distinct exceptional set $\cc_g\ne\cc_f$, but keeps unchanged the $\omega$-limit set of the points in $V$. Applying Corollary~\ref{Cor09876111} to both $f$ and $g$, and comparing the $\omega$-limit set with respect to the closure of the exceptional sets of $f$ and $g$, we are able to make precise the $\omega$-limit set of the points that are not attracted by periodic-like attracting orbits or cycles of interval (Theorem~\ref{Theorem1aa}). With that, we prove Theorem~\ref{mainTheoremMTheoFofA}. }

The least section, Section~\ref{SectionContractingLorenzmaps}, is dedicated to the contracting Lorenz maps and the proof of Theorem~\ref{mainTheoremMTheoLORB}. We have as its main ingredient  the induced map provided by Lemma~\ref{Lemma90102555}.

%In Section~\ref{SectionUnimodalmaps}, we apply the Interval Dichotomy Lemma to provide a proof of the uniqueness of the attractor of a $S$-unimodal map (Theorem~\ref{sunimodal}), even when this map exhibits a flat critical point. 

%This induced map plays an essential role to prove the section's main results. In the second part we study the contracting Lorenz maps without periodic-like attractor, proving the uniqueness of the attractor for those maps (Corollary~\ref{CororllaryWithoutPerAt}). In the third part we study the basin of attraction of the periodic-like attractors, concluding the proof of Theorem~\ref{mainTheoremMTheoLORB}.

\section{Notation, setting and preliminary facts}
\label{Sectionsettingsandpre}

Let $f:[0,1]\setminus\cc_f\to[0,1]$ be a $C^1$ local diffeomorphism. Without loss of generality, we can assume that $\cc_f\subset(0,1)$ and that $f(\{0,1\})\subset\{0,1\}$. Indeed, we can extend $f$ to a $C^3$ map $g;[-1,2]\setminus\cc_g\to[-1,2]$, with $\cc_g=\{0,1\}\cup\cc_f$, $Sg<0$, $g(\{-1,2\})\subset\{-1,2\}$ and $\omega_f(x)\subset[0,1]$ for every $x\in(-1,0)\cap(1,2)$. Notice that, in this case, the  attractors of $f$ and $g$ are the same.

  Given a set $U\subset[0,1]$, the pre-orbit of $U$ is $\co_f^-(U):=\bigcup_{j\ge0}f^{-j}(U)$, where $f^{-1}(U)=\{x\in[0,1]\setminus\cc_f\,;\,f(x)\in U\}$ and, for $j\ge1$, $f^{-j}(U)$ is inductively given by $f^{-j}(U)=f^{-1}(f^{-(j-1)}(U))$.

  % If $x\notin\co_f^-(\cc_f)$, the forward orbit of $x$ is $\co_f^+(x)=\{f^j(x)\,;\,j\ge0\}$ and the omega limit set of $x$, denoted by $\omega_f(x)$, is the set of accumulating points of the sequence $\{f^n(x)\}_n$: $$\omega_f(x)=\{y\in[0,1]\,;\,y=\lim_{j\to\infty}f^{n_j}(x)\text{ for some }n_j\to\infty\}.$$

%If $x\in\co_f^-(\cc_f)$ then $\co_f^+(x)=\{x,\cdots,f^n(x)\}$ where $n=\min\{j\ge0\,;\,x\in f^{-j}(\cc_f)\}$.
%A point $p\in[0,1]$ is called {\em wandering} if there is a neighborhood $V$ of it such that $\co_f^+(V)\cap V=\emptyset$, where $\co_f^+(V)=\bigcup_{x\in V}\co_f^+(x)$. If this is not the case, the point $p$ is called {\em non-wandering}. The {\em non-wandering set} of $f$, $\Omega(f)$, is the set of all non-wandering points $x\in[0,1]$. One can easily prove that $\Omega(f)$ is compact and that $\omega_f(x)\subset\Omega(f)$ $\forall\,x$.

Given $p\in[0,1]$, we set
 $$f^n(p_\pm)=
\lim_{0<\varepsilon\searrow0}f^n(p\pm \varepsilon)$$
$$\co_f^+(p_\pm)=\{f^n(p_\pm)\,;\,n\ge0\}$$
and
$$\omega_f(p_\pm)=\{y\,;\,\exists\,n_j\to\infty\text{ s.t. }y=\lim_{j\to\infty}f^{n_j}(p_\pm)\}.$$

For any point $x\in[0,1]$, we define the {\em forward (or positive) orbit} of $x$ as $$\co_f^+(x)=\co_f^+(x_-)\cup\co_f^+(x_+).$$
Notice that, if $x\notin\co_f^-(\cc_f)$, then $\co_f^+(x)=\{f^n(x)\,;\,n\ge0\}$.

A point $p\in[0,1]$ is called {\em periodic} if there exists $n\ge1$ such that $f^n(x_-)=f^n(x_+)=x$. We denote the set of periodic points of $f$ by $\per(f)$.
%Notice that a point $p\in[0,1]$ is right-periodic (see Section~\ref{SubsectionStaofMTh}) if and only if $p_-\in\co_f^+(p_-)$. Analogously,  $p$ is left-periodic if and only if $p_+\in\co_f^+(p_+)$. Moreover, $p$ is periodic if and only if $p$ is both right and left-periodic. 

The {\em omega-limit set} of a point $p\in[0,1]$ is defined as $$\omega_f(p)=\omega_f(p_-)\cup\omega_f(p_+).$$

Given any $U\subset[0,1]$, we denote the forward orbit of $U$ by $\co_f^+(U)$, that is, $$\co_f^+(U)=\bigcup_{x\in U}\co_f^+(x).$$

A set $U\subset[0,1]$ is called a {\em forward (or positive) invariant set} if $\co_f^+(x)\subset U$ for all $x\in U$. It easy to check that $\omega_f(x)$ and also $\overline{\co_f^+(x)}$ are positive invariant sets for every $x\in[0,1]$.

Given $x\in[0,1]\setminus\co_f^-(\cc_f)$ and $p\in[0,1]$, we say that $$p_-\in\omega_f(x)\text{ if }p\in\overline{\co_f^+(x)\cap[0,p)}.$$ Similarly,  $$p_+\in\omega_f(x)\text{ if }p\in\overline{\co_f^+(x)\cap(p,1]}.$$

Given an open subset $U\subset[0,1]$ and a point $p\in[0,1]$, we say that $p_-\in U$ if $\exists\,\varepsilon>0$ such that $(p-\varepsilon,p)\subset U$. Similarly, $p_+\in U$ if $\exists\,\varepsilon>0$ such that $(p,p+\varepsilon)\subset U$.

For any $n\ge1$, and $a,b\in[0,1]$, we consider the following notation:
$$f^n(a_+)=b_+\stackrel{\text{definition}}{\iff}f^n(x)\searrow b\text{ when }x\searrow a$$
$$f^n(a_+)=b_-\stackrel{\text{definition}}{\iff}f^n(x)\searrow b\text{ when }x\nearrow a$$
$$f^n(a_-)=b_+\stackrel{\text{definition}}{\iff}f^n(x)\nearrow b\text{ when }x\searrow a$$
$$f^n(a_-)=b_-\stackrel{\text{definition}}{\iff}f^n(x)\nearrow b\text{ when }x\nearrow a$$

A {\em left periodic-like attractor} is an attractor $A$ such that $A=\{p, f(p_-),\cdots,f^{n-1}(p_-)\}$, for some $p\in[0,1]$ and $n\ge1$ with $f^{n}(p_-)=p_-$, and such that $$\leb(\{x\in\beta_f(A)\,;\,p_-\in\omega_f(x)\})>0.$$
Analogously, we define a {\em right periodic-like attractor}. Notice that any attracting periodic-like orbit is a left periodic-like attractor or a right periodic-like attractor. A saddle-node is either a left or a right periodic-like attractor, but in general, an attracting periodic-like attractor can be both a left and a right periodic-like attractor.

\begin{definition}[Nice intervals, \cite{Ma}]
An interval $J=(a,b)\subset [0,1]$ is called {\em nice} with respect to $f$ if $\co_f^+(a)\cap J=\co_f^+(b)\cap J=\emptyset$. Notice that this means that $\co_f^+(a_\pm)\cap J=\co_f^+(b_\pm)\cap J=\emptyset$.
\end{definition}

We now briefly quote some basic results that we shall use in the sequel.  %The first one is a simple corollary of Mañe's theorem \cite{Man85}.
For that, let us denote by $\BB_0(f)$ the union of the basin of attraction of all attracting periodic-like orbits, that is, all left and all right periodic-like attractors.

%\begin{Lemma}\label{CorollaryMane}
%Let $f:[0,1]\setminus\cc_f\to[0,1]$ be a $C^3$ local diffeomorphism with $\cc_f\subset(0,1)$ being a finite set. If $f$ has negative Schwarzian derivative and there are no saddle-nodes then $\omega_f(x)\cap\cc_f\ne\emptyset$ for Lebesgue almost every $x\in[0,1]\setminus\BB_0(f)$.\end{Lemma}
%\begin{proof}
%As $Sf<0$, there are no weak repellers nor, by hypothesis, saddle-nodes and so, either a periodic orbit is a hyperbolic repeller or it is an attracting periodic orbit.
%Furthermore, the number of attracting periodic orbits is finite by the adapted Singer's Theorem.
%Let $\cu_n=\{x\in\co_f^-(\cc_f)\,;\,\co_f^+(x)\cap B_{1/n}(A_0\cup\cc_f)=\emptyset\}$, where $A_0$ is the union of the attracting periodic orbits of $f$ and $B_{1/n}(A_0\cup\cc_f)=\bigcup_{p\in A_0\cup\cc_f}B_{1/n}(p)$.
%Thus, if $x$ does not belong to the basin of attraction of an attracting periodic orbit and $\overline{\co_f^+(x)}\cap\cc_f=\emptyset$ then there is some $n\in\NN$ such that $x\in\cu_n$. As each $\cu_n$ is an uniformly expanding set and all uniformly expanding set of a $C^{1+}$ map of the interval has zero Lebesgue measure, we have that $\leb(\cu_n)=0$ $\forall\,n\ge1$. So, $\leb(\bigcup_{n\ge1}\cu_n)=0$, that is, $\omega_f(x)\cap\big(A_0\cup\cc_f)\ne\emptyset$.
%\end{proof}

%\begin{Definition}[Homterval]
A {\em homterval} is an open interval $I=(a,b)$ such that $f^n|_I$ is a homeomorphism for $n\ge1$. This is equivalent to assume that $I\cap\co_f^-(\cc_f)=\emptyset$. A homterval $I$ is called a {\em wandering interval} if $I\cap\BB_0(f)=\emptyset$ and $f^j(I)\cap f^k(I)=\emptyset$ for all $1\le j<k$.
%\end{Definition}

\begin{Lemma}[Homterval Lemma, see \cite{MvS}]\label{LemmaHomterval}
Let $I=(a,b)$ be a homterval of $f$. If $I$ is not a wandering interval, then $I\subset\BB_0(f)\cup\co_f^{-}(\per(f))$. Furthermore, if $f$ is $C^3$ with $Sf<0$, and $I$ is not a wandering interval, then the set $I\setminus\BB_0$ has at most one point.
\end{Lemma}
%\begin{proof}
%
%A {\em homterval} is an interval on which $f^j$ is monotone $\forall j \ge 0$, see, for example, \cite{G79}. Suppose $I$ is not a wandering interval. Then, there will be $k<\ell$ for which $f^k(I)\cap f^\ell(I)\ne\emptyset$. Define $f^k(I)=(a,b)$. If there is no integer $n$ such that  $c\in f^n(I)$, we can then consider the union $(a,b)\cup (f^{\ell-k}(a),f^{\ell-k}(b))$, and this is an interval, as their intersection is non-empty. 
%
%Then $T:=\bigcup_{j=0}^{\infty}f^{{(\ell-k)}j}(I)$ is a positively invariant homterval. Let $F$ is be the continuous extension of  $f^{\ell-k}|_{T}$ to $\overline{T}$. Let $R$ be the set all repeller fixed point of $F^2$, i.e., $R=\{p\in\overline{T}\,;\,F^2(p)=p$ and $\leb(\beta_{F^2}(p))=0\}$. As $F$ is a homeomorphism, $F(\overline{T})=\overline{T}$ and $\interior(\fix(F^2))=\emptyset$ (because $\interior(\fix(f^n))=\emptyset$), we can see that $\overline{T}\setminus R$ is open and dense in $\overline{T}$ and so, $I\setminus R$ is open and dense in $I$. Moreover, $I\setminus R\subset\overline{T}\setminus R\subset \widetilde{\BB}(F^2)\subset \overline{T}\setminus\BB_0(f)$, where $\widetilde{\BB}(F^2)$ is the set of all attracting fixed point of $F^2$.
%
%If $Sf<0$ then the condition $\fix(f^n)$ has empty interior $\forall\,n\ge1$ is automatic satisfied. Furthermore, as the extension $F$, given in the former paragraph, also has negative Schwarzian derivative, it follows that $\# R\le1$, proving the lemma. 
%
%
%\end{proof}

As we are dealing with subsets of the interval,  if the $\omega$-limit of a point is not totally disconnected, then its interior is not empty. This implies, by Lemma~\ref{lematres} below, that either $\omega_f(x)$ is a totally disconnected set or a {\em cycle of intervals}.
 
%\begin{Definition}[Cycle of Intervals]
%A cycle of intervals is a transitive finite union of non-trivial closed intervals. That is, $A$ is a cycle of intervals if $A=[a_1,b_1]\cup\cdots\cup[a_n,b_n]$, $0\le a_1<b_1<\cdots<a_n<b_n\le1$, such that $A=\overline{\co_f^+(a)}$ for some $a\in A$.
%\end{Definition}

The proofs of the following lemmas are standard.

\begin{Lemma}\label{lematres} If $p \in [0,1]\setminus\co_f^-(\cc_f)$,  then $\interior (\omega_f(p))\ne\emptyset$ if and only if $\omega_f(p)$ is a cycle of intervals.  Furthermore, each cycle of intervals contains at least one point of $\cc_f$ in its interior.
\end{Lemma}

\begin{Lemma}\label{LemmaFinitenessOfCycleOfIntervals}
Let $p,q\in[0,1]\setminus\co_f^-(\cc_f)$. If $\interior(\omega_f(p))\cap\interior(\omega_f(q))\ne\emptyset$, then $\omega_f(p)=\omega_f(q)$.
\end{Lemma}
%
%\begin{proof}
%Let $(\alpha,\beta)\subset \interior(\omega_f(p))\cap\interior(\omega_f(x))$. As $\omega_f(p)\cap(\alpha,\beta)\ne\emptyset$, $\exists n_p \ge 0$ such that $f^{n_0}(p)\in\omega_f(q)$. As $f$ is continuous in $\co_f^+(p)$, we have that $f^j(p) \in \omega_f(q) \forall j \ge n_p$ and then $\omega_f(p)\subset \omega_f(q)$. Similarly, we have that $\omega_f(p)\supset \omega_f(q)$.
%\end{proof}

A cycle of intervals may not be a minimal attractor (or even an attractor in Milnor's sense). Indeed, this is the case of the so called wild attractors \cite{BKNvS}. Nevertheless, as every cycle of intervals contains a critical point in its interior, it follows from Lemma~\ref{LemmaFinitenessOfCycleOfIntervals} above that the number of cycles of intervals for $f$ is always finite.

\begin{Corollary}\label{CorollaryFinitenessOfCycleOfIntervals}
$f$ has at most $\#\cc_f$ distinct cycles of intervals.
\end{Corollary} 

%%%%%%%%%%%%%%%%%%%%%%%%%%%%%%%%%%%%%%%%%%%%%%%
%\begin{figure}
%\begin{center}\includegraphics[scale=.24]{CicloDeIntervalos.png}\\
%\caption{$A=\overline{A_0}\cup\overline{A_1}$ is an example of a  cycle of intervals such that $\omega_f(x)=A$ for almost all $x\in[0,1]$.}\label{CicloDeIntervalos.jpg}
%\end{center}
%\end{figure}
%%%%%%%%%%%%%%%%%%%%%%%%%%%%%%%%%%%%%%%%%%%%%%%

%
%If $A$ is an attracting periodic-like orbit, then Singer's Theorem \cite{Si} (see also \cite{MvS}) assures that $A=\omega_f(c_-)$ or $\omega_f(c_+)$ for some $c\in\cc_f$.
%On the other hand, if $A$ is not a periodic-like attractor neither a cycle of intervals, then we can use Theorem~\ref{Theorem1aa} to write $A$ in terms of a subset of $(\cc_f)_{\pm}$.
%A cycle of intervals is the unique attractor that may not be traced by the critical orbits.
%For instance,  if $f:[0,1]\setminus\{1/2\}\to[0,1]$ is the complete logistic map, $f(x)=4x(1-x)$, then the whole interval $[0,1]$ is transitive and $\overline{\co_f^+((1/2)_\pm)}=\{1/2,1,0\}$.

\begin{Lemma}\label{Lemmaiuywe4jj}
Let $x\in[0,1]\setminus\co_f^-(\cc_f)$ be a point not contained in the basin of attraction of any attracting periodic-like orbit and let $p$ be a periodic-like point.
\begin{enumerate}
\item If $p$ is a left side periodic-like point, then $p\in\overline{\co_f^+(x)\cap[0,p)}$ $\iff$ $p\in\overline{\omega_f(x)\cap[0,p)}$.
\item If $p$ is a right side periodic-like point, then $p\in\overline{\co_f^+(x)\cap(p,1]}$ $\iff$ $p\in\overline{\omega_f(x)\cap(p,1]}$.
\end{enumerate}
\end{Lemma}
%\begin{proof}
%Let us prove the first item, the proof of the second being analogous. As ``$\Leftarrow$'' is immediate, we may assume that  $p$ is a left side periodic-like point and $p\in\overline{\co_f^+(x)\cap[0,p)}$. Let $n\ge1$ be such that  $f^n(p_-)=p_-$. Given $\varepsilon>0$, we need to show that $\omega_f(x)\cap(p-\varepsilon,p)\ne\emptyset$. As $x$ does not belong to an attracting periodic-like orbit, there exists $p-\varepsilon<a<p$ such that $f^n|_{[a,q)}$ is an increasing diffeomorphism with $f^n(q)<q$, $\forall\,q\in[a,p)$.
%Let $a'=(f^n|_{(a,q)})^{-1}(a)$. Note that $a<a'<p$. 
%Furthermore, for every $q\in [a,p)$ there is some $k_q\ge0$ such that $f^{k_q n}(q)\in[a,a']$. As a consequence, $\#\big(\co_f^+(x)\cap[a,a']\big)=\infty$. Thus, $\omega_f(x)\cap(p-\varepsilon,p)\supset\omega_f(x)\cap[a,a']\ne\emptyset$.\end{proof}

\begin{theorem}[Koebe's Lemma \cite{MvS}] \label{KoebeLemma}
For every $\varepsilon>0$ $\exists K>0$ such that the following holds: Let $M$, $T$ be intervals in $\RR$ with $M\subset T$ and denote respectively by $L$ and $R$ the left and right components of $T\setminus M$. If $f:T\to f(T)\subset\RR$ is a $C^3$ diffeomorphism with negative Schwarzian derivative and   
$$|f(L)|\ge\varepsilon|f(M)| \text{ and } |f(R)|\ge\varepsilon|f(M)|,$$
then $\frac{|Df(x)|}{|Df(y)|}\le K$ for $x,y\in M$.
\end{theorem}

\section{Markov maps and the interval dichotomy lemma}
\label{Sectionmarkovmaps}

Let $\XX$ be a compact metric space and $\mu$ be a finite measure defined on the Borel sets of $\XX$. Let  $F:V\to\XX$ be a measurable map defined on a Borel set $V\subset\XX$ with full measure (i.e., $\mu(V)=\mu(\XX)$), to which we shall refer. Note that we are not requiring $\mu$ to be $F$-invariant. The map $F$ is called {\em ergodic} with respect to $\mu$ (or $\mu$ is called ergodic with respect to $F$) if $\frac{\mu(U)}{\mu(\XX)}=0$ or $1$ for every $F$-invariant Borel set $U$, noting that here $F$-invariant means $U=F^{-1}(U)$.

\begin{Proposition}\label{ExpMeasures}
If a measure $\mu$ is ergodic with respect to $F$ (not necessarily invariant), then there is a compact set $A\subset\XX$ such that $\omega_F(x)=A$ for $\mu$-almost every $x\in\XX$.
\end{Proposition}
\begin{proof}

If $U\subset \XX$ is an open set, then either $\omega_F(x)\cap U\ne\emptyset$ for $\mu$-almost every $x \in\XX$ 
or $\omega_F(x)\cap U=\emptyset$ for $\mu$ almost every $x \in \XX$. 
Indeed, taking $\widetilde{U}=\{x\,;\,\omega_F(x)\cap U\ne \emptyset\}$ we have that $\widetilde{U}$ is invariant and then, by ergodicity, either $\mu(\widetilde{U})=0$ or $\mu(\widetilde{U})=\mu(\XX)$.

Notice that
if $\omega_F(x) \cap U \ne \emptyset$ for every open set $U$ and $\mu$-almost every $x$,
then $\omega_F(x)=\XX$ almost surely, proving the proposition.
Thus, we may suppose the existence of a non-empty  open $U\subset\XX$ such that $\omega_F(x)\cap U=\emptyset$ for $\mu$-almost every $x$.
Let $W$ be the maximal open set such that $\omega_F(x)\cap W=\emptyset$ for $\mu$ almost every $x\in\XX$.
Thus, $\omega_F(x)\subset A$ for $\mu$-almost every $x\in\XX$, where $A=\XX\setminus W$.

Now, we shall show that $\omega_F(x)\supset A$ for $\mu$-almost every $x \in\XX$.
For that, consider a countable and dense subset $A'$ of $A$. Given $p \in A'$ we have that necessarily $\omega_F(x)\cap B_{\varepsilon}(p) \ne \emptyset$ for $\mu$ almost every $x \in \XX$ and any $\varepsilon>0$,
for otherwise we would have $\varepsilon >0$ such that $\omega_F(x)\cap B_{\varepsilon}(p) = \emptyset$ for $\mu$-almost every $x\in\XX$,
but then $B_{\varepsilon}(p) \cup W$ would contradict the maximality of $W$.
Let $W_n=\{x\,;\,\omega_F(x)\cap B_{\frac{1}{n}}(p) \ne \emptyset\}$.
We have that $\mu(W_n)=\mu(\XX)$, $\forall\,n\in\NN$, and thus, $\mu(\bigcap_n W_n)=\mu(\XX)$. As $\omega_F(x)$ is a closed set and $dist(p,\omega_F(x))=0$ $\forall\,x \in \bigcap_n W_n$, it follows that $p \in \omega_F(x)$ for $\mu$-almost every $x\in \XX$ and every $p\in A'$. As $A'$ is countable $A'\subset\omega_F(x)$ for $\mu$-almost every $x\in\XX$. As $A'$ is dense in $A$, we get also that $A\subset\omega_F(x)$ for $\mu$-almost every $x\in\XX$, proving Proposition\,\ref{ExpMeasures}. 

\end{proof}

\begin{Lemma}\label{pagina1}Let $a < b \in \mathbb{R}$ and $V \subset (a,b)$ be an open set. Let $\mathcal{P}$ be the set of connected components of a Borel set $V$. Let $G:V\to(a,b)$ be a map satisfying:
\begin{enumerate}
\item 
$G(P)=(a,b)$ diffeomorphically, for any $P \in \mathcal{P}$;
\item $\exists V'\subset \bigcap_{j\ge 0}G^{-j}(V)$, with $\leb(V')>0$, such that
\begin{enumerate}
\item $\lim_{n\to\infty}|\mathcal{P}_n(x)|=0$, $\forall x \in V'$,\\
where $\mathcal{P}_n(x)$ is the connected component of $\bigcap^n_{j=0}G^{-j}(V)$ that contains $x$;
\item\label{BD99}$\exists K>0$ such that 
$ \big|\frac{DG^n(p)}{DG^n(q)}\big| \le K,$
 for all $n$, and $p, q \in \mathcal{P}_n(x)$, and $x \in V'$
\end{enumerate}
\end{enumerate}
Then, 
$\leb([a,b]\setminus V)=0$,  $\omega_G(x)=[a,b]$ for Lebesgue almost all $x \in [a,b]$.
\end{Lemma}
\begin{proof}
Firstly, we show that
\begin{Claim}\label{Claim75849302}
Every positively invariant set $U\subset V'$ with positive measure has measure $|b-a|$.
\end{Claim}
\begin{proof}[Proof]
Suppose that $U \subset V'$ is positively invariant with positive measure. Note that necessarily $U \subset \bigcap_{j\ge 0}G^{-j}(V) $.
From the Lebesgue Density Theorem, there is $p \in U$ (indeed for Lebesgue almost every $p \in U$) such that
$\lim_{n \to \infty} \frac{\leb(\mathcal{P}_n(p)\setminus U)}{\leb(\mathcal{P}_n(p))} =0$. So, it follows from the bounded distortion hypothesis, item (\ref{BD99}) in Lemma \ref{pagina1} and the forward invariance of $U$ that $\leb((a,b)\setminus U)=0$.
Indeed, 
\begin{equation}
\label{desig}
\leb((a,b)\setminus U)\le \leb(G^n(\mathcal{P}_n(p)\setminus U))
\end{equation}
$$
=\leb((a,b))\frac{\leb(G^n(\mathcal{P}_n(p)\setminus U))}{\leb(G^n(\mathcal{P}_n(p)))}
\le \leb((a,b)) K \frac{\leb(\mathcal{P}_n(p)\setminus U)}{\leb(\mathcal{P}_n(p))} \to 0
$$
Here the inequality (\ref{desig}) follows from the fact that $U \supset G^n(\mathcal{P}_n(p)\cap U)$ and, then, $(a,b) \setminus U \subset (a,b) \setminus G^n(\mathcal{P}_n(p)\cap U))=G^n(\mathcal{P}_n(p)\setminus U)$. And the last equality we get by writing $(a,b)=G^n(\mathcal{P}_n(p))$, and as $G^n(\mathcal{P}_n(p))$ is a bijection, $(a,b)$ can be written as a disjoint union of $ G^n(\mathcal{P}_n(p)\setminus U)$ and $ G^n(\mathcal{P}_n(p) \cap U)$.
\end{proof}
As $V'$ is a positively invariant set, we get $\leb((a,b)\setminus V') =0$. Thus, $\leb(V)=\leb(V')=\leb([a,b])$. As an invariant set is also a positively invariant set, it also follows  from the claim that $G$ is ergodic with respect to Lebesgue measure (more precisely, with respect to $\leb|_{[a,b]}$).

From Proposition~\ref{ExpMeasures}, 
there is a compact set $A \subset [a,b]$ such that $\omega_G(x)=A$ for Lebesgue almost every $x\in [a,b]$.

\begin{Claim}
\label{claimzwei}
$\leb(A)>0$.
\end{Claim}

\begin{proof}[Proof]
Suppose that $\leb(A)=0$. Then, given $\varepsilon >0$, there exists an open neighborhood $M_\varepsilon$ of $A$ such that $\leb(M_\varepsilon)<\varepsilon$. Let $\Omega_\varepsilon(n)=\{x \in (a,b);\co_G^+(G^n(x))\subset M_\varepsilon\}$.
Observe that $\leb(\bigcup_n \Omega_\varepsilon(n))=|b-a|$, as $\omega_G(x)=A$ for Lebesgue almost all $x$. Then, $\exists n_0$ such that $\leb( \Omega_\varepsilon(n_0))>0$. As $\Omega_\varepsilon(n_0)$ is positively invariant, it also follows that $G^{n_0}(\Omega_\varepsilon(n_0))$ is positively invariant.
Then, by Claim \ref{Claim75849302}, $\leb(G^{n_0}( \Omega_\varepsilon(n_0)))=|b-a|$. This is a contradiction since $G^{n_0}(\Omega_\varepsilon(n_0))\subset M_\varepsilon$ and $\leb(M_\varepsilon)<\varepsilon$.
\end{proof}

As $A$ is positively invariant and $\leb(A)>0$, by Claim \ref{claimzwei} it follows from Claim \ref{Claim75849302} that $\leb(A)=|b-a|$. As $A$ is a compact set, it follows that $A=[a,b]$. Thus $\omega_G(x)=A=[a,b]$ for Lebesgue almost every $x$, proving the Lemma.
\end{proof}

\begin{Lemma}\label{Lemma375945194558}
Let $U\subset I=(a,b)\subset\RR$ be an open set and $F:U\to (a,b)$ be a $C^3$ local diffeomorphism with $SF<0$. Let $\cp$ be the collection of connected components of $U$. If there is a positively invariant set $V\subset U$ with positive measure such that
\begin{enumerate}
\item $F(\cp(x))=I$ $\forall\,x\in V$ and
\item $V$ does not intersect the basin of attraction of any periodic-like attractor of $F$,
\end{enumerate}
then either $\omega_F(x)\subset\partial I$ for almost every $x\in V$ or $\leb(I\setminus V)=0$ and $\omega_F(x)=[a,b]$ for Lebesgue almost every $x\in[a,b]$.
\end{Lemma}
\begin{proof}
We may suppose that $\leb(\{x\in V\,;\,\omega_f(x)\cap I\ne\emptyset\})>0$ and choose $a<a'<b'<b$ such that $\leb(V')>0$, where $V'=\{x\in V\,;$ $\omega_F(x)\cap (a',b')\ne\emptyset\}$.
Write $J=(a',b')$ (we will consider $J\subset I$ instead of $I$, so that we can apply Koebe's Lemma, Theorem~\ref{KoebeLemma}).

It follows from Koebe's Lemma that there is $K$, depending only on $\frac{a'-a}{b-a}$ and $\frac{b-b'}{b-a}$, such that 
\begin{equation}\label{claim0987}
\bigg| \frac{(F^n)'(p)}{(F^n)'(q)}\bigg|\le K, \forall p, q\in J_n(x),\forall n,\forall x \in F^{-n}(J),
\end{equation}
where $J_n(x)$ is the connected component of $F^{-n}(J)$ that contains $x$.

Given $x\in I$ let $$\cu(x)=\{n\in\NN\,;\,x\in F^{-n}(J)\}.$$
\begin{claim}%\label{Claim911875}
We have $\lim_{\cu(x)\ni n\to\infty}|J_n(x)|=0$, for Lebesgue almost every $x\in V'$.  
\end{claim}
\begin{proof}[Proof of the Claim]
Suppose that there is a Lebesgue density point $x$ of $V'$ such that  $$\lim_{\cu(x)\ni n\to \infty} |J_n(x)|>0.$$ Then, $M:=\interior(\bigcap_{n\in\cu(x)}J_n(x))$ is an open interval with  $J_1(x)\supset\cdots\supset J_n(x)\searrow M.$
As $F^n(J_n(x))=J$ $\forall n$, it follows from the bounded distortion (\ref{claim0987}) above that $F^n(M)\stackrel{n\to\infty}{\longrightarrow}J$.
Thus, there exists $\ell$ big enough so that 
$F^\ell(M)\cap M\ne\emptyset$.
As $F^{n\ell}|_M$ is a diffeomorphism $\forall n$, defining $\widetilde{M} = \bigcup_{n\ge 0} F^{n\ell}(M)$, we have that $F^\ell|_{\widetilde{M}}$ is a diffeomorphism and $F^\ell(\widetilde{M})=\widetilde{M}$. This implies, since $SF<0$, that $F^\ell|_{\widetilde{M}}$ has one, say $p_1$, or at most two attracting fixed-like points $p_1,p_2$ such that $\leb(\widetilde{M} \setminus \beta(A))=0$, where $A=\{p_1\}$ or $\{p_1,p_2\}$. Therefore, Lebesgue almost every point of the neighborhood of $x$ is contained in the basin of attraction of some attracting periodic-like orbit, contradicting the fact that $V'$ does not intersect the basin of any periodic-like attractor (recall that $x$ is a density point of $V'$).
\end{proof}

It follows from the claim above that $\overline{\per(F)}\supset V'\cap(a',b')$. Furthermore, from the claim, the invariance of $V'$ and the bounded distortion (\ref{claim0987}), we  also get that $\leb((a',b')\setminus V')=0$. As a consequence, $\overline{\per(F)}\supset(a',b')$.

Let $\alpha_0,\beta_0\in(a',b')\cap\per(f)$ be such that $\alpha_0<\beta_0$, $\leb(\{x\in V'\,;\,\omega_F(x)\cap\big(\co_f^+(\alpha_0)\cup\co_f^+(\beta_0)\big)\ne\emptyset\})=0$ and $\leb(\{x\in V'\;,\,\omega_F(x)\cap(\alpha_0,\beta_0)\ne\emptyset\})>0$.

%Given $m>0$, chose $\alpha_m,\beta_m\in\per(F)\cap[a'',b'']$, with $a''\le\alpha_m<a''+1/m<b''-1/m<\beta_m\le b''$. Let $\cl_m$ be the set of all connected component of $(\alpha_m,\beta_m)\setminus(\co_F^+(\alpha_m)\cup\co_F^+(\beta_m))$. Note that $V'\subset\bigcup_{m\ge1}\bigcup_{L\in\cl_m}\overline{L}$. such that $\leb(\{x\in V\,;\,\omega_F(x)\cap(\alpha,\beta)\ne\emptyset\ne0\}$.
%
%
%
%  Let $L$ be the set of $p\in(a',b,)$ such that there is $n\ge1$ such that  $\leb(V'\cap(a',p)\setminus J_n(p))$ and $\leb(V'\cap(p,b')\setminus J_n(p))>0$. Define $V'(p,n)=\{x\in V'\,;\,p\in\omega_F(x)\cap J_n(p)\ne\emptyset\}$. Observe that $\leb(V'\setminus\bigcup_{(p,n)}V'(p,n))=0$.
%
%Given $(p,n)$ as above, let $\alpha_0\in\per(F)\cap(a',p)$ and $\beta_0\in\per(F)\cap(p,b')$. Let $\alpha=\max\big((\co_F^+(\alpha_0)\cup\co_F^+(\beta_0))\cap(a',p))\big)$ and $\beta=\min\big((\co_F^+(\alpha_0)\cup\co_F^+(\beta_0))\cap(p,b'))\big)$.

Let $T=(\alpha,\beta)$ be a connected component of $(\alpha_0,\beta_0)\setminus\big(\co_F^+(\alpha_0)\cup\co_F^+(\beta_0)\big)$ such that $\leb(\{x\in V'\,;\,\omega_F(x)\cap T\ne\emptyset\})>0$. Let $T^*=\{x\in T\,;\,\co_F^+(F(x))\cap T\ne\emptyset\}$, $\mathcal{T}$ be the set of all connected components of $T^*$ and $F_T:T^*\to T$ be the first return map to $T$. As $\co_F^+(\partial T)\cap T=\emptyset$, we have that $F_T(W)=T$ for all $W\in\mathcal{T}$. It follows again from (\ref{claim0987}) that
$$
\bigg| \frac{(F_T^n)'(p)}{(F_T^n)'(q)}\bigg|\le K, \forall p, q\in \mathcal{T}_n(x),\forall n,\forall x \in V'',
$$
where $V''=\{x\in T\cap V'\,;\,\omega_F(x)\cap T\ne\emptyset\}$ and $\mathcal{T}_n(x)$ is the connected component of $F_T^{-n}(T)$ containing $x$. That is, $F_T$ satisfies the hypothesis of Lemma~\ref{pagina1}. Thus, $\leb(T\setminus V')=\leb(T\setminus V'')=0$ and $\omega_{F_T}(x)\supset[\alpha,\beta]$ for Lebesgue almost every $x \in T$. Chose any $p\in V'\cap T$ and $n\ge1$ big enough so that $J_n(p)\subset T$. As $F^n(J_n(p))=(a',b')$ and $\leb\circ F^{-1}\ll\leb$, we have that $\leb((a',b')\setminus V)=\leb((a',b')\setminus V')=0$ and $\omega_F(x)\supset[a',b']$ for almost all $x\in (a',b')$. 
Finally, as we can take $a'$ as close to $a$ and $b'$ as close to $b$ as wished, we conclude that $\leb(I\setminus V)=0$ and $\omega_F(x)=[a,b]$ for Lebesgue almost every $x \in [a,b]$.

\end{proof}

In the remainder of this section, $\cc_f\subset(0,1)$ is a finite set and $f:[0,1]\setminus\cc_f\to[0,1]$ is a $C^3$ local diffeomorphism, with $Sf<0$ and $f(\{0,1\})\subset\{0,1\}$.

Denote by $\cv_f$ the set of ``lateral exceptional values'' of $f$, i.e., $\cv_f=\{f(c_\pm)\,;\,c\in\cc_f\}$, and let $\co_f^+(\cv_f)=\{f^n(c_\pm)\,;\,c\in\cc_f$ and $n\ge1\}$.

\begin{Lemma}[Interval Dichotomy]\label{tudoounadadershchw} 
Let $I=(a,b)\subset [0,1]$ be a nice interval such that $I\cap\co_f^+(\cv_f) = \emptyset$.
If $\leb(I\setminus \BB_0(f))>0$, then
either  $\omega_f(x)\cap I=\emptyset$ for almost every $x\in I\setminus\BB_0(f)$
or  $\omega_f(x)\supset I$ for almost every $x\in I$.
\end{Lemma}
\begin{proof}

Let $F:I^*\to I$ be the first return map to $I$, with  $I^*=\{x\in I\,;\,\co_f^+(f(x))\cap I\ne\emptyset\}$. Let $\cp$ be the set of connected components of $I^*$.

\begin{claim}
%\label{claim09yrqyrq}
$F$ is a local diffeomorphism having negative Schwarzian derivative and $F(P)= I$ for $\forall P\in\cp$.
\end{claim}
\begin{proof}[Proof of the Claim] As $f$ is a local diffeomorphism with $Sf<0$, it follows that $F$ is also a local diffeomorphism with $SF<0$.
Given $P\in\cp$, 
there is some $m>0$ such that  $F|_P=f^m|_P$. Write $P=(p_0,p_1)$. As $I$ is a nice interval, if $F(P)\ne I$, then $\exists i\in\{0,1\}$, $0\le n<m$ and $c\in\cc_f$ such that $f^n(p_i)=c$ and $f^{m-n}(c_-)$ or $f^{m-n}(c_+)\in I$, contradicting our hypothesis.
\end{proof}

Now let $V=\{x\in I\setminus \BB_0; \#(\co_f^+(x)\cap I)=\infty\}$, i.e., $V=\big(\bigcap_{n\ge0}F^{-n}(I)\big)\setminus\BB_0$, and assume that $\leb(V)>0$. 
Note that $V$ is an $F$-positively invariant set with positive measure and it does not intersect the basin of attraction of any periodic-like attractor of $F$. Thus, the first return map $F$ satisfy all the hypotheses of Lemma~\ref{Lemma375945194558}. As a consequence, $\leb(I\setminus V)=0$ and $\omega_f(x)\supset\omega_F(x)=\overline{I}\supset I$ for almost every $x\in I$.

\end{proof}

Let $\BB_1(f)=\{x\in[0,1]\,;\,\interior(\omega_f(x))\ne\emptyset\}$. By Lemma~\ref{lematres}, $\BB_1(f)$ is the set of points $x\in[0,1]$ such that $\omega_f(x)$ is a cycle of intervals. In particular, $\BB_1(f)$ is contained in the union of the basins of attraction of all cycles of intervals.

\begin{Corollary}\label{Cor09876111}$\omega_f(x)\subset\overline{\co_f^+(\cv_f)}$ for almost all $x\in[0,1]\setminus(\BB_0(f)\cup\BB_1(f))$.
\end{Corollary}
\begin{proof}
As the collection $\cp$ of all connected components of $(0,1)\setminus\overline{\co_f^+(\cv_f)}$ is a countable set of intervals. Furthermore, each $I\in\cp$ is a nice interval. Thus, it follows from the interval dichotomy lemma that $$0\le\leb\bigg(\bigg\{x\in[0,1]\setminus(\BB_0(f)\cup\BB_1(f))\,;\,\omega_f(x)\not\subset\overline{\co_f^+(\cv_f)}\bigg)\le$$
$$\le\sum_{I\in\cp}\leb\big(\big\{x\in I\setminus(\BB_0(f)\cup\BB_1(f))\,;\,\omega_f(x)\cap I\ne\emptyset\big\}\big)=0.$$
\end{proof}

A $C^3$ map $f:[0,1]\to[0,1]$, with $f(0)=f(1)=0$, is called {\em $S$-unimodal} if it has at most two fixed points, a single critical point $c\in(0,1)$ and negative Schwarzian derivative.
Blokh and Lyubich \cite{BL91} proved that non-flat $S$-unimodal maps display a unique metrical attractor. Using induced maps on Hofbauer-Keller towers, Keller \cite{Ke} obtained the same fact without the hypothesis of non-flatness of the map. These results by Blokh-Lyubich and Keller can also be obtained from Lemma~\ref{tudoounadadershchw} and Corollary~\ref{Cor09876111}, as shown in the Appendix.

\begin{Theorem}\label{sunimodal}
If $f:[0,1]\to[0,1]$ is a $S$-unimodal map with critical point $c\in(0,1)$, then there is an attractor $A\subset[0,1]$ such that $\omega_f(x)=A$ for almost every $x\in[0,1]$. In particular, $\leb(\beta_f(A))=1$.
The attractor $A$ is either a periodic orbit or a cycle of intervals or a Cantor set and, in this last case, $A=\overline{\co_f^+(c)}$.
\end{Theorem}

\color{black}

\section{Proof of the finiteness of the number of attractors}\label{SecMainProof}

%$$\beta_f(L,R)=\beta\big(\overline{\co_f^+((L)_-)}\cup\overline{\co_f^+((K)_+)}\big)=\big\{x\,;\,\omega_f(x)\subset \overline{\co_f^+((L)_-)}\cup\overline{\co_f^+((K)_+)}\big\} $$

\begin{Lemma}\label{LemmaCritnoPoco}
Let $g:[0,1]\setminus\cc_g\to[0,1]$ be a $C^3$ local diffeomorphism with $Sg<0$ and $\cc_g\subset(0,1)$ being a finite set. If $\exists\,I=(c_0,c_1)$ and $p\in \overline{I}$ such that $c_0,c_1\in\cc_g$, $I\cap\cc_g=\emptyset$, and $g|_{I}$ is a contraction with attracting fixed-like point $p$, then 
%$$\omega_f(x)\subset\bigcup_{c\in\cc_g\setminus\cc_g^{^\pm}(I)}\overline{\co_g^+(c_\pm)}$$
$$\omega_g(x)\subset\bigcup_{c\in\cc_g\setminus\cc_g^{^\pm}(I)}\overline{\co_g^+(c_\pm)}:=\bigcup_{c\in\cc_g\setminus\cc_g^{^-}(I)}\overline{\co_g^+(c_-)}\;\;\cup \bigcup_{c\in\cc_g\setminus\cc_g^{^+}(I)}\overline{\co_g^+(c_+)},$$ 

for almost all $x\in[0,1]\setminus(\BB_0(g)\cup\BB_1(g))$, where $\cc_g^\pm(I) = \{ c \in \cc_g ; \co_g^+(c_\pm)\cap I \ne \emptyset\} $.

\end{Lemma}
%%%%%%%%%%%%%%%%%%%%%%%%%%%%%%%%%%%%%%%%%%%%%%
\begin{figure}
\begin{center}\includegraphics[scale=.25]{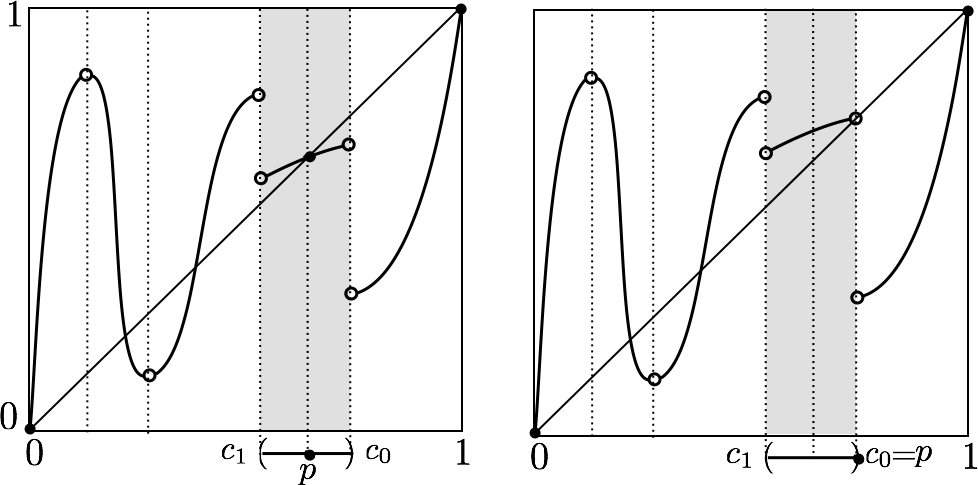}\\
\caption{\small Two examples of  maps as in Lemma~\ref{LemmaCritnoPoco}. For the map on the left $p$ is a fixed point and for the one on the right $p$ is a fixed-like point}\label{MapPoco.pdf}
\end{center}
\end{figure}
%%%%%%%%%%%%%%%%%%%%%%%%%%%%%%%%%%%%%%%%%%%%%%
\begin{proof}
Given $y\in[0,1]$, let $V(y)=\{x\in[0,1]\setminus(\BB_0(g)\cup\BB_1(g))\,;\,y\in\omega_f(x)\}$.

\begin{claim}
If $a\in\cc_g^{\pm}(I)$, then $\leb(V(f^j(a_\pm)))=0$, $\forall\,j\ge0$ s.t. $g^j(a_\pm)\notin\bigcup_{c\in\cc_g\setminus\cc_g^{^\pm}(I)}\overline{\co_g^+(c_\pm)}$.\end{claim}
\begin{proof}[Proof of the Claim]
We may assume that $a\in\cc_g^{^-}(I)$, the case $a\in\cc_g^+(I)$ being  analogous. Let $n=\min\{j\ge0\,;\,g^j(a_-)\in I\}$ and let  $\varepsilon>0$ be so that $g^n|_{(a-\varepsilon,a)}$ is a diffeomorphism and $g^n((a-\varepsilon,a))\subset I$.

Suppose that $\exists0\le\ell<n$ such that $\leb(V(g^\ell(a_-)))>0$ and $g^\ell(a_-)$ $\notin$ $\bigcup_{c\in\cc_g\setminus\cc_g^{^\pm}(I)}\overline{\co_g^+(c_\pm)}$.

%Let  $b:=f^{\ell}(a_-)$ and $r\in(0,\varepsilon)$ be such that $(b-r,b+r)\cap \bigcup_{c\in\cc_g\setminus\cc_g^{^\pm}(I)}\overline{\co_g^+(c_\pm)}=\emptyset$.
As $\omega_g(c_{\pm})=p$, $\forall\,c\in\cc_g^\pm(I)$, $b:=f^{\ell}(a_-)$ is an isolated point of $\overline{\co_g^+(\cv_g)}$.
Let $(q_0,b)$ and $(b,q_1)$ be the connected components of $[0,1]\setminus\overline{\co_g^+(\cv_g)}$ containing $b$ in its boundary.
We may assume that $g^{\ell}((a-\varepsilon,a))=(g^{\ell}(a-\varepsilon),b)$, i.e., $g^{\ell}|_{(a-\varepsilon,a)}$ preserves orientation (the case of  orientation reversing is analogous). In this case, $V(b)\cap(g^\ell(a-\varepsilon),b)=\emptyset$ and so, the forward orbit of a point of $V(b)$ accumulates on $b$ only by the right side.

For each $x\in V(b)\cap(b,q_1)$, let $n_x=\min\{i\ge1\,;\,g^i(x)\in(b,q_1)\}$ and let $I_x$ be the maximal open interval satisfying the following  three conditions: (1) $x\in I_x\subset(b,q_1)$, (2) $g^{n_x}|_{I_x}$ is a diffeomorphism and (3) $g^{n_x}(I_x)\subset(q_0,q_1)$.

Notice that  $g^{n_x}(I_x)$ is equal either to $(q_0,q_1)$ or to $(b,q_1)$. Suppose for instance that $g^{n_x}(I_x)=(b,q_1)$ and write $I_x=(\alpha,\beta)$.
So, $\exists\,0\le i<n_x$ and $c\in\cc_f$ such that $g^i(\alpha)=c$ and $g^{n_x}(\alpha_+)=b$ or $g^i(\beta)=c$ and $g^{n_x}(\beta_-)=b$.
As $b\notin\bigcup_{c\in\cc_g\setminus\cc_g^{^\pm}(I)}\overline{\co_g^+(c_\pm)}$, we get that
\begin{enumerate}
\item	either $g^i((\alpha,\beta))=(c,g^i(\beta))$, $c\in\cc_g^+(I)$ and $g^{n_x}((\alpha,\beta))=(g^{n_x-i}(c_+),q_1)$  $=$ $(b,q_1)$
\item  or $g^i((\alpha,\beta))=(g^i(\beta),c)$, $c\in\cc_g^{^-}(I)$ and $g^{n_x}((\alpha,\beta))=(g^{n_x-i}(c_+),q_1)$ $=$ $(b,q_1)$.
\end{enumerate}
 In any case there exists $\delta>0$ such that $(b,b+\delta)$ belongs to $\beta_g(p)$, the basin of attraction of the fixed point $p$. This is a contradiction, as it implies that $V(b)\subset\beta_g(p).$ Therefore, we have that $g^{n_x}(I_x)=(q_0,q_1)$ for every $x\in V(b)\cap(b,q_1)$.

Let $(q_0,q_1)^*=\bigcup_{x\in V(b)}I_x$.
As $I_x\cap I_y\implies I_x=I_y$, $\forall\,x,y\in V(b)$, define $R:(q_0,q_1)^*\to\NN$ by $R|_{I_x}\equiv n_x$ $\forall\,x\in V(b)$.
Let $G:(q_0,q_1)^*\to(q_0,q_1)$ be the induced map given by $F(x)=g^{R(x)}(x)$.
As $b\in\omega_F(x)$ for every $x\in V(b)$, it follows from Lemma~\ref{Lemma375945194558} that $\omega_g(x)\supset\omega_F(x)=[q_0,q_1]$ for almost every $x\in[q_0,q_1]$ which is a contradiction, proving the claim. 
\end{proof}

It follows from the claim that $$\bigg\{q\in \bigcup_{c\in\cc_g^\pm(I)}\overline{\co_g^+(c_\pm)}\,;\,Leb(V(q))>0\bigg\}\subset\bigcup_{c\in\cc_g\setminus\cc_g^{^\pm}(I)}\overline{\co_g^+(c_\pm)}.$$
Thus, using Corollary~\ref{Cor09876111}, we get that 
$$\omega_g(x)\subset\overline{\co_f^+(\cv_f)}\,\,\cap \bigcup_{c\in\cc_g\setminus\cc_g^{^\pm}(I)}\overline{\co_g^+(c_\pm)}\subset\bigcup_{c\in\cc_g\setminus\cc_g^{^\pm}(I)}\overline{\co_g^+(c_\pm)},$$
for almost every $x\in\BB_0(g)\cup\BB_1(g)$.
\end{proof}

We now begin to turn our attention to the parallax argument, as we have
%stressed its importance in Section 2.1 for us to obtain main results in this paper.
mentioned in the introduction. 

Given a $C^3$ local diffeomorphism $f:[0,1]\setminus\cc_f\to[0,1]$, with $Sf<0$, let us write
$$(\cc_f)^-\cap\BB_0(f):=\{c\in\cc_f\,;\,(c-\varepsilon,c)\subset\BB_0(f)\text{ for some }\varepsilon>0\}$$
and
$$(\cc_f)^+\cap\BB_0(f):=\{c\in\cc_f\,;\,(c,c+\varepsilon)\subset\BB_0(f)\text{ for some }\varepsilon>0\}.$$
Also write $$(\cc_f)^\pm\setminus\BB_0(f):=\cc_f\setminus((\cc_f)^\pm\cap\BB_0(f)).$$
\begin{Corollary}\label{CorollaryCritnoPoco}
	If $f:[0,1]\setminus\cc_f\to[0,1]$ is a $C^3$ local diffeomorphism with $Sf<0$ and $\cc_f\subset(0,1)$ being a finite set, then  
$$\omega_f(x)\subset\bigcup_{c\in(\cc_f)^\pm\setminus\BB_0(f)}\overline{\co_f^+(c_\pm)}:=\bigcup_{c\in(\cc_f)^-\setminus\BB_0(f)}\overline{\co_f^+(c_-)}\;\;\cup \bigcup_{c\in(\cc_f)^+\setminus\BB_0(f)}\overline{\co_f^+(c_+)},$$
for almost every $x\in[0,1]\setminus(\BB_0(f)\cup\BB_1(f))$.
\end{Corollary}
\begin{proof}
For each left periodic-like attractor $A$, write 
$A$ $=$ $\{p,f(p_-),$ $\cdots,$ $f^{n_A-1}(p_-)\}$ $=$ $\co_f^+(p_-)$g , where the period of $p$, $n_A\ge1$, is the smaller integer bigger than one such that $f^{n_A}(p_-)=p_-$. 
Let $I_A^-:=(p_0,p)$ be such that $f^n|_{(p_0,p)}$ is monotonous and $p_0<f^{n_A}(p_0)<p$. 
Similarly, define $I_A^+$ for each right periodic-like attractor $A$. If $A$ is a left periodic-like attractor but not a right one, then set $I_A=I_A^-$. Also, set $I_A=I_A^+$ whenever $A$ is a right periodic-like attractor but not a left one. Finally, if $A$ is both a left and right periodic-like attractor, set
$$I_A=
\begin{cases}
I_A^-\cup\{p\}\cup I_A^+ & \text{ if }A\cap\cc_f=\emptyset\\
I_A^-\cup I_A^+ & \text{ if }A\cap\cc_f\ne\emptyset\\
\end{cases}.$$

Let $\cc_{f,A}^\pm=\{c\in\cc_f\,;\,\co_f^+(c_\pm)\cap\beta_f(A)\ne\emptyset\}$ and observe that $c\in\cc_{f,A}^\pm$ $\iff$ $\exists\,\ell\ge1$ such that $f^{\ell}(c_\pm)\in V_A$, where $V_A:=I_A\cup\cdots\cup f^{n_A-1}(I_A)$.

%Analogously define $V_A^+$ for a right periodic-like attractor.
%If $A$ is both a left and a right periodic-like attractor, write
%$$V_A=
%\begin{cases}
%V_A^-\cup\{p\}\cup V_A^+ & \text{ if }A\cap\cc_f=\emptyset\\
%V_A^-\cup V_A^+ & \text{ if }A\cap\cc_f\ne\emptyset\\
%\end{cases}.$$

Let $\{A_1,\cdots,A_s\}$ be the set of all periodic-like attractors $A$ such that $\cc_{f,A}^\pm\ne\emptyset$. Let $V=V_{A_1}\cup\cdots\cup V_{A_s}$ and  $g:[0,1]\setminus\cc_g\to[0,1]$ be given by
	$$g(x)=\begin{cases}
		f(x) & \text{ if }x\notin V\\
		f^{n_{A_j}}(x) & \text{ if }x\in V_{A_j}
	\end{cases},$$
	where $\cc_g=\cc_f\cup \partial V$.
	Noting that $\BB_0(f)=\BB_0(g)$, $\BB_1(f)=\BB_1(g)$, $\omega_f(x)=\omega_g(x)$ $\forall\,x\notin\BB_0(f)=\BB_0(g)$ and that $$\cc_g^{\pm}(V):=\bigcup_{j=0}^s\bigcup_{i=0}^{n_{A_j}-1}\cc_g^{\pm}(f^i(I_{A_j}))=\bigcup_{j=0}^s\bigcup_{i=0}^{n_{A_j}-1}\cc_f^{\pm}(f^i(I_{A_j}))=:\cc_f^{\pm}(V),$$ it follows from Lemma~\ref{LemmaCritnoPoco} that
	$$\omega_f(x)=\omega_g(x)\subset\bigcap_{j=0}^s\bigcap_{i=0}^{n_{A_j}-1}\bigg(\bigcup_{\gamma\in\cc_g\setminus\cc_g^{^\pm}(f^i(I_{A_j}))}\overline{\co_g^+(\gamma_\pm)}\bigg) \stackrel{*}{=} \bigcup_{\gamma\in\cc_g\setminus\cc_g^{^\pm}(V)}\overline{\co_g^+(\gamma_\pm)}=$$
	$$=\bigcup_{\gamma\in\cc_f\setminus\cc_f^{^\pm}(V)}\overline{\co_f^+(\gamma_\pm)}\subset\bigcup_{\gamma\in(\cc_f)^\pm\setminus\BB_0(f)}\overline{\co_f^+(\gamma_-)}.$$
for almost every $x\in[0,1]\setminus(\BB_0(f)\cup\BB_1(f))$, where ($*$) is true because $\cc_g^\pm(f^i(I_{A_j}))\cap\cc_g^\pm(f^k(I_{A_\ell}))=\emptyset$ whenever $(i,j)\ne(k,\ell)$.
	
\end{proof}

By definition, if $J$ is a wandering interval, then $J\cap\cc_f=\emptyset$.
Nevertheless, the border of $J$, $\partial J$, may contain some $c\in\cc_f$.
%In this case, we have either $c_-\in J$ or $c_+\in J$.
Let $\cw_f^-$ be the set of points $c\in\cc_f$ such that $(a,c)$ is a wandering interval for some $a<c$. Similarly, $\cw_f^+$ is the set of points $c\in\cc_f$ such that $(c,b)$ is a wandering interval for some $b>c$.
If $c\in\cw_f^-$, define $p_{c_-}$ as the infimum of all $0<t<c$ such that $(t,c)$ is a wandering interval and define $J_{c_-}=(p_{c_-},c)$. Analogously, we define $p_{c_+}$ and $J_{c_+}$ when $c\in\cw_f^+$.
The interval $J_{c_\pm}$ is called the {\em exceptional wandering interval} associated to $c_{\pm}$.

%We say that $A$ is an {\em exceptional wandering interval attractor} (or {\em Ewi attractor}) if $A=\omega_f(c_\pm)$ for some $c\in\cw_f^\pm$. The basin of attraction of $A$ contains the exceptional wandering interval $J_{c_\pm}$. Indeed, an Ewi attractor is the smallest attractor containing an exceptional wandering interval.

\begin{Example}\label{ExemplodeEwi}
Let $f:[0,1]\setminus\{c\}\to[0,1]$, $0<c<1$, be a $C^3$ contracting Lorenz map with $Sf<0$ such that $\leb(f([v_0,v_1]\setminus\{c\}))<|v_0-v_1|$, where $v_0=f(c_+)$ and $v_1=f(c_-)$. In this case, the map $G:=f|_{J\setminus\{c\}}$ is injective but not surjective, where $J=[v_0,v_1]$. Such a map $G$ is called a gap map and it is known that it has a well defined rotation number. If its rotation number is irrational, then $I_0:=(G(v_1),G(v_0))=(f(v_1),f(v_0))$ is a wandering interval. In this case, $I=(a,v_0)$ is a wandering interval for $f$, where $a<v_0$ and  $f(a)=f(v_1)$. Thus, choose $f$ so that $G$ has an irrational rotation number and consider $F:[0,1]\setminus\{a,c\}\to[0,1]$ given by
$$F(x)=\begin{cases}
f(x) & \text{ if }x>a\\
f(x)/f(a) & \text{ if }x<a
\end{cases},
$$
see Figure~\ref{GAP-Ewi.png}.
The interval $I$ is a wandering interval for $F$ with the exceptional point $a$ belonging to the boundary of $I$. 
%%%%%%%%%%%%%%%%%%%%%%%%%%%%%%%%%%%%%%%%%%%%%%
\begin{figure}
\begin{center}\includegraphics[scale=.3]{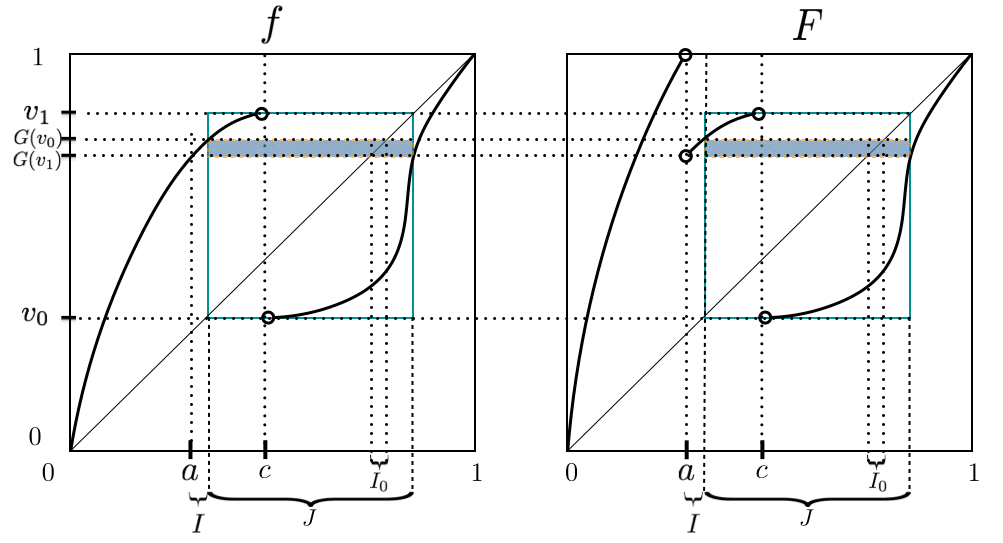}\\
\caption{The map $F$ on the right is equal to $f$ in $(a,1]$ and it has a wandering interval $I$ with $\partial I\cap\cc_F\ne\emptyset$.
}
\label{GAP-Ewi.png}
\end{center}
\end{figure}
%%%%%%%%%%%%%%%%%%%%%%%%%%%%%%%%%%%%%%%%%%%%%%
\end{Example}

For this, define $$\BB(f)=\BB_0(f)\cup\BB_1(f)\cup\BB_2(f),$$
where
$$\BB_2(f)=\{x\in[0,1]\setminus\co_f^-(\cc_f)\,;\,\co_f^+(x)\cap J_{c_\pm}\ne\emptyset\text{ for some }c\in(\cc_f)_{0}^{\pm}\}$$
and $$(\cc_f)_{0}^{\pm}=\{c\in\cw_f^\pm\,;\,p_{c_\pm}\in\co_f^-(\cc_f)\}.$$

Set $(\cc_f)_\pm^\star=\big((\cc_f)^\pm\setminus\BB_0(f)\big)\setminus(\cc_f)^\pm_0$. That is, $(\cc_f)_-^\star$ is the set of points $c\in\cc_f$ such that $(c-\varepsilon,c)\not\subset\BB_0(f)$  $\forall\,\varepsilon>0$ and $c\in\cw_f^-$ only if $p_{c_-}\notin\co_f^-(\cc_f)$. Similarly, $(\cc_f)_+^\star$ is the set of points $c\in\cc_f$ such that $(c,c+\varepsilon)\not\subset\BB_0(f)$ $\forall\,\varepsilon>0$ and $c\in\cw_f^+$ only if $p_{c_+}\notin\co_f^-(\cc_f)$. Given $c\in(\cc_f)_\pm^\star$,  define the {\em shadow of $c_\pm$} as 
$$c_\pm^\star=
\begin{cases}
c &\text{ if }c\notin\cw_f^\pm\\
p_{c_\pm} &\text{ if }c\in\cw_f^\pm
\end{cases}
$$

Given $L$ and $K\subset\cc_f$, let $$\Agemo_f(L,K)=\{x\in[0,1]\setminus\BB(f)\;;\,\cc_f^-(x)=L\text{ and }\cc_f^+(x)=K\},$$ 
where $\cc_f^\pm(x)=\{c\in(\cc_f)_\pm^\star\,;\,(c_\pm^\star)_\pm\in\omega_f(x)\}$, that is,  $$\cc_f^-(x)=\{c\in(\cc_f)_-^\star\,;\,c_-^\star\in\overline{(0,c_-^\star)\cap\co_f^+(x)}\,\}$$
and
$$\cc_f^+(x)=\{c\in(\cc_f)_+^\star\,;\,c_+^\star\in\overline{(c_+^\star,1)\cap\co_f^+(x)}\,\}.$$

%Thus, it follows from Corollary~\ref{CorollaryCritnoPoco} and Lemma~\ref{ParallaxI} that 
%\begin{equation}\label{Lemmarewritten}\omega_f(x)\subset A(\cc_f,\cc_f),\text{ for almost every }x\in[0,1]\setminus\BB(f).
%\end{equation}

%\color{blue}
%Given $x\in[0,1]$, let $[(\cc_f)_{\pm}]_x=\{\alpha\in(\cc_f)_{\pm}\,;\,\alpha^{\star}\in\omega_f(x)\}$. It follows from Lemma~\ref{ParallaxI} that $[(\cc_f)_{\pm}]_x=[(\cc_f)_{\pm}^0]_x$, for almost every $x\in[0,1]\setminus\BB(f)$.
%\color{black}

%Given any $c\in\cc_f$, let $$\cu(c_\pm)=\{x\in[0,1]\setminus\BB(f)\,;\,(p_{c_\pm})_\pm\notin\omega_f(x)\}.$$
\begin{Lemma}[Parallax]\label{ParallaxII}
Given $L$ and $K\subset\cc_f$, we have that
$$\omega_f(x)=\bigg(\bigcup_{c\in L}\overline{\co_f^+(c_-)}\bigg)\cup\bigg(\bigcup_{c\in K}\overline{\co_f^+(c_+)}\bigg),$$
for almost every $x\in\Agemo_f(L,K)$.
\end{Lemma}
\begin{proof}

As $\omega_f(x)\supset \big(\bigcup_{c\in L}\overline{\co_f^+(c_-)}\big)\cup\big(\bigcup_{c\in K}\overline{\co_f^+(c_+)}\big)$ for every $x\in\Agemo_f(L,K)$, we only need  to show that the reverse inclusion is true for almost every $x\in\Agemo_f(L,K)$.

Define $$B_n=\bigg(\bigcup_{c\in (\cc_f)_0^\pm}J_{c_\pm}\bigg)\cup\bigg(\bigcup_{c\in(\cc_f)_-^\star\setminus L}(c_-^\star-1/n,c)\bigg)\cup\bigg(\bigcup_{c\in(\cc_f)_+^\star\setminus K}(c,c_+^\star+1/n)\bigg)$$
and
$$\Agemo_n=\{x\in\Agemo_f(L,K)\,;\,\co_f^+(x)\cap B_n=\emptyset\}.$$
Note that $$\Agemo_f(L,K)=\bigcup_{n\ge n_0}\Agemo_n,\;\;\forall\,n_0\ge1.$$

Let $n_0\ge1$ so that $\cc_f\cap B_{n_0}=\emptyset$ and $n>n_0$ such that $\leb(\cu_n)>0$.
For each $c\in (\cc_f)^\star_-\setminus L$
choose points $c_-^\star-1/n<a_{c_-}<b_{c_-}<c_-^\star$ such that $a_{c_-}\notin\co_f^-(\cc_f)$ and $b_{c_-}\in\co^{-}_{f}(\cc_{f})$.
By the maximality of $J_{c_-}$, the interval $(c_-^\star-1/n,c_-^\star)$ is not wandering.
Thus, it follows from Lemma~\ref{LemmaHomterval} that either $\#((c_-^\star-1/n,c_-^\star)\setminus\BB_0(f))\le1$, or $(c_-^\star-1/n,c_-^\star)\cap\co^{-}_{f}(\cc_{f})\ne\emptyset$.
The first situation is impossible. Indeed, note that either $c>c_-^\star=p_{c_-}\notin\co_f^{-}(\cc_f)$ (because $c\notin(\cc_f)_0^-$) or $c=c_-^\star$. In any case,  $\omega_f(c_-^\star)=\omega_f(c_-)$ and, as a consequence, $c\in(\cc_f)^-\cap\BB_0(f)$, contradicting our hypothesis.
So, we have $(c_-^\star-1/n,c_-^\star)\cap\co^{-}_{f}(\cc_{f})\ne\emptyset$ and we are free to choose $b_{c_-}$ as above. After that, choose any $a_{c_-}\in(c_-^\star-1/n,b_{c_-})$. Analogously, 
choose, for each $c\in (\cc_f)^\star_+\setminus K$, points $c_+^\star<b_{c_+}<a_{c_+}<c_+^\star+1/n$ with  $a_{c_+}\notin\co_f^-(\cc_f)$ and $b_{c_+}\in\co^{-}_{f}(\cc_{f})$.

For each $c\in(\cc_f)^\pm\setminus\BB_0(f)$, let
$$q_{c_-}=
\begin{cases}
(c-p_{c_-})/2 &\text{ if }c\in(\cc_f)_0^-\\
(c-b_{c_-})/2 &\text{ if }(\cc_f)^\star_-\setminus L
\end{cases}$$
and
$$q_{c_+}=
\begin{cases}
(p_{c_+}-c)/2 &\text{ if }c\in(\cc_f)_0^+\\
(b_{c_+}-c)/2 &\text{ if }(\cc_f)^\star_+\setminus K
\end{cases}.$$

Consider $g:[0,1]\setminus\cc_g\to[0,1]$ given by 
$$g(x)=
\begin{cases}
f(x) & \text{ if }x\notin{B_n}\\
q_{c_\pm}+\sigma(f(x)-f(q_{c_\pm})) & \text{ if }x\in J_{c_\pm}\text{ for some }c\in (\cc_f)_0^\pm\\
q_{c_-}+\sigma(f(x)-f(q_{c_-})) & \text{ if }x\in(b_{c_-},c)\text{ for some }c\in (\cc_f)^\star_-\setminus L\\
q_{c_+}+\sigma(f(x)-f(q_{c_+})) & \text{ if }x\in(c,b_{c_+})\text{ for some }c\in (\cc_f)^\star_+\setminus K
\end{cases},
$$
where $\sigma=(2\sup|f'|)^{-1}$ and
$$\cc_g=\cc_f\cup\{p_{c_\pm}\,;\,c\in(\cc_f)_0^\pm\}\cup\{b_{c_-}\,;\,c\in(\cc_f)_-^\star\setminus L\}\cup\{b_{c_+}\,;\,c\in(\cc_f)_+^\star\setminus K\}.$$

Note that $\BB_0(g)\supset(\BB_0(f)\cup B_n)$. 
As $f(x)=g(x)$ for all $x\notin B_n$, we have that $f^j(x)=g^j(x)$ $\forall\,x\in\Agemo_n$ and $\forall\,j\ge0$. Thus, it follow from Corollary~\ref{CorollaryCritnoPoco} that
$$\omega_f(x)=\omega_g(x)\subset\bigcup_{c\in(\cc_g)^\pm\setminus\BB_0(g)}\overline{\co_g^+(c_\pm)},$$
for almost every $x\in [0,1]\setminus(\BB_0(g)\cup\BB_1(g))\supset\Agemo_n$.

Note that  $c_\pm$ and $(p_{c_\pm})_\mp\in\BB_0(g)$ for every $(\cc_f)_0^\pm$,  $c_-$ and $(b_{c_-})_+\in\BB_0(g)$ for every $c\in (\cc_{f})_-^\star\setminus L$ and also,
$c_+$ and $(b_{c_+})_-\in\BB_0(g)$ for every $c\in (\cc_f)_-^\star\setminus K$.
Thus, writing
$$\widetilde{L}=\{p_{c_-}\,;\,c\in(\cc_f)_0^-\,\text{ and }\,(p_{c_-})_-\notin\BB_0(g)\}\cup\{b_{c_-}\,;\,c\in(\cc_f)_-^\star\setminus L\,\text{ and }\,(b_{c_-})_-\notin\BB_0(g)\}$$
and
$$\widetilde{K}=\{p_{c_+}\,;\,c\in(\cc_f)_0^+\,\text{ and }\,(p_{c_+})_+\notin\BB_0(g)\}\cup\{b_{c_+}\,;\,c\in(\cc_f)_+^\star\setminus K\text{ and }(b_{c_+})_+\notin\BB_0(g)\},$$
we have $$(\cc_g)^-\setminus\BB_0(g)=L\cup\widetilde{L}\;\;\text{ and }\;\;(\cc_g)^+\setminus\BB_0(g)=K\cup\widetilde{K}.$$
Furthermore,
$$
c\in\cc_f\implies
\begin{cases}
	c\in L &\text{ or }c_-\in\BB_0(g)\\
	c\in K &\text{ or }c_+\in\BB_0(g)
\end{cases},
$$
$$
\eta\in\widetilde{L} \implies \co_g^+(\eta_-)\subset\{\eta,f(\eta),\cdots,f^{n_{\eta}-1}(\eta)\}\cup\bigg(\bigcup_{c\in L}\co_f^+(c_-)\,\cup\,\bigcup_{c\in K}\co_f^+(c_+)\bigg)
$$
and
$$
\eta\in\widetilde{K}\implies \co_g^+(\eta_+)\subset\{\eta,f(\eta),\cdots,f^{n_{\eta}-1}(\eta)\}\cup\bigg(\bigcup_{c\in L}\co_f^+(c_-)\,\cup\,\bigcup_{c\in K}\co_f^+(c_+)\bigg),
$$
where $n_{\eta}=\min\{j\ge0\,;\,\eta\in f^{-j}(\cc_f)\}$.

Given $\gamma\in[0,1]$ and $j\ge0$, let
$$\Agemo_n(j,\gamma_-)=\{x\in B_n\,;\,\gamma\in\overline{\co_f^+(x)\cap f^j((\gamma-\varepsilon,\gamma))}\;\;\forall\,\varepsilon>0\}$$
and 
$$\Agemo_n(j,\gamma_+)=\{x\in B_n\,;\,\gamma\in\overline{\co_f^+(x)\cap f^j((\gamma,\gamma+\varepsilon))}\;\;\forall\,\varepsilon>0\}.$$

As $\co_g^+(c_-)=\co_f^+(c_-)$ for every $c\in L$ and $\co_g^+(c_+)=\co_f^+(c_+)$ for every $c\in K$, the proof of the lemma follows from the claim below.

\begin{claim}
If $\gamma\in\widetilde{L}$ and $\leb\big(\Agemo_n(\ell,\gamma_-)\big)>0$ for some $\ell\ge0$, then $$f^\ell(\gamma_-)\in\bigcup_{c\in L}\overline{\co_f^+(c_-)}\,\cup\,\bigcup_{c\in K}\overline{\co_f^+(c_+)}.$$
Similarly, $f^\ell(\gamma_+)\in\bigcup_{c\in L}\overline{\co_f^+(c_-)}\,\cup\,\bigcup_{c\in K}\overline{\co_f^+(c_+)}$, whenever $\gamma\in\widetilde{K}$, $\leb\big(\Agemo_n(\ell,\gamma_+)\big)>0$ and $\ell\ge0$.
\end{claim}	

Suppose that $\gamma\in\widetilde{L}$ and $\leb\big(\Agemo_n(\ell,\gamma_-)\big)>0$ for some $\ell\ge0$, the case when $\gamma\in\widetilde{K}$ is analogous.
Let $n_\gamma=\min\{j\ge0\,;\,\gamma\in f^{-j}(\cc_f)\}$ and $\delta_\gamma>0$ small such that $f^{n_\gamma}|_{(\gamma-\delta_\gamma,\gamma)}$ is a diffeomorphism.
Note that $g^j(\gamma_-)=f^j(\gamma)$ for $0\le j \le n_\gamma$. 

Firstly let us assume that $0\le\ell\le n_\gamma$ and
set $p:=f^{\ell}(\gamma)=g^{\ell}(\gamma_-)$.

Suppose by contradiction that $p\notin\bigcup_{c\in L}\overline{\co_f^+(c_-)}\,\cup\,\bigcup_{c\in K}\overline{\co_f^+(c_+)}$. 
As $\co_f^+(a_-)=\co_g^+(a_-)$ and $\co_f^+(b_+)=\co_g^+(b_+)$  for every $a\in L$ and $b\in K$, we get
\begin{equation}\label{EQQut76u}
	p\notin\bigcup_{c\in L}\overline{\co_g^+(c_-)}\,\cup\,\bigcup_{c\in K}\overline{\co_g^+(c_+)}.
\end{equation}

Let $(\alpha',\beta')$ be the connected component of $[0,1]\setminus\big(\bigcup_{c\in L}\overline{\co_g^+(c_-)}\,\cup\,\bigcup_{c\in K}\overline{\co_g^+(c_+)}\;\big)$ containing $p$.
%Note that $(\alpha',\beta')$ is a  nice interval for $g$.

%For each $c\in\cc_g^-\cap\BB_0(g)$ let $t_{c_-}<c$ be such that $[t_{c_-},c)\subset\BB_0(g)$ and $t_{c_-}\notin\co_g^-(\cc_g)$. Similarly, consider $t_{c_+}>c$ such that $(c,t_{c_+}]\subset\BB_0(g)$ and $t_{c_+}\notin\co_g^-(\cc_g)$, whenever $c\in\cc_g^+\cap\BB_0(g)$.
%Let $m=\max\{n_{\eta}\,;\,\eta\in\widetilde{L}\cup\widetilde{K}\}$ and $$T=\bigcup_{c\in\cc_g^\pm\cap\BB_0(g)}\{g^j(t_{c_\pm})\,;\,0\le j\le m\}.$$
% As $\leb(\Agemo_n(\ell,\gamma_-))>0$, it follows that $p\notin T$. Thus, let $(\alpha'',\beta'')$ be the connected component of $[0,1]\setminus T$ containing $p$.
 
Note that $p\notin A$, where $$A=\bigcup_{c\in\widetilde{L}}\{g^j(a_{c_-})\,;\,0\le j\le m\}\;\;\cup\; \bigcup_{c\in\widetilde{K}}\{g^j(a_{c_+})\,;\,0\le j\le m\}$$
and $m=\max\{n_{\eta}\,;\,\eta\in\widetilde{L}\cup\widetilde{K}\}$.
Indeed, if $A\ni g^j(a_{c_\pm})=p$, then $a_{c_\pm}\notin\BB_0(g)$ and so $g^i(a_{c_\pm})=f^i(a_{c_\pm})$ $\forall\,i\ge0$. So, $a_{c_\pm}\notin\co_f^-(\cc_f)\ni p$, which is a contradiction. Therefore, let $(\alpha'',\beta'')$ be the connected component of $[0,1]\setminus A$ containing $p$.

Let us assume that $f^{\ell}|_{(\gamma-\delta_\gamma,\gamma)}$ preserves orientation, the other case is analogous.
Thus, $\leb(\Agemo_n(\ell,\gamma_-)\cap (\alpha,p))>0$ and $p\in\overline{\co_f^+(x)\cap[0,p)}$ for every $x\in\Agemo_n(\ell,\gamma_-)$.
As a consequence, $p_-$ cannot be an attracting period-like point and $$\alpha'''=\max\bigg\{[0,p)\cap\bigcup_{c\in(\cc_g)^\pm\cap\BB_0(g)}\co_g^+(c_\pm)\bigg\}$$ is well defined.

If $p_+$ is not an attracting periodic-like point, then define
\begin{equation}\label{EqBETA1}
\beta'''=\min\bigg\{\{1\}\cup(p,1)\cap\bigcup_{c\in(\cc_g)^\pm\cap\BB_0(g)}\co_g^+(c_\pm)\bigg\}.	
\end{equation}
Otherwise, i.e., if $p_+$ is an attracting periodic point, let $t_0$ be the period of $p_+$. As $p_-$ does not belong to the orbit of $p_+$, let $\beta_0>p$ be such that $g^t|_{(p,\beta_0)}$ is an orientation preserving diffeomorphism, $g^t((p,\beta_0))=(p,g^t((\beta_0)_-))\subset(p,\beta_0)$ and $g^j((p,\beta_0))\cap(\alpha''',\beta_0)=\emptyset$  for every $1\le j<t_0$.
%\color{red}
%\begin{equation}\label{EqBETA2}
%\beta'''=\min\bigg\{\{\beta_0\}\cup(p,\beta_0)\cap\bigg(\bigcup_{c\in(\cc_g)^\pm\cap\BB_0(g)}\{g^{j}(c_\pm)\,;\,0\le j\le R(x)\}\bigg)\bigg\}.	
%\end{equation}
%\color{black}
For each $c\in\cc_g^\pm((p,\beta_0))$, set $s(c_\pm)=\min\{j\ge0\,;\,f^j(c_\pm)\in(p,\beta_0)\}$ and define
\begin{equation}\label{EqBETA2}\beta'''=\min\{g^{s(c_\pm)}(c_\pm)\,;\,c\in\cc_f^\pm((p,\beta_0))\}.
\end{equation}

Set 
$$I=(\alpha,\beta)=(\alpha',\beta')\cap(\alpha'',\beta'')\cap(\alpha''',\beta''').$$
By construction, we get that
\begin{equation}\label{EqclaimXYZ0}
\co_g^+(c_\pm)\cap(\alpha,p)=\emptyset\text{ for every }c\in\cc_g	
\end{equation}
and also that
\begin{equation}\label{EqclaimXYZ1}
\co_g^+(\alpha)\cap(\alpha,p)=\emptyset=\co_g^+(\beta)\cap(\alpha,\beta).	
\end{equation}

For each $x\in\Agemo_n(\ell,\gamma_-)\cap(\alpha,p)$, let $R(x)=\min\{i\ge1\,;\,g^i(x)\in(\alpha,p)\}$ and let $I_x$ be the maximal open interval satisfying the following  three conditions: (i) $x\in I_x\subset(\alpha,p)$, (ii) $g^{R(x)}|_{I_x}$ is a diffeomorphism and (iii) $g^{R(x)}(I_x)\subset(\alpha,\beta)$.

\begin{subclaim}
	$g^{R(x)}(I_x)=(\alpha,\beta)$, $\forall\,x\in\Agemo_n(\ell,\gamma_-)\cap(\alpha,p)$.
\end{subclaim}
\begin{proof}[Proof of the subclaim]
Suppose for instance that $g^{R(x)}(I_x)\subsetneqq(\alpha,\beta)$ and write $(a,b)=I_x$.
So, $\exists\,0\le i<R(x)$ and $c\in\cc_g$ such that $g^i(a)=c$ and $g^{R(x)}(a_+)\in(\alpha,\beta)$ or $g^i(b)=c$ and $g^{R(x)}(b_-)\in(\alpha,\beta)$.

It follows from (\ref{EqclaimXYZ0}), (\ref{EqclaimXYZ1}) and from $g^{R(x)}(x)\in(\alpha,p)$ that
\begin{enumerate}
\item[(I)] either $g^{R(x)}(a_+)=\alpha$, $g^i(b_-)=c$ for some $0\le i<R(x)$ and $g^{R(x)}(b_-)\in[p,\beta)$
\item[(II)] or $g^{R(x)}(b_-)=\alpha$, $g^i(a_+)=c$ for some $0\le i<R(x)$ and $g^{R(x)}(a_+)\in[p,\beta)$. 
\end{enumerate}
%\begin{center}
%	either $g^i(a_+)=c$ for some $0\le i<R(x)$, $g^{R(x)}(a_+)\in[p,\beta)$ and $g^{R(x)}(b_-)=\alpha$\\
%or $g^i(b_-)=c$ for some $0\le i<R(x)$, $g^{R(x)}(b_-)\in[p,\beta)$ and $g^{R(x)}(a_+)=\alpha$.
%\end{center} 

We may assume (I), i.e., that $g^{R(x)}(a_+)=\alpha$, $g^i(b_-)=c$ for some $0\le i<R(x)$ and $g^{R(x)}(b_-)\in[p,\beta)$,  the proof of (II) is analogous. 
That being so, either $g^i(a_+)<g^i(b_-)=c$ or $g^i(a_+)>g^i(b_-)=c$.
Assume that $g^i(a_+)<g^i(b_-)=c$, the second case is analogous.

If $c\in(\cc_f)^-\setminus\BB_0(f)=L\cup\widetilde{L}$, then $g^{R(x)}(b_-)$ $=$ $g^{R(x)-i}(c_-)\in A$ $\cup$ $ \bigcup_{c\in L}\overline{\co_g^+(c_-)}$ $\cup$ $\bigcup_{c\in K}\overline{\co_g^+(c_+)}$, contradicting that
$(\alpha,\beta)$ $\cap$ $\big( A$ $\cup \bigcup_{c\in L}\overline{\co_g^+(c_-)}$ $\cup$ $\bigcup_{c\in K}\overline{\co_g^+(c_+)}\big)$ $=$ $\emptyset$.

Thus, we may suppose that 
$c\in(\cc_f)^-\cap\BB_0(f)$. In this case, there is $r>0$ such that $(c-r,c)\subset\BB_0(g)$ and this implies that
\begin{equation}\label{Eqsubclaim1}
	g^{R(x)}(b_-)=g^{R(x)-i}(c_-)\ne p,
\end{equation}
otherwise $(g^{R(x)-i}(c-r),p)\subset\BB_0(g)$, contradicting that $\leb(\Agemo_n(\ell,\gamma_-))>0$.

If $p_+$ is not an attracting periodic-like point, then it follows from (\ref{EqBETA1}) and (\ref{Eqsubclaim1}) that $g^{R(x)}(b_-)=g^{R(x)-i}(c_-)\notin[p,\beta''')\supset[p,\beta)$, contradicting (I).
Thus, we may assume that $p_+$ is an attracting periodic-like point.
%Suppose that $g^{i+k}(I_x)\cap(\alpha,\beta)\ne\emptyset$ for some $0\le k<R(x)-i$.
%That being so, let $k'=\max\{0\le k<R(x)-i\,;\,g^{i+k}(I_x)\cap(\alpha,\beta)\ne\emptyset\}$.
%As $Sg<0$ and $g^{R(x)}|_{I_x}$ is a diffeomorphism, if 
%\color{red}
%In this case $\beta=\beta'''$. Thus, as $0\le R(x)-i\le R(x)$, it follows from (\ref{EqBETA2}) and (\ref{Eqsubclaim1}) that $g^{R(x)}(b_-)=g^{R(x)-i}(c_-)\notin[p,\beta''')=[p,\beta)$, which contradicts (I) again.
%\color{black}

In this case, $\beta=\beta'''\le\beta_0$. Suppose for instance that $g^{k}(c_-)\in[p,\beta''')$ for some $0\le k<R(x)-i$.
As $i+k<R(x)$, we get that $g^{i+k}(x)\notin(\alpha,\beta)$.
So, either 
\begin{equation}\label{Eqsubclaim2}
	g^{i+k}(a_-)<g^{i+k}(x)\le\alpha<p\le g^{k}(c_-)<\beta=\beta'''\le\beta_0
\end{equation}
or
\begin{equation}\label{Eqsubclaim3}
p\le g^{k}(c_-)<\beta=\beta'''\le\beta_0<g^{i+k}(x)<g^{i+k}(a_-).	
\end{equation}
As (\ref{Eqsubclaim2}) implies that $g^{R(x)-(i+k)}(\alpha)\in(\alpha,\beta)$, it follows from (\ref{EqclaimXYZ1}) that $g^{R(x)-(i+k)}(\alpha)\in[p,\beta)$ and so, $g^{R(x)-(i+k)}((\alpha,p))\subset(p,\beta)\subset(p,\beta_0)$. In particular, $\Agemo_n(\ell,\gamma_-)\cap(\alpha,p)\subset\BB_0(g)$ which is a contradiction.
Thus, we may assume (\ref{Eqsubclaim3}).

Therefore, $R(x)-(i+k)=j t_0$ for some $j\ge1$ and $g^{R(x)-(i+k)}|_{(p,g^{i+k}(a_-))}$ is an orientation preserving diffeomorphism, see definition of $t_0$, $\beta_0$ and $s(c_-)$ in the paragraph above (\ref{EqBETA2}). But this is a contradiction, as $$g^{R(x)-(i+k)}(g^{i+k}(a_-))=\alpha<g^{R(x)}(x)=g^{R(x)-(i+k)}(g^{i+k}(x)).$$

So, we conclude that  $g^{k}(c_-)\notin[p,\beta''')$ for every $0\le k<R(x)-i$. This implies that $s(c_-)=R(x)-i$ and as a consequence, $g^{R(x)}(b_-)=g^{R(x)-i}(c_-)>\beta$, contradicting (I).
\end{proof}

%
%Indeed, choose  $x\in\Agemo_n(\ell,\gamma_-)\cap(\alpha,p)$ and write $(a,b)=I_x$.
%It is obvious that $g^{R(x)}(I_x)=(\alpha,\beta)$ whenever $g^j(a)$ and $g^j(b)\notin\cc_g$ for every $0\le j<R(x)$.
%So, assume for instance that $\cc_g\cap\{g^j(a_+)\,;\,0\le j<R(x)\}\ne\emptyset$ and that $\alpha<g^{R(x)}(a_+)<\beta$, the case when  $g^j(b_-)\in\cc_g$ for every $0\le j<R(x)$ is analogous.

Let $U:=\bigcup_{x\in (\alpha,p)\cap\Agemo_n(\ell,\gamma_-)}I_x$.
As $I_x\cap I_y\ne\emptyset\implies I_x=I_y$, $\forall\,x,y\in(\alpha,p)\cap\Agemo_n(\ell,\gamma_-)$, the application $R:U\to\NN$ given by $R|_{I_x}\equiv R(x)$ $\forall\,x\in (\alpha,p)\cap\Agemo_n(\ell,\gamma_-)$ is well defined.
Let $G:U\to(\alpha,\beta)$ be the induced map given by $G(x)=g^{R(x)}(x)$.
As $p\in\omega_G(x)$ for every $x\in (\alpha,p)\cap\Agemo_n(\ell,\gamma_-)$, it follows from Lemma~\ref{Lemma375945194558} that $\omega_g(x)\supset\omega_G(x)=[\alpha,\beta]$ for almost every $x\in(\alpha,\beta)$. As $\omega_f(x)=\omega_g(x)$ for every $x\in\Agemo_n(\ell,\gamma_-)$, we get that almost every point $x\in (\alpha,p)\cap\Agemo_n(\ell,\gamma_-)$ is contained in $\BB_1(f)$, contradicting the definition of $\Agemo_n(\ell,\gamma_-)$.

%%%%%%%%%%%%%%%%%%%%%%%%%%%%%%%%%%%%%%%%%%%%%%
\begin{figure}
\begin{center}\includegraphics[scale=.25]{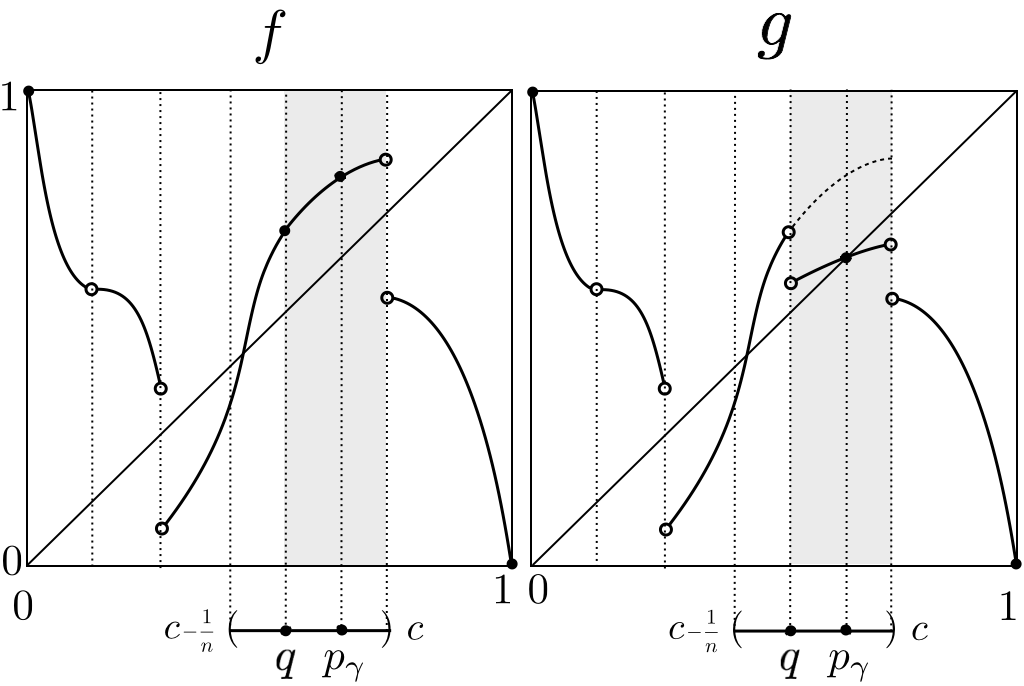}\\
\caption{
The map $g$ (on the right side) is equal to $f$ outside the interval $[q,c]$, it has $p\in (q,c)$ as a fixed point and $\cc_f\cup\{q\}$ as its exceptional set.
}
%In this picture we have in the left side a map $f$ and in the right side the associated map $g$ which is equal to $f$ outside the interval $[q,c]$, has $p\in (q,c)$ as a fixed point and $\cc_f\cup\{q\}$ as its exceptional set (see Lemma~\ref{ParallaxII}).
\label{MapaPerturbadocaso2.png}
\end{center}
\end{figure}
%%%%%%%%%%%%%%%%%%%%%%%%%%%%%%%%%%%%%%%%%%%%%%

\end{proof}

\begin{Theorem}\label{Theorem1aa}
For almost every point $x\in[0,1]\setminus(\BB_0(f)\cup\BB_1(f))$ we have that 
\begin{equation}\label{Eq8287vbj}
\omega_f(x)=
\bigcup_{\footnotesize{
\begin{array}{c}
(c_\pm^\star)_\pm\in\omega_f(x)\\
c\in(\cc_f)_{\pm}^\star
\end{array}
}}
\overline{\co_f^+((c_\pm^\star)_\pm)}.	
\end{equation}
\end{Theorem}
\begin{proof}
It follows from the Parallax Lemma~\ref{ParallaxII} that 
(\ref{Eq8287vbj}) is true 
for almost every $x\in[0,1]\setminus(\BB_0(f)\cup\BB_1(f)\cup\BB_2(f))$. Thus, we need only to check that (\ref{Eq8287vbj}) is also true for almost every $x\in\BB_2(f)$.

Recall that $x\in\BB_2(f)$ means that $f^{j_0}(x)\in J_{c_\pm}$ for some $j_0\ge0$ and an exceptional wandering interval $J_{c_\pm}=(p_-,c)$ or $(c,p_+)$, where $c\in\cc_f$. As there are at most $2\#\cc_f$ exceptional wandering intervals, $n_x=\min\{n\ge1\,;,f^{n_x}(x)\notin\BB_2(f)\}$ is well define. Moreover, $f^{n_x}(x)\in[0,1]\setminus(\BB_0(f)\cup\BB_1(f)\cup\BB_2(f))$, because $f^{n_x}(x)\in f^{n_x-j_0}(J_{c_\pm})$ and $f^{n_x-j_0}(J_{c_\pm})$ still is a wandering interval and so, it cannot intersect $\BB_0(f)$ or $\BB_1(f)$. As $\omega_f(x)=\omega_f(f^{n_x}(x))$ and $\leb\circ f^{-1}\ll\leb$, it follows from Lemma~\ref{ParallaxII} that (\ref{Eq8287vbj}) is true for almost every $x\in\BB_2(f)$.
\end{proof}

\begin{Corollary}\label{CorollaryTirandoMane}
If $f:[0,1]\setminus\cc_f\to[0,1]$ is a $C^3$ local diffeomorphism with negative Schwarzian derivative, with $f(\{0,1\})\subset\{0,1\}$ and $\cc_f\subset(0,1)$ finite, then $\omega_f(x)\cap\cc_f\ne\emptyset$ for almost every $x\in[0,1]\setminus\BB_0(f)$.\end{Corollary}
\begin{proof}
By Lemma~\ref{lematres}, $\interior(\omega_f(x))\cap\cc_f\ne\emptyset$ for every $x\in\BB_1(f)$ and so, $\omega_f(x)\cap\cc_f\ne\emptyset$ for every $x\in\BB_1(f)$.
It follows from Theorem~\ref{Theorem1aa} that, for almost every $x\in[0,1]\setminus(\BB_0(f)\cup\BB_1(f))$, $\exists c\in\cc_f$ depending on $x$ such that $c_\pm^\star\in\omega_f(x)$. If $c_\pm^*=c$, we are done.
Otherwise, $c_\pm^*=p_{c_\pm}$ and $J=(p_{c_-},c)$ or $(c, p_{c_+})$ is a wandering interval. As, $\omega_f(x)\ni c_\pm^\star=p_{c_\pm}$, we get $\omega_f(x)\supset\omega_f(J)$.
So, the proof is concluded if we show that $\omega_f(J)\cap\cc_f\ne\emptyset$.
To prove that $\omega_f(J)\cap\cc_f\ne\emptyset$, suppose by contradiction that $\exists\varepsilon>0$ and $n_0\ge1$ such that $f^n(J)\cap B_{\varepsilon}(\cc_f)=\emptyset$ for every $n\ge n_0$, where $B_{\varepsilon}(\cc_f)=\bigcup_{c\in\cc_f}(c-\varepsilon,c+\varepsilon)$. Thus, consider any non-flat $C^2$ map such that $g(x)=f(x)$ for every $x\in[0,1]\setminus B_{\varepsilon}(\cc_f)$. Note that $I=f^{n_0}(J)$ is a wandering interval for $f$. As $f^j|_I=g^j|_I$ for every $j\ge0$, $I$ is also a wandering interval for $g$, contradicting the fact that $C^2$ non-flat maps don't admit wandering intervals (see Theorem~A, chapter IV in \cite{MvS}).
\end{proof}

%Now we can prove our main theorem.

\begin{proof}[Proof of Theorem~\ref{mainTheoremMTheoFofA}]
Using Singer's Theorem (see \cite{MvS}), we can conclude that the basin of attraction of each periodic-like attractor contains at least a ``lateral neighborhood of critical point'', that is, there is $\varepsilon>0$ and $c\in\cc_f$ such that $(c,c+\varepsilon)\subset\beta_f(A_j)$, i.e., $c\in(\cc_f)^-\cap\BB_0(f)$, or $(c-\varepsilon,c)\subset\beta_f(A_j)$, i.e., $c\in(\cc_f)^+\cap\BB_0(f)$.
%Let $C_0$ be the set of critical points with a lateral neighborhood contained in the basin of attraction of a periodic-like attractor. 
Each cycle of intervals contains at least one critical point in its interior (Lemma~\ref{lematres}, Section \ref{Sectionsettingsandpre}). So, let $C_1$ be the set of all critical point contained in the interior of a cycle of interval. Of course that $\big((\cc_f)^\pm\cap\BB_0(f)\big)\cap C_1=\emptyset$ and so, there exist at most $\#\cc_f$ cycles of intervals and $2(\#\cc_f-\#C_1)$ periodic-like attractors.

On the other hand, if $x\in[0,1]\setminus(\BB_0(f)\cup\BB_1(f))$, we have that $(c_\pm^\star)_\pm=c_\pm\notin\omega_f(x)$ $\forall\,c\in(\cc_f)^\pm\cap\BB_0(f)$.
Therefore, it follows from the Theorem~\ref{Theorem1aa} that for almost every $x\in[0,1]\setminus(\BB_0(f)\cup\BB_1(f))$ there is some $L_x\subset \cc_f\setminus((\cc_f)^-\cap\BB_0(f))$ and $K_x\subset \cc_f\setminus((\cc_f)^+\cap\BB_0(f))$ such that
$$\omega_f(x)=\bigg(\bigcup_{c\in L_x}\overline{\co_f^+(c_-)}\bigg)\cup\bigg(\bigcup_{c\in K_x}\overline{\co_f^+(c_+)}\bigg),$$
The number of possible pairs $(L_x,K_x)$ is at most $2^{a+b}-1$, where $a=\#(\cc_f\setminus((\cc_f)^-\cap\BB_0(f))$ and $b=\#(\cc_f\setminus((\cc_f)^+\cap\BB_0(f))$. Thus, $2^{a+b}-1\le 2^{2\#\cc_f-N}-1$, where $N$ is the number of periodic-like attractors.

Given $L\subset \cc_f\setminus((\cc_f)^-\cap\BB_0(f))\cup C_1)$ and $K\subset \cc_f\setminus((\cc_f)^+\cap\BB_0(f))\cup C_1)$, let $$A_{L,K}:=\bigg(\bigcup_{c\in L_x}\overline{\co_f^+(c_-)}\bigg)\cup\bigg(\bigcup_{c\in K_x}\overline{\co_f^+(c_+)}\bigg)$$
and let $\{A_1,\cdots,A_s\}=\{A_{L,K}\,;\,\leb(\{x\,;\,\omega_f(x)=A_{L,K}\})>0\}$.
Thus, $s\le 2^{2\#\cc_f-N}-1$.
Letting $\{A_{s+1},\cdots,A_{s+N}\}$ be the periodic-like attractors of $f$ and $\{A_{s+N+1},$ $\cdots,$ $A_{s+N+M}\}$ being the set of all cycle of intervals $I_1\cup\cdots\cup I_t$ such that $\leb(\{x\,;\,\omega_f(x)=I_1\cup\cdots\cup I_t\})>0$,
we get that $M\le\#C_1$ and that
\begin{equation}\label{Eqpkj7jm}
n\le N+M+2^{2\#\cc_f-N}-1\le \#\cc_f+2^{2\#\cc_f}-1<+\infty	
\end{equation}
 and $$\leb(\beta_f(A_1))+\cdots+\leb(\beta_f(A_{n}))=1.$$
Moreover, $\omega_f(x)=A_j$ for almost every $x\in\beta_f(A_j)$ and every $1\le j \le n$.
\end{proof}

\section{Contracting Lorenz maps}
\label{SectionContractingLorenzmaps}

In the previous section we have obtained an upper bound for the number of attractors for maps of the interval based on the number of their critical points. Applying inequality  (\ref{Eqpkj7jm}) to a contracting Lorenz map, we get that the number of attractors is at most $2^{4}=16$.  In this section we refine our method in the specific case of contracting Lorenz maps, proving Theorem~\ref{mainTheoremMTheoLORB}, where we state that this number is one, for maps that do not display attracting periodic orbits.

Contracting Lorenz maps, that comes from dynamics of three dimensional flows, is an emblematic case, as it is the simplest case of maps of the interval presenting a critical point that is also a discontinuity.

In this section we shall deal with orientation preserving maps of the interval. We say that $f:[0,1]\setminus\{c\}\to[0,1]$, $c\in(0,1)$, is an {\em orientation preserving map} if for every $x\in [0,1]\setminus\{c\}$ there is $\varepsilon>0$ such that $f(a)<f(b)$ $\forall\,a,b\in(x-\varepsilon,x+\varepsilon)\cap [0,1]\setminus\{c\}$ with $a<b$. If $f$ is a local diffeomorphism, this condition is equivalent to $f'(x)>0$ for all $x\in [0,1]\setminus\{c\}$.

\begin{Lemma}\label{Lemma98fgsn54fg}
Let $f:[0,1]\setminus\{c\}\to[0,1]$ be an orientation preserving local homeomorphism and $J=(a,b)$ be a nice interval containing $c$. Let $F:J^*\to J$ be the first return map to $J$, where $J^*=J\cap\bigcup_{j\ge1}f^{-j}(J)$. If $I$ is a connected component of $J^*\cap(a,c)$ such that $F(I)\ne J$, then $c\in\partial I$ and $F(I)=(a,F(c_-))$. Analogously, if $I$ is a connected component of $J^*\cap(c,b)$ such that $F(I)\ne J$, then $c\in\partial I$ and  $F(I)=(F(c_+),b)$.
\end{Lemma}
\begin{proof}
Let $I=(t,s)$ be a connected component of $J^*\cap(a,c)$, the case when $I$ is contained in $(c,b)$ is analogous.
Let $n\ge1$ be such that $F|_I=f^n|_I$ and let $T=(t,s')$ with $s\le s'$ be the maximal interval such that $f^n|_T$ is a homeomorphism and that $F^n(T)\subset J$.
If $s<s'$, then there is some $1\le\ell<n$ such that $f^{\ell}(T)\cap J\ne\emptyset$.
As $f^{\ell}(I)\cap J=\emptyset$, since $n$ is the first return time to $J$ points in $I$, we get $J\not\supset f^{\ell}(T)$. As $c\notin f^{\ell}(T)$, because $f^n|_T$ is a diffeomorphism,  and as $f$ is orientation preserving, we get $f^{\ell}(t)<f^{\ell}(s)<a<f^{\ell}(s')$. But this implies that $f^{n-\ell}(a)\in(a,b)$, contradicting our assumption. So, $s=s'$ and $T=I$.

Now, suppose that $F(I)\ne J$. By the maximality of $T$ there is some $0\le\ell<n$ such that $c\in f^{\ell}(\partial T)$. As $f^{j}(I)\cap J=\emptyset$ $\forall\,1\le j<n$, we get $c\in \partial I$, i.e., $I=T=(t,c)$ and so, $F(I)=(F(t_+),F(c_-))$, $F(t_+):=\lim_{\varepsilon\downarrow0}F(t+\varepsilon)$.
Furthermore, as $t\ne c$ and as  $f^j(I)\cap(a,b)=\emptyset$ $\forall\,0\le j<n$, it follows that $f^j(t)\ne c$ $\forall0\le j<n$. So,
\begin{equation}
\label{Eqyuio876}
F(I)=(f^n(t),F(c_-))
\end{equation}

We claim that $f^{n}(t)=a$. Indeed, 
as $f^n(t)$ is well defined, $f$ is orientation preserving and $f^n((t,c))\subset(a,b)$, we get $a\le f^n(t)<b$. Thus, if $f^n(t)\ne a$, then $a<f^n(t)<b$. By continuity, there is $\varepsilon>0$ such that $f^n|_{(t-\varepsilon,c)}$ is a diffeomorphism, $f^j((t-\varepsilon,c))\cap(a,b)=\emptyset$ $\forall0<j<n$ and $f^n((t-\varepsilon,c))\subset(a,b)$, which contradicts the fact that $I$ is a connected component of $J^*$.

Thus, as $F(I)=(a,F(c_-))$, the proof of the lemma follows from (\ref{Eqyuio876}).

\end{proof}

\begin{Lemma}\label{LemmaI1i1iate}
Let $f:[0,1]\setminus\{c\}\to[0,1]$ be an orientation preserving local homeomorphism. If $J=(a,b)$ is a nice interval containing $c$, then  either $\#\{j\ge0\,;\,f^j(x)\in(a,c)\}<\infty$ $\forall\,x\in(a,c)$ or $\exists J'=(a',b)$ such that $c\in J'\subset J$, $\co_f^+(a')\cap(a',b)=\emptyset$  and $a'\in\per(f)$.
Similarly, either $\#\{j\ge0\,;\,f^j(x)\in(c,b)\}<\infty$ $\forall\,x\in(c,b)$ or $\exists J'=(a,b')$ such that $c\in J'\subset J$,  $\co_f^+(b')\cap(a,b')=\emptyset$ and $b'\in\per(f)$.
\end{Lemma}
\begin{proof} If $a\in\per(f)$ there is nothing to prove. Thus, we may assume that $a\notin\per(f)$.

Suppose  $\exists\,x_0\in(a,c)$ such that $\#\{j\ge0\,;\,f^j(x_0)\in(a,c)\}=\infty$. Let $F:J^*\to J$ be the first return map to $J$, where $J^*=J\cap\bigcup_{j\ge1}f^{-j}(J)$.

\begin{claim}
$\fix(F)\cap(a,c)\ne\emptyset$.
\end{claim}
\begin{proof}[Proof of the claim]

Firstly note that $a\in\partial J$ does not belong to the boundary of any connected component of $J^*$.
Indeed, if $I=(a,p)$ is a connected component of $J^*$, then by Lemma~\ref{Lemma98fgsn54fg}, $F(a_+)=a_+$.
Let $n$ be such that $F|_I=f^n|_I$.
As $\co_f^+(a)\cap(a,b)=\emptyset$, we have that $a\notin\co_f^-(c)$.
So, $F(a_+)=a_+$ implies that $f^n(a)=a$, contradicting our assumption.

First assume that $J^*\cap(a,c)$ has more than one connected component. In this case, let $I=(p,q)$ be a connected component of $J^*$ such that $q<c$. Then, by Lemma~\ref{Lemma98fgsn54fg}, $F(I)=J\supset I$. This implies that either $F$ has a fixed point $y\in I$ proving the claim or that $p=a$ and $F(a_+)=a_+$. As we have seen before, $F(a_+)=a_+$ implies that $f^n(a)=a$, where  $F|_I=f^n|_I$, contradicting our assumption.

Now, suppose that $J^*\cap(a,c)$ has only one connected component. Let $I=(p,q)$ be this single connected component and let $n\ge1$ be such that $F|_I=f^n|_I$.
We may suppose that $F|_I$ does not have a fixed point.
%Let $I=(p,c)$ for some $p\in[a,c)$. By Lemma~\ref{Lemma98fgsn54fg}, $c$ must belong to the boundary of $I$ or we have that $p=a$ and also $F(a_+)=a_+$. Again $F(a_+)=a_+$ implies that $a\in\per(f)$, contradicting our assumption. Thus, $a<p<c$.
As $F(I)=(a,F(c_-))$, by Lemma~\ref{Lemma98fgsn54fg}, and as $F(x)\ne x$ $\forall\,x\in I$, we get $F(x)<x$, $\forall\,x\in I$.
Because $I=J^*\cap(a,c)$ and $\#\{j\ge0\,;\,f^j(x_0)\in(a,c)\}=\infty$, we get $$a<p<\cdots<F^n(x_0)<F^{n-1}(x_0)<\cdots<F(x_0)<x_0<c.$$
So, $a':=\lim_{n}F^n(x_0)$ belongs to $[p,x_0]\subset(a,c)$ and $F(a')=a'$. \end{proof}

To finish the proof of the lemma, let $a'\in\fix(F)\cap[a,c)$ and let $n$ be the period of $a'$ with respect to $f$. As $F(a')=f^n(a')=a'$ and $F$ is the first return map to $(a,b)$, we get $\{f(q),\cdots,f^{n-1}(a')\}\cap(a,b)=\emptyset$. Thus, $\co_f^+(a')\cap(a',b)=\emptyset$, proving the lemma.

\end{proof}

%\begin{Lemma}\label{Lemma9uhi7tfcv}
%Let $f:[0,1]\setminus\{c\}\to[0,1]$ be a $C^1$ contracting Lorenz map. If $c_-$ is not a periodic-like point and $\exists\,\varepsilon>0$ such that $c_-\in\omega_f(x)$ for all $x\in(c-\varepsilon,c)\setminus\co_f^-(c)$ then $c_-\in\omega_f(x)$ for all $x\notin\co_f^-(c)$ in a neighborhood of $c$. Analogously, if $c_+$ is not a periodic-like point and $\exists\,\varepsilon>0$ such that $c_+\in\omega_f(x)$ for all $\forall\,x\in(c,c+\varepsilon)\setminus\co_f^-(c)$ then $c_+\in\omega_f(x)$ for all $x\notin\co_f^-(c)$ in a neighborhood of $c$. 
%\end{Lemma}
%\begin{proof}
%Suppose that $c_-$ is not a periodic-like point and that $c_-\in\omega_f(x)$ for all $x\in(c-\varepsilon,c)\setminus\co_f^-(c)$, with $\varepsilon>0$. Let $I=(c-\varepsilon,c)$. As $I\cap f^n(I)\ne\emptyset$, $I$ is not a wandering interval. As for all $x\in(c-\varepsilon,c)\setminus\co_f^-(c)$, we have that $\omega_f(x)\supset\co_f^+(c_-)$ and as $\co_f^+(c_-)$ is not a finite set, it follows that $I$ does not intersect the basin of attraction of any periodic-like orbit. Thus, by the homterval lemma, there exists $\ell\ge1$ such that $f^{\ell}|_I$ is a diffeomorphism and that $f^{\ell}(I)\ni c$ an as a consequence $c_-\in\omega_f(x)$ for all $x\notin\co_f^-(c)$ in the open interval  $f^{\ell}(I)$ containing $c$.
%\end{proof}
%

\begin{Lemma}\label{Lemma90102555}
Let $f:[0,1]\setminus\{c\}\to[0,1]$ be an orientation preserving local homeomorphism and let $(a,b)$ be an interval with $c\in (a,b)$.
If $\co_f^+(a)\cap(a,b)=\emptyset$, $a\in\per(f)$,  $c_-$ does not belong to the basin of a periodic-like attractor, $c\in\overline{\co_f^+(c_-)\cap(0,c)}$  and $\co_f^+(c_-)\cap(c,b)=\emptyset$, then there exist an open set $U\subset(a,c)$ and a continuous map $R:U\to\NN$ such that
\begin{enumerate}
\item $U\supset\{x\in(a,c)\,;\,\co_f^+(f^j(x))\cap(a,b)\ne\emptyset$ $\forall\,j>0\}$;
\item $F:U\to(a,b)$ given by $F(x)=f^{R(x)}(x)$ is a local homeomorphism;
\item $F(I)=(a,b)$ for every connected component $I$ of $U$.
\end{enumerate}
Similarly, if $\co_f^+(b)\cap(a,b)=\emptyset$, $b\in\per(f)$, $c_+$ does not belong to the basin of a periodic-like attractor, $c\in\overline{\co_f^+(c_+)\cap(0,c)}$  and $\co_f^+(c_+)\cap(c,b)=\emptyset$, then there exist an open set $U\subset(c,b)$ and a continuous map $R:U\to\NN$ such that
\begin{enumerate}
\item $U\supset\{x\in(c,b)\,;\,\co_f^+(f^j(x))\cap(a,b)\ne\emptyset$ $\forall\,j>0\}$;
\item $F:U\to(a,b)$ given by $F(x)=f^{R(x)}(x)$ is a local homeomorphism;
\item $F(I)=(a,b)$ for every connected component $I$ of $U$.
\end{enumerate}
\end{Lemma}
%%%%%%%%%%%%%%%%%%%%%%%%%%%%%%%%%%%%%%%%%%%%%%%
%\begin{figure}
%\begin{center}\includegraphics[scale=.32]{ApliIndESPERTA.png}
%\caption{The aim of Lemma~\ref{Lemma90102555} is to construct an induced map $F:U\to(a,b)$, with $U\subset(a,c)$ or $(c,b)$ as sketched in this picture, that is, each connected component of the domain has to be sent by the induced map onto the interval $(a,b)$ and the domain $U$ has to contain all the points of $(a,c)$ or $(c,b)$ that will return to $(a,b)$.}\label{ApliIndESPERTA.png}
%\end{center}
%\end{figure}
%%%%%%%%%%%%%%%%%%%%%%%%%%%%%%%%%%%%%%%%%%%%%%%
\begin{proof}
Our purpose in this lemma is to construct an induced map $F:U\to(a,b)$, with $U\subset(a,c)$ or $U\subset(c,b)$ such that $F(I)=(a,b)$ for every connected component $I$ of $U$. The map $F$ will not be the restriction of the first return map to $(a,b)$ and it will be constructed inductively. For this, suppose that $a\in\per(f)$,  $c_-$ does not belong to the basin of a periodic-like attractor, $c\in\overline{\co_f^+(c_-)\cap(0,c)}$  and $\co_f^+(c_-)\cap(c,b)=\emptyset$ (the other case concerning $c_+$ is analogous).

%{\color{blue}Informally, we begin with the first map $F$ to $(a,b)$  and consider its restriction to $(a,c)$, see Figure~\ref{}. If $I$ is a connected component of $F|_{(a,c)}$ then either $F(I)=(a,b)$ or $I=(\alpha,c)$. Thus, we will maintain $F|_{(a,\alpha)}$ in the final indyced map, but change the induced time in $(\alpha,c)$. Using that $f$ is a orientation preserving map,   
%}

Let $r:\cu\to\NN$ be the first return time to $(a,b)$, i.e., $r(x)=\min\{j\ge1\,;\,f^j(x)\in(a,b)\}$, where $\cu$ is the set of points $x\in(a,b)\setminus\{c\}$ such that $f^n(x)\in(a,b)$ for some $n\in\NN$.
Let $\cf:\cu\to(a,b)$ be the first return map to $(a,b)$, that is, $\cf(x)=f^{r}(x)$.
% (see Figure~\ref{ApliIndESPERTA2.png} for a sketch of some possible graphics of $\cf$).
%%%%%%%%%%%%%%%%%%%%%%%%%%%%%%%%%%%%%%%%%%%%%%%
%\begin{figure}[H]
%\begin{center}\includegraphics[scale=.35]{ApliIndESPERTA2.png}
%\caption{In this picture we have sketches of some possible graphics of the first return map $\cf$ to the interval $(a,b)$ (see Lemma~\ref{Lemma90102555}).}\label{ApliIndESPERTA2.png}
%\end{center}
%\end{figure}
%%%%%%%%%%%%%%%%%%%%%%%%%%%%%%%%%%%%%%%%%%%%%%%

Let $U_0=\cu\cap(a,c)$ and $\cp_0$ be the collection of connected components of $U_0$.
As $a\in \per(f)$ and $\co_f^+(c_-)\cap(a,c)\ne\emptyset$, there are $I_a,I_0\in\cp_0$ such that $a\in\partial I_a$ and $c\in\partial I_0$. Write $I_a=(a,\alpha)$ and $I_0=(t_0,c)$. Note that $\cf(I)=(a,b)$ $\forall I\in\cp_0\setminus\{I_0\}$ and, as $f$ preserves orientation, 
$\cf(I_0)=(a,f^{r(I_0)}(c_-))\subset(a,c)$. Furthermore, $I_0\ne I_a$. Otherwise $I_0=(a,c)$ and, as $\cf(I_0)\subset(a,c)$, this will imply the existence of a periodic-like attractor, contradicting our hypothesis. Set $R_0=r|_{(a,c)}$ and $F_0=\cf|_{U_0}$.

We now construct a sequence $F_n:U_n\to(a,b)$ of $f$-induced maps defined on open sets $U_n\subset(a,c)$, with induced time $R_n$, i.e., $F_n(x)=f^{R_n(x)}(x)$.
The collection of connected components of $U_n$ will be denoted by $\cp_n$, $n\in\NN$. For each $n\ge0$ there will be an element of $\cp_n$, denoted by $I_n$, such that $c\in\partial I_n$.
This sequence will satisfy the following properties:
\begin{enumerate}
\item $I_a\in\cp_n$ $\forall\,n$;
\item $F_n(I)=(a,b)$ $\forall\,I\in\cp_n\setminus\{I_n\}$ $\forall\,n$;
\item $F_n(I_n)=\big(a,f^{R_n(I_n)}(c_-)\big)\subset(a,c)$ $\forall\,n$;
\item $R_0(I_0)<R_1(I_1)<R_2(I_2)<\cdots$;
\item $a<t_0<t_1<t_2<t_3<\cdots$, where $(t_n,c)=I_n$ $\forall\,n$;
\item $U_{n+1}\cap(a,t_n)=U_n\cap(a,t_n)$ $\forall\,n$;
\item $F_{n+1}|_{(a,t_n)\cap U_{n+1}}=F_{n}|_{(a,t_n)\cap U_n}$ $\forall\,n$;
\item $U_n\supset\{x\in(a,c)$ $;$ $\co_f^+(f^j(x))\cap(a,b)\ne\emptyset$ $\forall\,j>0\}$ $\forall\,n$.
\end{enumerate}

%%%%%%%%%%%%%%%%%%%%%%%%%%%%%%%%%%%%%%%%%%%%%%
\begin{figure}
\begin{center}\includegraphics[scale=.3]{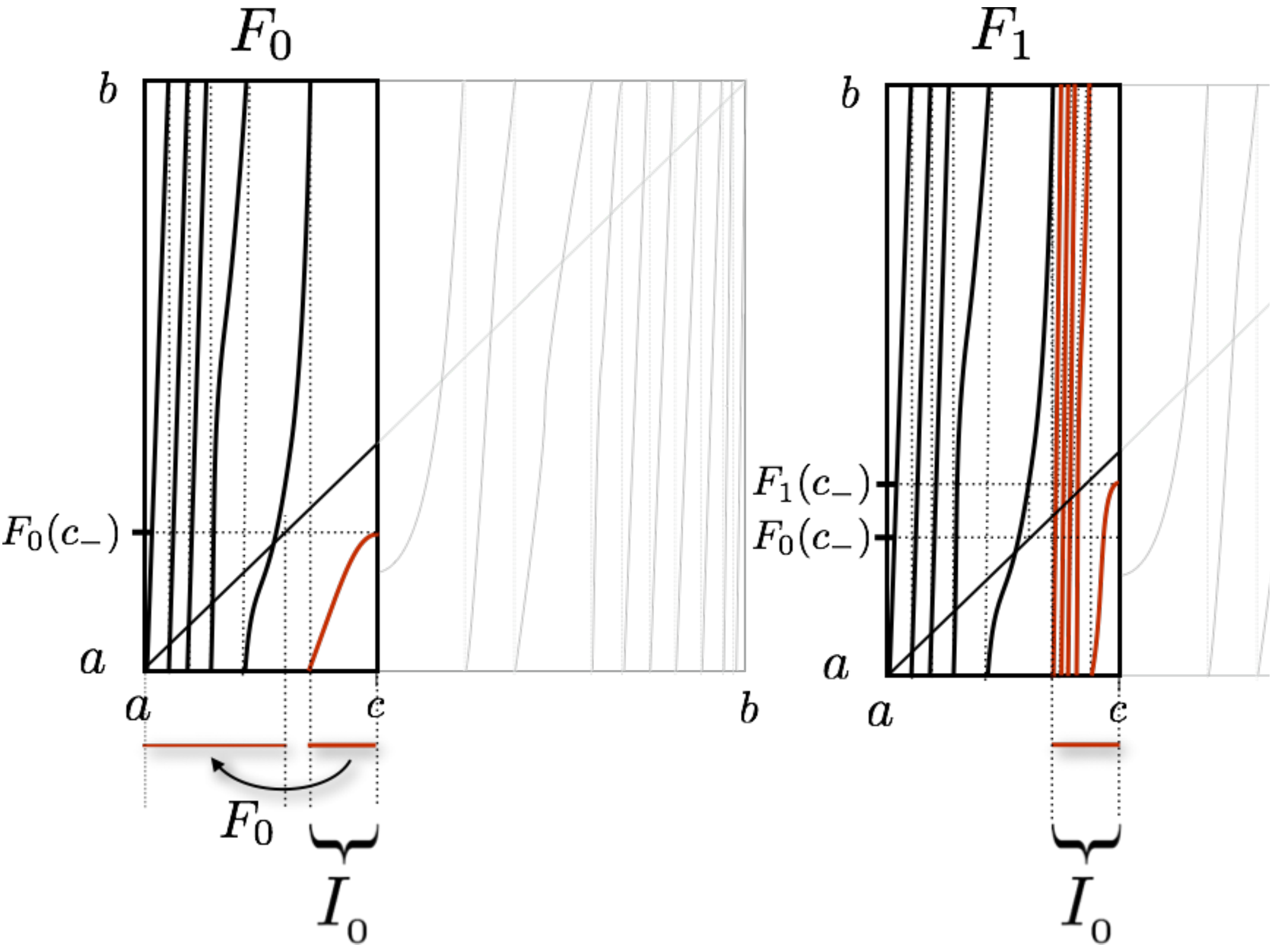}
\caption{In this picture, $F_0$ is the restriction to the interval $(a,c)$ of the first return map to $(a,b)$ and $F_1$ is the first step in the inductive construction in the proof of Lemma~\ref{Lemma90102555}.}\label{ApliIndESPERTA3.png}
\end{center}
\end{figure}
%%%%%%%%%%%%%%%%%%%%%%%%%%%%%%%%%%%%%%%%%%%%%%

Let $\ell_0=1+\max\{j\ge1\,;\,F^{j}(I_0)\subset I_a\}$.
As $c\in\overline{\RE(\co_{f}^+(c_-))\cap(a,c)}$, $\ell_0$ is the first return time with respect to $F_0$ of ``$c_-$'' to $(\alpha,c)$.
Let $R_1(x)=R_0(x)$ if $x\in(a,t_0)\cap U_0$ and $R_1(x)$ $=$ $\sum_{j=0}^{\ell_0}R_0\big(F_0^j(x)\big)$ for $x\in(t_0,c)\cap F_0^{-\ell_0}(U_0)$.
Set $U_1=\big(U_0\cap(a,t_0)\big)\cup\big((t_0,c)\cap F_0^{-\ell_0}(U_0)\big)$ and let $F_1:U_1\to(a,b)$ be given by $F_1(x)=f^{R_1}(x)$.

Let $\cp_1$ be the collection of connected components of $U_1$. Let $I_1=(F_0^{\ell_0}|_{I_0})^{-1}(I)$, where $I$ is the element of $\cp_0$ containing $F_0^{\ell_0}(c_-)$. By construction, $I_0\supsetneqq  I_1\in\cp_1$, $c\in\partial I_1$ and $R_1(I_1)=R_0(I_0)$ $+$ $(\ell_0-1)r(I_a)$ $+$ $R_0(F_0^{\ell_0}(c_-))$ $\ge R_0(I_0)+1$.
As $U_1\supset F_0^{-\ell_0}(U_0)$, we get $U_1\supset\{x\in(a,c)$ $;$ $\co_f^+(f^j(x))\cap(a,b)\ne\emptyset$ $\forall\,j>0\}$. Define $F_1:U_1\to(a,b)$ by $F_1(x)=f^{R_1(x)}(x)$ (see Figure~\ref{ApliIndESPERTA3.png}).

Inductively, suppose that $F_{n-1}:U_{n-1}\to(a,b)$ is already defined.  
Let $\ell_{n-1}=1+\max\{j\ge1\,;\,F_{n-1}^{j}(I_n)\subset I_a\}$, i.e., $\ell_{n-1}$ is the first return time with respect to $F_{n-1}$ of ``$c_-$'' to $(\alpha,c)$.
Let $R_n(x)=R_{n-1}(x)$ if $x\in(a,t_{n-1})\cap U_{n-1}$ and $R_{n}(x)$ $=$ $\sum_{j=0}^{\ell_{n-1}}R_{n-1}\big((F_{n-1})^j(x)\big)$ for $x\in(t_{n-1},c)\cap (F_{n-1})^{-\ell_{n-1}}(U_{n-1})$.
Set $U_n=\big(U_{n-1}\cap(a,t_{n-1})\big)\cup\big((t_{n-1},c)\cap (F_{n-1})^{-\ell_{n-1}}(U_{n-1})\big)$. Let $F_n:U_n\to(a,b)$ be given by $F_n(x)=f^{R_n}(x)$, $\cp_n$ be the collection of connected components of $U_n$ and $I_n=((F_{n-1})^{\ell_{n-1}}|_{I_{n-1}})^{-1}(I)$, where $I$ is the element of $\cp_{n-1}$ containing $(F_{n-1})^{\ell_{n-1}}(c_-)$. By construction, $I_{n-1}\supsetneqq  I_n\in\cp_{n}$, $c\in\partial I_{n}$ and $R_{n}(I_{n})=R_{n-1}(I_{n-1})$ $+$ $(\ell_{n-1}-1)r(I_a)$ $+$ $R_{n-1}((F_{n-1})^{\ell_{n-1}}(c_-))$ $\ge R_{n-1}(I_{n-1})+1$.
As $U_n\supset (F_{n-1})^{-\ell_{n-1}}(U_{n-1})$, we get $U_n\supset\{x\in(a,c)$ $;$ $\co_f^+(f^j(x))\cap(a,b)\ne\emptyset$ $\forall\,j>0\}$.

\begin{Claim}$t_n\to c$.\end{Claim}
\begin{proof} Otherwise $f^j|_{(t_{\infty},c)}$ will be a homeomorphism for all $\,j\in\NN$, where $t_{\infty}=\lim_n t_n$ (because $F_k$ is monotone on $(t_{\infty},c)$ $\forall\,k\ge1$ and $R_k((t_{\infty},c))=R_k(I_k)\to\infty$). That is, $(t_{\infty},c)$ is a homterval. 
It follows from the homterval lemma (Lemma~\ref{LemmaHomterval}) that $(t_{\infty},c)$ is either a wandering interval or $(t_{\infty},c)\subset\BB_0\cup\co_f^-(\per(f))$. As $(t_{\infty},c)$ can not be a wandering interval, because $c\in\overline{\co_f^+(c_-)\cap(0,c)}$, we get $(t_{\infty},c)\subset\BB_0\cup\co_f^-(\per(f))$. As $c\in\overline{\co_f^+(c_-)\cap(0,c)}$ and  $(t_{\infty},c)\subset\BB_0\cup\co_f^-(\per(f))$, it follows that $c_-$ is an attracting periodic-like point, which contradicts that $c_-$ does not belong to the basin of a periodic-like attractor.
\end{proof}

 To finish the proof, set $t_{-1}=a$, $U=\bigcup_{n\ge0}U_n\cap(t_{n-1},t_n)$ and $F:U\to(a,c)$ by $F|_{U_n\cap(t_{n-1},t_n)}=F_n|_{U_n\cap(t_{n-1},t_n)}$ for $n\ge0$.

%Let $T_0=(t_0,c)$ be the maximal interval such that $f^{n_0}|_{T_0}$ is a homeomorphism and $f^{n_0}(T_0)\subset(a,c)$. \begin{claim}
%$f^{n_0}(T_0)=(a,f^{n_0}(c_-))$. 
%\end{claim}
%\begin{proof}[Proof of the claim]
%If not, there is some $1\le s<n_0$ such that $f^s(t_0)=c$. As $\co_f^+(c_-)\cap(c,c+\varepsilon)=\emptyset$, we get $b\in f^s(T_0)$. Thus, $f^{n_0-s}(b)\in f^{n_0}(T_0)\subset(a,c)$. But this is impossible because $(a,b)$ is a nice interval.
%
%
%%%%%%%%%%%%%%%%%%%%%%%%%%%%%%%%%%%%%%%%%%%%%%%
%\begin{figure}[H]
%  \begin{center}\includegraphics[scale=.35]{UmPtoCrit2.png}\\
% \caption{}\label{UmPtoCrit2.png}
%  \end{center}
%\end{figure}
%%%%%%%%%%%%%%%%%%%%%%%%%%%%%%%%%%%%%%%%%%%%%%%
%\end{proof}

\end{proof}

%
%%%%%%%%%%%%%%%%%%%%%%%%%%%%%%%%%%%%%%%%%%%%%%%
%\begin{figure}
%\begin{center}\includegraphics[scale=.45]{UmPtoCrit.png}
%\caption{}\label{UmPtoCrit.png}
%\end{center}
%\end{figure}
%%%%%%%%%%%%%%%%%%%%%%%%%%%%%%%%%%%%%%%%%%%%%%%
%
%

\begin{Lemma}\label{Lemma85826jfuio6} Let $f:[0,1]\setminus\{c\}\to[0,1]$ be a $C^3$ contracting Lorenz map with negative Schwarzian derivative and without attracting periodic-like orbits. Suppose that $c$ is not in the boundary of a wandering interval.
If $\exists\,a<c$ such that $c\in\omega_f(x)$ for all $x\in(a,c)\setminus\co_f^-(c)$, then $\exists b>c$ such that $c_-\in\omega_f(x)$ and $c_+\in\omega_f(x)$ for all $x\in(a,b)\setminus\co_f^-(c)$.
Similarly,  
if $\exists\,b>c$ such that $c\in\omega_f(x)$ for every $x\in(c,b)\setminus\co_f^-(c)$, then $\exists a<c$ such that $c_-\in\omega_f(x)$ and $c_+\in\omega_f(x)$ $\forall\,x\in(a,b)\setminus\co_f^-(c)$.
\end{Lemma}
\begin{proof}
Suppose that $c\in\omega_f(x)$ for all $x\in(a,c)\setminus\co_f^-(c)$. It follows from the homterval lemma that there is some $n\ge1$ such that $c\in f^n((a,c))$ and that $f^n|_{(a,c)}$ is a diffeomorphism. Therefore, $c\in\omega_f(x)$ for every $x\in f^n((a,c))\setminus\co_f^-(c)$. Let $I=(\widetilde{a},\widetilde{b})$ be the maximal open interval such that
\begin{equation}\label{EQodvipb90yhbv8exfnc}c\in\omega_f(x)\text{ for every }x\in I\setminus\co_f^-(c).
\end{equation}

Let $t_0,t_1\ge0$ be the smallest integers such that $f^{t_0}((\widetilde{a},c))\cap I$ and $f^{t_1}((c,\widetilde{b}))\cap I\ne\emptyset$ (because of (\ref{EQodvipb90yhbv8exfnc}), these numbers are well defined). Furthermore, it follows from the maximality that $I$ is a nice interval and $f^{t_0}((\widetilde{a},c))\subset I\supset f^{t_1}((c,\widetilde{b}))$.
As $f$ is an orientation preserving map, $I=(\widetilde{a},\widetilde{b})$ is a nice interval and $f$ does not admit attracting periodic-like orbits, we get $f^{t_0}(\widetilde{a})=\widetilde{a}<c<f^{t_0}(c_-)<\widetilde{b}$, $\widetilde{a}<f^{t_1}(c_+)<c<\widetilde{b}=f^{t_1}(\widetilde{b})$ and that both $F_0:=f^{t_0}|_{(\widetilde{a},c)}$ and $F_1:=f^{t_1}|_{(c,\widetilde{b})}$ are diffeomorphisms.
Thus, the first return map  to $[\widetilde{a},\widetilde{b}]$, $F:[\widetilde{a},\widetilde{b}]\setminus\{c\}\to[\widetilde{a},\widetilde{b}]$, is given by
$$F(x)=\begin{cases}
F_0(x)&\text{ if }x<c\\
F_1(x)&\text{ if }x>c\\
\end{cases}$$
and it is conjugated to a contracting Lorenz map.

Suppose that exists $x\in(\widetilde{a},\widetilde{b})$ such that $c_+\notin\omega_f(x)\subset\omega_F(x)$.
Note that $\co_F^+(x)\cap(\widetilde{a},c)\ne\emptyset\ne\co_F^+(x)\cap(c,\widetilde{b})$, because $F|_{(\widetilde{a},c)}$ is strictly increasing, $F|_{(c,\widetilde{b})}$ is strictly decreasing and $\lim_{\varepsilon\downarrow0}F(c+\varepsilon)<c<\lim_{\varepsilon\downarrow0}F(c-\varepsilon)$.
Thus, letting $\beta=\min(\omega_F(x)\cap(c,\widetilde{b}))$, we get that $c<\beta<\widetilde{b}$. As $(c,\beta)$ cannot be a wandering interval, it follows from the homterval Lemma that $\exists\,j\ge1$ such that $F^j|_{(c,\beta)}$ is a diffeomorphism and $c\in f^{j}((c,\beta))$. As $F$ preserves the orientation, $F^j(c_+)<c<F^j(\beta)$. This implies by the definition of $\beta$ that $F^j(c_+)<c<\beta\le F^j(\beta)<\widetilde{b}$. Thus, $F^j$ has fixed point $p\in(c,\beta]$. So, $c\notin\omega_f(p)\in I$, contradicting the definition of $I$.

The proof when $\exists\,b<c$ such that $c\in\omega_f(x)$ $\forall\,x\in(c,b)\setminus\co_f^-(c)$ is completely analogous.

\end{proof}

\subsubsection*{Contracting Lorenz maps without periodic attractors} To prove Theorem~\ref{mainTheoremMTheoLORB} we shall consider two main cases: maps with or without periodic-like attractors. Firstly, we will prove the uniqueness of attractors for maps without periodic-like attractor (Corollary~\ref{CororllaryWithoutPerAt}). Thereafter, we study maps with periodic-like attractors, showing that we get one or, at most, two attractors.

\begin{Lemma}[Maps with $c$ in the boundary of a wandering interval]\label{LemmaWCritico}  
Let $f:[0,1]\setminus\{c\}\to[0,1]$ be a $C^3$ contracting Lorenz map with $Sf<0$.
If $c$ belongs to the boundary of a  wandering interval, then there is a transitive Cantor set $A$ such that $\omega_f(x)=A$
for almost every $x\in[0,1]$. Furthermore, $A=\omega_f(c_-)$ or $\omega_f(c_+)$ .
\end{Lemma}

\begin{proof}
Suppose that $(c-\delta,c)$ is a wandering interval for some $\delta>0$, the proof of the other case is analogous. Let $J=(a,c)$ be the maximal wandering interval containing $(c-\delta,c)$. As a cycle of intervals must contains $c$ in it interior, it follows that $f$ cannot have cycles of intervals. So, as $f$ does not have cycle of intervals, $\overline{\co_f^+(c_+)}$ is a totally disconnected compact set.

We observe that $f$ also cannot have a periodic-like attractors. Indeed, if $A$ is a periodic-like attractor for $f$, then it follows from Singer's Theorem \cite{Si} that $(c-\varepsilon,c)$ or $(c,c+\varepsilon)\subset\beta_f(A)$ for some $\varepsilon>0$. As $(a,c)$ is a wandering interval, $(c-\varepsilon,c)\not\subset\beta_f(A)$. So, $(c,c+\varepsilon)\subset\beta_f(A)$. As a consequence, $(a,c+\varepsilon)\cap f^j(I)=\emptyset$ $\forall\,j\ge1$, contradicting Corollary~\ref{CorollaryTirandoMane}.

As $(a,c)$ is a wandering interval, it follows from Corollary~\ref{CorollaryTirandoMane} that
\begin{equation}\label{Eq987634h}
c_+\in\omega_f(x)\text{ for almost every }x\in[0,1]
\end{equation}
As a consequence, $(c,b)$ cannot be a wandering interval $\forall\,b>c$. So, $a=c_-^\star$ and $c_+^\star=c$. Furthermore, using Theorem~\ref{Theorem1aa}, we get that
\begin{equation}\label{Eq876fvbnk}\omega_f(x)=
\begin{cases}
	\overline{\co_f^+(a)}\cup\overline{\co_f^+(c_+)} & \text{ if }a\in\omega_f(x)\vspace{0.1cm}\\
	\overline{\co_f^+(c_+)} & \text{ if }a\notin\omega_f(x)
\end{cases}
\end{equation}
for almost every $x\in[0,1]$.

It follows from (\ref{Eq987634h}) that $c_+\in\omega_f(J)=\omega_f(a)=\omega_f(c_-)$.
In particular, $\overline{\co_f^+(c_+)}\subset\omega_f(a)=\omega_f(c_-)$.
Thus, if $a$ is recurrent, we get $$a\in\omega_f(a)=\omega_f(J)=\omega_f(c_-)=\overline{\co_f^+(a)}\cup\overline{\co_f^+(c_+)}$$ and it follows from (\ref{Eq876fvbnk}) that $\omega_f(x)=\omega_f(c_-)$ for almost every $x\in[0,1]$. Furthermore, as $a\in\omega_f(a)=\omega_f(c_-)$ and $a\notin\co_f^-(c)$, it follows from Lemma~\ref{LemmaRecorrente} in the Appendix that $\omega_f(c_-)$ is a transitive Cantor set. Thus, to conclude the proof, we may assume that $a\notin\omega_f(a)$.

\begin{Claim}\label{Claim87ygbnnbg}
If $\exists\,\varepsilon>0$ such that $c\in\omega_f(x)$ for every $x\in(c,c+\varepsilon)\setminus\co_f^-(c)$, then $c\in\omega_f(c_+)$.
\end{Claim}
\begin{proof}[Proof of the claim]
As $c\in\omega_f(J)$, we have that $c\in\omega_f(x)$ $\forall\,x\in(a,c+\varepsilon)\setminus\co_f^-(c)$. Thus, let $I=(\alpha,\beta)$ be the maximal open interval such that $c\in\omega_f(x)$ for every $x\in I\setminus\co_f^-(c)$.
Note that $I$ is a nice interval.
Furthermore, if $F$ is the first return map to $[\alpha,\beta]$, then it follows from the maximality of $I$ that Dom$(F)=[\alpha,\beta]\setminus\{c\}$ and that $F$ is conjugated to a contracting Lorenz map $\widetilde{f}$ ($F$ is renormalizable to $\widetilde{f}$). If $c\notin\omega_f(c_+)$, then $F(c_-)=\alpha$, which implies that $F$ is not injective,  contradicting Lemma~\ref{LemmaFUGAdoCherry} (in the Appendix)  applied to $\widetilde{f}$.
\end{proof}

\begin{Claim}\label{Claimuu98recu}
$c\in\omega_f(c_+)$	
\end{Claim}
\begin{proof}[Proof of the claim]
Suppose that $c\notin\omega_f(c_+)$. Let $(\alpha_0,\beta_0)$ be the connected component of $[0,1]\setminus\overline{\co_f^+(c_+)}$. Observes that $\alpha_0<a$.
By Claim~\ref{Claim87ygbnnbg}, there exists $p\in(c,\beta_0)\setminus\co_f^-(c)$ such that $c\notin\omega_f(p)$.
Let $(\alpha_1,\beta_1)$ be the connected component of $[0,1]\setminus\overline{\co_f^+(p)}$ containing $c$. Notice that both  $(\alpha_0,\beta_0)$ and $(\alpha_1,\beta_1)$ are nice intervals.
Thus, $(\alpha,\beta):=(\alpha_0,\beta_0)\cap(\alpha_1,\beta_1)$ is a nice interval containing $(a,\beta)$.

Let $\NN_0=\{j\ge1\,;\,f^j(J)\subset(c,\beta)\}$ and $V=\bigcup_{j\in\NN_0}f^j(J)$.
For each $x\in V$, let $R(x)=\min\{\ell\ge1\,;\,f^\ell(x)\in(c,\beta)\}$
and let $I_x$ be the maximal open interval containing $x$ such that $f^{R(x)}|_{I_x}$ is a diffeomorphism and that $f^{R(x)}(I_x)\subset(a,\beta)$.
As $f$ preserves orientation and $\co_f^+(\beta)\cap(a,\beta)=\emptyset$, it follows from the maximality of $I_x$ that $I_x\cap I_y\ne\emptyset$ $\implies$ $I_x=I_y$ $\forall\,x,y\in V$.
Thus, the map $F:\bigcup_{x\in V}I_x\to(a,\beta)$ given by $F(y)=f^{R(x)}(y)$, for every $y\in I_x$ $\forall\,x\in V$, is well defined.
As $c\in\omega_F(x)$ for every $x\in V$. It follows from Lemma~\ref{Lemma375945194558} that $\omega_f(x)\supset\omega_F(x)=[a,\beta]$ for almost every $x\in[a,\beta]$, which is a contradiction.
\end{proof}

\begin{Claim}\label{Claimsususu}
If $\leb(V(a))>0$, then $a\in\overline{\co_f^+(c_+)}$, where  $V(a)=\{x\in[0,1]\setminus\co_f^-(c)\,;\,a\in\omega_f(x)\}$.
\end{Claim}
\begin{proof}[Proof of the claim]
Suppose by contradiction that $\leb(V(a))>0$ and $a\notin\overline{\co_f^+(c_+)}$. As we are assuming that $a\notin\omega_f(a)$ ($=\omega_f(J)=\omega_f(c_-)$), let $\alpha=\max\{x<c\,;\,x\in\overline{\co_f^+(c_-)}\cup\overline{\co_f^+(c_+)}\}$. Observe that $(\alpha,c)$ is a nice interval containing $a$. %As $\leb\circ f^{-1}\ll\leb$, $\leb(U)>0$, where $U=(\alpha,c)\cap V(a)$.

Let $F:(\alpha,c)^*\to(\alpha,c)$ be the first return map to $(\alpha,c)$, with $(\alpha,c)^*=\{x\in(\alpha,c)$ $;$ $\co_f^+(f(x))$ $\cap$ $(\alpha,c)$ $\ne$ $\emptyset\}$.
As $(\alpha,c)$ is a nice interval and $\co_f^+(c_\pm)\cap(\alpha,c)=\emptyset$, $F(I)=(\alpha,c)$ for every connected component $I$ of $(\alpha,c)^*$. Because of $\leb\circ f^{-1}\ll\leb$, we get that $\leb((\alpha,c)\cap V(a))>0$ and, as a consequence, it follows from Lemma~\ref{Lemma375945194558} that $\omega_f(x)\supset\omega_F(x)=[\alpha,c]$ for almost every $x\in(\alpha,c)$, which is a contradiction.
\end{proof}

It follows from Claim~\ref{Claimuu98recu} that $f(c_+)\in\omega_f(f(c_+))=\omega_f(c_+)=\overline{\co_f^+(c_+)}$.
As $f(c_-)\notin\co_f^-(c)$, it follows from Lemma~\ref{LemmaRecorrente} in the Appendix, and from the fact that $f(c_+)\notin\per(f)$, that $\omega_f(c_+)$ is a transitive Cantor set.
So, to finish the proof, notice that if $\leb(V(a))>0$, then we can use (\ref{Eq876fvbnk}) and Claim~\ref{Claimsususu} to obtain that $\omega_f(x)=\omega_f(c_+)=\overline{\co_f^+(c_+)}=\overline{\co_f^+(a)}\cup \overline{\co_f^+(c_+)}$ for almost every $x\in[0,1]$. On the other hand, if $\leb(V(a))=0$, then the proof follows directly from (\ref{Eq876fvbnk}).
\end{proof}

\begin{Proposition}\label{PropositioTudoNoCritico}
Let $f:[0,1]\setminus\{c\}\to[0,1]$ be a $C^3$ contracting Lorenz map with negative Schwarzian derivative and such that $c$ is not in the boundary of a wandering interval. If $f$ does not have an attracting  periodic-like orbit, then $c_-\in\omega_f(x)$ and $c_+\in\omega_f(x)$ for Lebesgue almost all $x\in[0,1]$. 
\end{Proposition}
\begin{proof}
Suppose that there exists $W\subset[0,1]$ with positive measure and such that $c_+\notin\omega_f(x)$ $\forall\,x\in W$ (the case where $c_-\notin\omega_f(x)$ for a positive set of points $x\in[0,1]$ is analogous).

As $\leb(\{x\,;\,c\notin\omega_f(x)\})=0$ (Corollary~\ref{CorollaryTirandoMane}), we get that $c_-\in\omega_f(x)$ for Lebesgue almost every $x\in W$. Thus, there is some $\varepsilon>0$ and $V\subset W$, with $\leb(V)>0$, such that
\begin{equation}\label{Eqoioi9987}\overline{\co_f^+(x)}\cap(c,c+\varepsilon)=\emptyset
\end{equation}
 and $c_-\in\omega_f(x)$ $\forall\,x\in V$. Furthermore, as $\leb\circ f^{-1}\ll\leb$, $\leb(V\cap(c-\delta,c))>0$ $\forall\delta>0$.

Notice that $\co_f^+(c_-)\cap(c,c+\varepsilon)=\emptyset$. Otherwise $\co_f^+(x)\cap(c,c+\varepsilon)\ne\emptyset$ $\forall\,x\in V$, because $f$ is continuous and $c\in\overline{\co_f^+(x)\cap(0,c)}$ $\forall\,x\in V$.

If $\exists\,\delta>0$ such that $c\in\omega_f(x)$ for every $x\in(c-\delta,c)\setminus\co_f^-(c)$ or for every $x\in(c,c+\delta)\setminus\co_f^-(c)$,
then it follows from  Lemma~\ref{Lemma85826jfuio6} that there is an interval $(a,b)\subset(c-\delta,c+\delta)$ containing $c$ such that $c_-$ and $c_+\in\omega_f(x)$ for every $x\in(a,b)\setminus\co_f^-(c)$.
This is a contradiction with the fact that $\leb(W\cap(a,c))>0$.

So, for each $\delta>0$ there exist $c-\delta<a_0<c<b_0<c+\delta$ such that $\overline{\co_f^+(a_0)}\not\ni c\notin \overline{\co_f^+(b_0)}$. Letting $(a_1,b_1)$ be the connected component of $[0,1]\setminus\overline{\co_f^+(a_0)}$ and $(a_2,b_2)$ being the connected component of $[0,1]\setminus\overline{\co_f^+(b_0)}$, we have that both $(a_1,b_1)$ and $(a_2,b_2)$ are nice intervals and that $c-\delta<a_1<c<b_2<c+\delta$. Thus, $(a_{\delta},b_{\delta}):=(a_1,b_1)\cap(a_2,b_2)$ is a nice interval containing $c$ and contained in $(c-\delta,c+\delta)$.

Firstly, suppose that $c_-$ is not recurrent, i.e., $c_-\notin\omega_f(c_-)$.
In this case, let $\delta\in(0,\varepsilon)$ be such that $\co_f^+(f(c_-))\cap(c-\delta,c+\delta)=\emptyset$ and $J:=(a_\delta,b_\delta)$. It is easy to see that $\co_f^+(a)\cap(a,b)=\emptyset$. Let 
$F:J^*\to J$ be the first return map to $J$, where $J^*=\{x\in J\,;\,\co_f^+(f(x))\cap J\ne\emptyset\}$, and  set $U=J^*\cap(a_\delta,c)$.
As $\co_f^+(c_-)\cap J=\emptyset$, it follows from Lemma~\ref{Lemma98fgsn54fg} that $F(I)=J$ for every connected component $I$ of $U$. 

Now suppose that $c_-$ is recurrent, $c_-\in\omega_f(c_-)$. In this case, consider the nice interval $(a_{\varepsilon},b_{\varepsilon})$.
As $\#\{j\ge0\,;\,f^j(x_0)\in(a_\varepsilon,c)\}=\infty$ if $x_0\in V$, it follows from Lemma~\ref{LemmaI1i1iate} that there is $a\in(a_\varepsilon,c)$,  such that $a\in\per(f)$ and $\co_f^+(a)\cap(a,b_\varepsilon)=\emptyset$. Set $b=b_\varepsilon$ and $J=(a,b)$. 

Let $U\subset(a,c)$, $F:U\to(a,b)$ and $R:U\to\NN$ be given by Lemma~\ref{Lemma90102555}. In this case, we also have $F(I)=J$ for every connected component of $U$.  

Note that, independently of $c_-$ being recurrent or not, $V\subset U$.
%Let $\cc_0\subset(c,b)$ be a finite set and $g:(c,b)\setminus\cc_0\to(c,b)$ be any $C^3$ orientation preserving local diffeomorphism with $S g<0$.
%Set $\cu=U\cup(c,b)\setminus\cc_0$ and $G:\cu\to(a,b)$ by
%$$G(x)=\begin{cases}
%F(x) & \text{ if }x\in U\\
%g(x) & \text{ if }x\in(c,b)\setminus\cc_0
%\end{cases}
%$$
%
%Because $G^n(x)=F^n(x)$ $\forall\,x\in V$ and $\forall\,n\in\NN$, we get $G(V)\subset V$.
%Let $\cp$ be the collection of connected components of $\cu$. Notice that $G(P)=(a,b)$ $\forall\,P\in\cp$.
Thus, as $SF<0$ and $Leb\big(\{x\in\bigcap_{n\ge0}F^{-n}(U)\,;\,\omega_f(x)\not\subset\{a,b\}\}\big)\ge Leb(V)>0$, it follows from Lemma~\ref{Lemma375945194558} that $\omega_F(x)=[a,b]$ for almost every $x\in V$. This implies that $\co_f^+(x)\cap(c,c+\varepsilon)\ne\emptyset$ for almost every $x\in V$, which contradicts (\ref{Eqoioi9987}).

\end{proof}

\begin{Corollary}\label{CororllaryWithoutPerAt}
Let $f:[0,1]\setminus\{c\}\to[0,1]$ be a $C^3$ contracting Lorenz map with $Sf<0$ and such that $c$ is not in the boundary of a wandering interval. If $f$ does not have attracting periodic-like orbits and $\leb(\BB_1(f))=0$, then $\omega_f(x)=\overline{\co_f^+(c_-)\cup\co_f^+(c_+)}$ for Lebesgue almost every $x\in[0,1]$.
\end{Corollary}
\begin{proof}
This Corollary follows directly from Theorem~\ref{Theorem1aa} and Proposition~\ref{PropositioTudoNoCritico} above.
As $c$ is not in the boundary of a wandering interval, $\cw_f=(\cc_f)_{\infty}^{\pm}=\emptyset$ and also $c^{\star}_\pm=c$, where $(\cc_f)_{\pm}=\{c_-,c_+\}$.
As a consequence, it follows from Theorem~\ref{Theorem1aa} that $\omega_f(x)=\overline{\co_f^+(c_-)}\cup\overline{\co_f^+(c_+)}$ for almost every $x\in[0,1]$, since, by Proposition~\ref{PropositioTudoNoCritico}, we have that $c_-$ and $c_+\in\omega_f(x)$ for almost every $x\in[0,1]$.
\end{proof}

\begin{Lemma}\label{LemmaBB1(f)}
		Let $f:[0,1]\setminus\{c\}\to[0,1]$ be a $C^3$ contracting Lorenz map with negative Schwarzian derivative. If $\leb(\BB_1(f))>0$, then there is a cycle of intervals $A$ such that $\omega_f(x)=A$ for almost every $x\in[0,1].$
\end{Lemma}
	\begin{proof}
Suppose that $\leb(\BB_1(f))>0$. Thus, $f$ has a cycle of intervals $A=I_1\cup\cdots\cup I_n$, $I_j=[a_j,b_j]$. It follows from Corollary~\ref{CorollaryFinitenessOfCycleOfIntervals} that $A$ is the unique cycle of intervals of $f$ and so, $\BB_1(f)=\{x\in[0,1]\,;\,\omega_f(x)=A\}$. Furthermore, $c\in(a_k,b_k)$ for some $1\le k\le n$ (Lemma~\ref{lematres}). 

If $\interior(\omega_f(c_-))$ or $\interior(\omega_f(c_+))$ $\ne$ $\emptyset$, it follows from Lemma~\ref{lematres} that that $\omega_f(c_-)$ or $\omega_f(c_+)=A$. In this case, it follows from Proposition~\ref{PropositioTudoNoCritico} that $\omega_f(x)= A$ for almost every $x\in[0,1]$, proving the lemma.

Thus, we may assume that $\interior(\omega_f(c_\pm))$ $=$ $\emptyset$.
That being so, let $I=(a,b)$ be a connected component of $A\setminus\overline{\co_f^+(c_-)\cup\co_f^+(c_+)}$. Note that $I$ is a nice interval.
As $\co_f^+(c_\pm)\cap I=\emptyset$ and $\leb(\{x\,;\,\omega_f(x)= A\supset I\})>0$, it follows from the interval dichotomy lemma (Lemma~\ref{tudoounadadershchw}) that $\omega_f(x)\supset I$ for almost every $x\in I$.
By the homterval lemma, either $c\in I$ or there is some $\ell\ge1$ such that $f^{\ell}|_I$ is a diffeomorphism and $c\in f^{\ell}(I)$.
Thus, $\omega_f(x)\supset f^{\ell}(I)$ for almost every $x\in I$ and also for almost every $x\in T:=f^{\ell}(I)$.
As a consequence, $\interior(\omega_f(x))\cap\interior(A)\ne\emptyset$ for almost every $x\in T$. Therefore, it follows from Lemma~\ref{LemmaFinitenessOfCycleOfIntervals} that 
$\omega_f(x)=A$ for almost every $x\in T$.
Let $T'=\{x\in T\,;\,\omega_f(x)=A\}$.
As $T$ is an open interval containing $c$, it follows from (\ref{Eqks9hvf}) that $\leb(\bigcup_{n\ge0}f^{-n}(T))=1$ and so, $\leb(\bigcup_{n\ge0}f^{-n}(T'))=1$, since $f_*\leb$ is absolutely continuous.
That is, $\leb(\BB_1(f))=1$ and $\omega_f(x)=A$ for almost every $x\in[0,1]$.
	\end{proof}

\begin{Lemma}\label{CherryA}
	Let $f:[0,1]\setminus\{c\}\to[0,1]$ be a $C^3$ contracting Lorenz map with $Sf<0$ and $\leb(\BB_1(f))=0$.
	If $\exists\,\delta>0$ such that $c_-$ and $c_+\in\omega_f(x)$ for every $x\in(c-\delta,c+\delta)\setminus\co_f^-(c)$, then $c_-\in\omega_f(c_+)$ and $c_+\in\omega_f(c_-)$.
\end{Lemma}
	\begin{proof}		
Let $(a,b)$ be the maximal open interval $J$ containing $c$ and such that $c\in\omega_f(x)$ for every $x\in J\setminus\co_f^-(c)$. Notice that $(a,b)$ is a nice interval.
\begin{claim}
	$c\in\omega_f(c_\pm)$
\end{claim}
\begin{proof}[Proof of the claim]
Suppose by contradiction that $c\notin\omega_f(c_-)$ (the case when $c\notin\omega_f(c_+)$ is analogous). In this case, set $V=(c,b)\setminus\co_f^-(c)\subset(a,b)$ and $R:V\to(a,b)$ by $R(x)=\min\{j\ge1\,;\,f^j(x)\in(c,b)\}$. For each $x\in V$, define $I_x$ as the maximal open interval $J$ containing $x$ and such that $f^{R(x)}(I_x)\subset (a,b)$. Let $x\in V$ and write $(p,q)=I_x$. If $f^{R(x)}(I_x)\subsetneqq(a,b)$, then $\exists0\le j<R(x)$ such that (1) $c=f^{j}(p)<b<f^{j}(x)$ and $a<f^{R(x)}(p_+)<b$ or (2) $a\le f^{j}(p)<f^{j}(q)=c$ and $a<f^{R(x)}(q_-)=f^{R(x)-j}(c_-)<b$. In the case (1), we get that $a<f^{R(x)}(p_+)<f^{R(x)-j}(b)<f^{R(x)}(x)<b$ which contradicts the fact that $(a,b)$ is a nice interval. The second case says that $a<f^{R(x)-j}(c_-)<b$, in contradiction to our assumption. 

Thus, $f^{R(x)}(I_x)=(a,b)$ $\forall\,x\in V$. As $(a,b)$ is a nice interval, it follows that $I_x=I_y$ whenever $I_x\cap I_y\ne\emptyset$.
Therefore, the map $F:V\to(a,b)$ given by $F(y)=f^{R(x)}(y)$ for $y\in I_x$ and $x\in V$ is well defined.
As $SF<0$, $V$ is $F$ positive invariant, $\leb(V)=|b-c|>0$ and $c\in\omega_F(x)$ $\forall\,x\in V$,
it follows from Lemma~\ref{Lemma375945194558} that $\omega_F(x)=[a,b]$ for almost every $x\in(a,b)$. As this implies that $\leb(\BB_1(f))>0$, we get a contradiction.
\end{proof}

As $f$ does not have a periodic-like attractor, there is some $n\ge1$ such that $f^n(c_\pm)\notin\co_f^-(c)$.
Moreover, as $c\in\omega_f(c_\pm)$ and $\exists m_0,m_1\ge n$ such that $f^{m_1}(c_-)$ and $f^{m_1}(c_+)\in (c-\delta,c+\delta)\setminus\co_f^-(c)$.
Thus, $c_-$ and $c_+\in \omega_f(f^{m_0}(c_-))=\omega_f(c_-)$.
In the same way,  $c_-$ and $c_+\in \omega_f(f^{m_0}(c_+))=\omega_f(c_+)$, proving the lemma.
		 
\end{proof}

\begin{Lemma}
\label{Lemma987tfevbntn9g}
	Let $f:[0,1]\setminus\{c\}\to[0,1]$ be a $C^3$ contracting Lorenz map with $Sf<0$, $\leb(\BB_1(f))=0$ and such that $c$ is not in the boundary of a wandering interval. If $f$ does not have attracting periodic-like orbits, then $c_-\in\omega_f(c_+)\text{ and }c_+\in\omega_f(c_-)$.
\end{Lemma}

\begin{proof}
Suppose by contradiction that $c_+\notin\omega_f(c_-)$, the other case being analogous. Let $\delta>0$ be so that $(c,c+\delta)\cap\overline{\co_f^+(c_-)}=\emptyset$.

If $c\in\omega_f(x)$ for every $x\in(c-\delta,c)\setminus\co_f^-(c)$ or for every $x\in(c,c+\delta)\setminus\co_f^-(c)$,
then it follows from  Lemma~\ref{Lemma85826jfuio6} that there is an interval $(a,b)\subset(c-\delta,c+\delta)$ containing $c$ such that $c_-$ and $c_+\in\omega_f(x)$ for every $x\in(a,b)\setminus\co_f^-(c)$. Thus, it follows from Lemma~\ref{CherryA} that $c_+\in\omega_f(c_-)$, contradicting our hypothesis. 

So, for each $\delta>0$ there exist $c-\delta<a_0<c<b_0<c+\delta$ such that $\overline{\co_f^+(a_0)}\not\ni c\notin \overline{\co_f^+(b_0)}$. Letting $(a_1,b_1)$ be the connected component of $[0,1]\setminus\overline{\co_f^+(a_0)}$ and $(a_2,b_2)$ being the connected component of $[0,1]\setminus\overline{\co_f^+(b_0)}$, we have that both $(a_1,b_1)$ and $(a_2,b_2)$ are nice intervals and that $c-\delta<a_1<c<b_2<c+\delta$. Thus, $(a_{\delta},b_{\delta}):=(a_1,b_1)\cap(a_2,b_2)$ is a nice interval containing $c$ and contained in $(c-\delta,c+\delta)$.

%It follows from {\color{magenta}Lemma~\ref{LemmaI1i1iate}} that there is $a\in(a_\delta,c)$,  such that $a\in\per(f)$ and $\co_f^+(a)\cap(a,b_\delta)=\emptyset$.
%Let $U\subset(a,c)$, $F:U\to(a,b)$ and $R:U\to\NN$ be given by Lemma~\ref{Lemma90102555}. In this case, we also have $F(I)=J$ for every connected component of $U$.  

Let $(c,b)^*=\{x\in(c,b)\,;\,\co_f^+(f(x))\cap(c,b)\ne\emptyset\}$. Note that $\leb((c,b)^*)=|c-b|$. Given $x\in(c,b)^*$, let $R(x)=\min\{j\ge1\,;\,f^j(x)\in(c,b)\}$.
Also, let $I_x$ be the maximal open interval containing $x$ and such that $f^{R(x)}|_{I_x}$ is a diffeomorphism and $f^{R(x)}(I_x)\subset(a,b)$. 

We claim that $f^{R(x)}(I_x)=(a,b)$. Indeed, write $(\alpha,\beta)=I_x$ and suppose by contradiction that $f^{R(x)}(I_x)\subsetneqq(a,b)$.
Thus, $\exists\,0\le j<R(x)$ such that $f^j(\alpha)=c$ and $f^{R(x)}(\alpha)=f^{R(x)-j}(c_+)\in(a,b)$ or $f^j(\beta)=c$ and $f^{R(x)}(\beta)=f^{R(x)-j}(c_-)\in(a,b)$.
First, assume that $f^j(\alpha)=c$. In this case, as $f$ is orientation preserving, we get that $f^j((I_x))=(c,f^j(\beta))$.
As $j<R(x)$, $f^j(x)\notin(c,b)$. So, $b\in f^j(I_x)$, but this contradicts the fact that $(a,b)$ is a nice interval. Now assume that $f^j(\beta)=c$. In such case, $c<f^{R(x)}(x)<f^{R(x)}(\beta)=f^{R(x)-j}(c_-)<b$, which contradicts that $(c,c+\delta)\cap\co_f^+(c_-)=\emptyset$.

As $(a,b)$ is a nice interval, it follows from the maximality of $I_x$ that $I_x=I_y$ whenever $I_x\cap I_y\ne\emptyset$. Thus, $F:(c,b)^*\to(a,b)$ given by $F(y)=f^{R(x)}(y)$ for $y\in I_x$ is well defined.
As $c\in\omega_F(x)$ for almost every $x\in(a,b)^*$, it follows from Lemma~\ref{Lemma375945194558} that $\omega_F(x)=[a,b]$ for almost every $x\in(a,b)$. This implies that $\omega_f(x)\supset[a,b]$ for almost every $x\in[0,1]$. That is, $\leb(\BB_1(f))=1$, contradicting our assumption.
\end{proof}

\begin{Corollary}\label{CorollaryMassa}
	Let $f:[0,1]\setminus\{c\}\to[0,1]$ be a $C^3$ contracting Lorenz map with $Sf<0$, $\leb(\BB_1(f))=0$ and such that $c$ is not in the boundary of a wandering interval. If $f$ does not have attracting periodic-like orbits, then
	\begin{enumerate}
		\item $c_-\in\omega_f(c_-)=\overline{\co_f^+(c_-)}=\overline{\co_f^+(c_+)}=\omega_f(c_+)\ni c_+$;
		\item $A:=\overline{\co_f^+(c_-)}$ is a transitive Cantor set;
		\item $\omega_f(x)=A$ for almost every $x\in[0,1]$.
			\end{enumerate}
	\end{Corollary}
	\begin{proof}
Item (1) is a direct consequence of Lemma~\ref{Lemma987tfevbntn9g}. To prove the second item, note that $\overline{\co_f^+(c_-)}=\overline{\co_f^+(v)}$, where $v=f(c_-)$.
So, $v\in\omega_f(v)=A$, which implies that $A$ is transitive. Moreover, as $v\notin\per(f)$ and $v\in\bigcup_{j\ge0}f^{-j}([0,1]\setminus\{c\})$, it follows from Lemma~\ref{LemmaRecorrente} in the Appendix that $A$ is perfect.
If $A$ is not totally disconnected, then $\interior(A)\ne\emptyset$ and this implies by Proposition~\ref{PropositioTudoNoCritico} that $\interior(\omega_f(x))\supset\interior(A)\ne\emptyset$ for almost every $x$. That is, $\leb(\BB_1(f))=1$, which contradicts our hypothesis. So, $A$ is compact, perfect and totally disconnected.
Finally, item (3) follows straightforwardly from Corollary~\ref{CororllaryWithoutPerAt}.

	\end{proof}

\subsubsection*{Contracting Lorenz maps with periodic attractors} 
By Singer's Theorem \cite{Si}, it follows that a contracting Lorenz map can have at most two attracting periodic-like orbits.
%In Figure~\ref{PontosFixosAtratores.png} we give some simple examples of contracting Lorenz maps with one or two periodic-like fixed points. In these  examples it is not difficult to guess that the union of the basins of attraction of the attracting periodic-like orbits contains almost every point.
Nevertheless, it is not obvious  that the union of the basins of attraction of the attracting periodic-like orbits contains almost every point.
%this is the case when a contracting Lorenz map has attracting periodic orbits with large periods. Assuming the non-flatness condition,  this was proved in \cite{StP}. Here we present a proof of this result without the additional hypothesis of non-flatness of the critical point.

%%%%%%%%%%%%%%%%%%%%%%%%%%%%%%%%%%%%%%%%%%%%%%%
%\begin{figure}
%\begin{center}\includegraphics[scale=.33]{PontosFixosAtratores.png}\\
%\caption{In the left hand side picture, we have a contracting Lorenz map with an attracting  fixed point $p$ whose basin of attraction is $(0,1)\setminus\co_f^-(c)$. In the center one, we have a contracting Lorenz map with two attracting fixed points: $q=0$ and $p\in(c,1)$. In this case, $\beta_f(q)=(0,c)$ and $\beta_f(p)=(c,1)$. Finally, in the right hand side picture, the contracting Lorenz map has a single attracting fixed-like point $p=c$ and its basin of attraction is $(0,1)$.}\label{PontosFixosAtratores.png}
%\end{center}
%\end{figure}
%%%%%%%%%%%%%%%%%%%%%%%%%%%%%%%%%%%%%%%%%%%%%%%

\begin{Lemma}\label{LemmaDuasBacias}
Let $f:[0,1]\setminus\{c\}\to[0,1]$ be a $C^3$ contracting Lorenz map with negative Schwarzian derivative. Suppose that $\{A_j\,;\,1\le j\le t\}$ are the periodic-like attractors of $f$, where $1\le t\le 2$ is the number of them. If $\delta>0$ such that $\leb(((c-\delta,c)\cup(c,c+\delta))\setminus(\bigcup_{j}\beta_f(A_j))=0$, then $\leb(\bigcup_{j}\beta_f(A_j))=1$.
\end{Lemma}
\begin{proof}
Let $J=(a,b)$ be the maximal open interval containing $c$ such that $\leb((a,b)\setminus(\beta_f(A_1)\cup\beta_f(A_2)))=0$. It is easy to check that the maximality of $(a,b)$ implies that $\big(\co_f^+(a)\cup\co_f^+(b))\cap(a,b)=\emptyset$ and also that $a,b\in\per(f)$. 

Similarly to the proof of Lemma~\ref{LemmaSinglePerAtt}, we claim that $\co_f^+(x)\cap J=\emptyset$ for almost   every $x\in[0,1]$. To show this, consider $g:[0,1]\setminus\cc_g\to[0,1]$, defined by 
\begin{equation}\label{EgGLorCont}
g(x)=
\begin{cases}
f(x) & \text{ if }x\notin(a,b)\\
\lambda_a (f(x)-f(a)) +f(a) & \text{ if }x\in(a,c)\\
\lambda_b (f(x)-f(b)) & \text{ if }x\in(c,b)\\
\end{cases}
\end{equation}
where $\lambda_a=(1-f(a))/|f((a,c))|$, $\lambda_b=f(b)/|f((c,b))|$ and  $\cc_g=\{c\}\cup\partial J$, see Figure~\ref{OperacaLorenz.pdf}.

%%%%%%%%%%%%%%%%%%%%%%%%%%%%%%%%%%%%%%%%%%%%%%
\begin{figure}
\begin{center}\includegraphics[scale=.22]{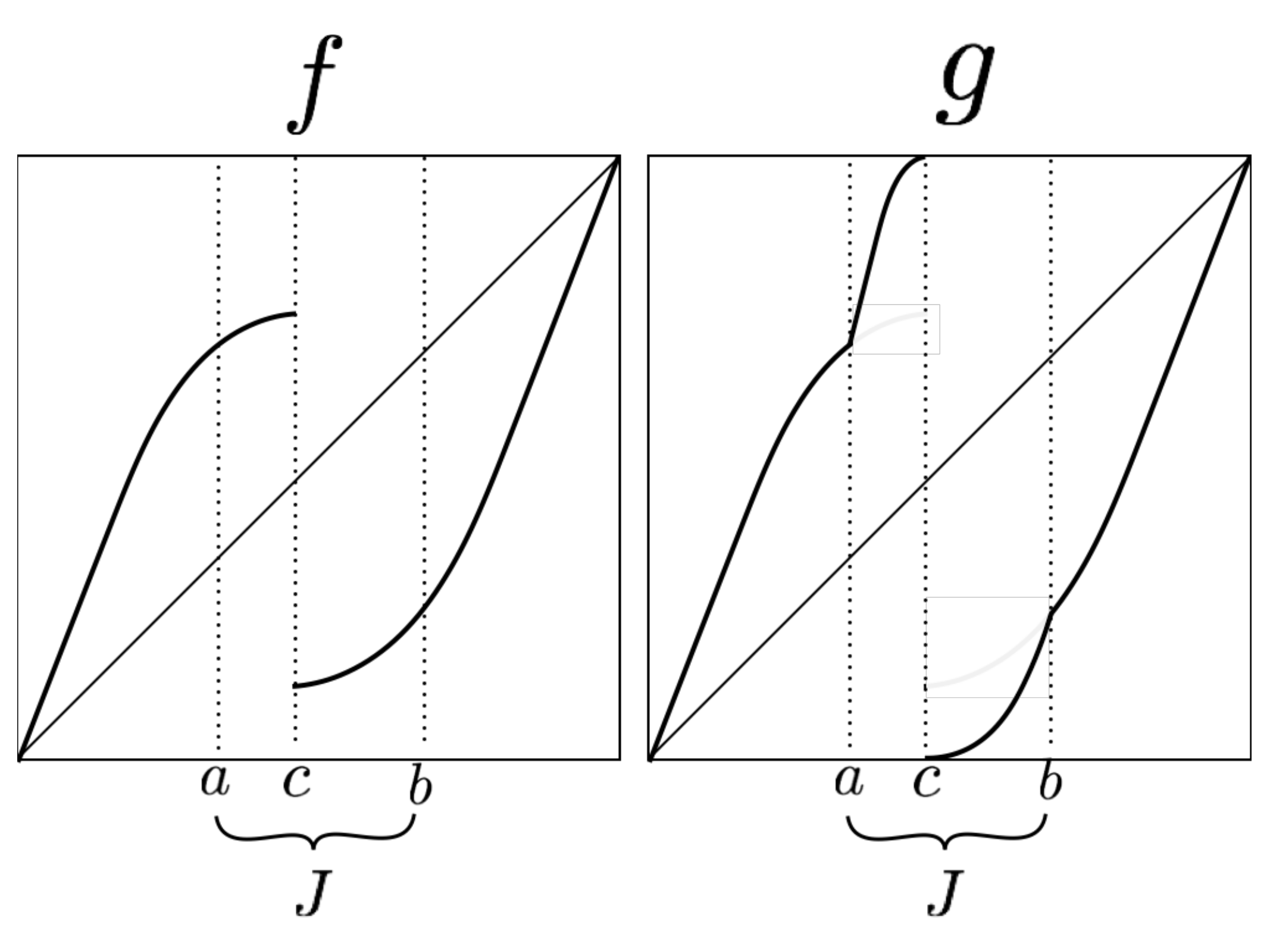}\\
\caption{}%In the left side we have a contracting Lorenz map and in the right side the associated map $g$ that is equal to $f$ outside the interval $J$, see Lemma~\ref{LemmaDuasBacias}.
\label{OperacaLorenz.pdf}
\end{center}
\end{figure}
%%%%%%%%%%%%%%%%%%%%%%%%%%%%%%%%%%%%%%%%%%%%%%

Letting $\cu=[0,1]\setminus\BB_0(f)$ and assuming by contradiction that $\leb(\cu)>0$, one can show (following the proof of Lemma~\ref{LemmaSinglePerAtt} in the Appendix) the existence of a connected component $I_{t_0}$ of $([0,1]\setminus J)\setminus(\co_f^+(a)\cup\co_f^+(b))$ such that $\leb(\{x\in\cu\,;\,\omega_g^+(x)\cap I_{t_0})\ne\emptyset\}\big\})>0$ and that $I_{t_0}\cap(\co_g^+(\cv_g)=\emptyset$. Thus, by the interval dichotomy lemma (Proposition \ref{tudoounadadershchw}), $\omega_g(x)\supset I_{t_0}$ for almost every $x\in I_{t_0}$.
By the homterval lemma, we get $n\ge0$ such that $g^n|_{I_{t_0}}$ is a diffeomorphism and $\cc_g\cap g^n(I_{t_0})=\{a,b,c\}\cap g^n(I_{t_0})\ne\emptyset$ and so, $g^n(I_{t_0})\cap J\ne\emptyset$. As a consequence, the orbit (with respect to $g$ and also to $f$) of almost every point of $\cu\cap I_{t_0}$ intersects $J$, contradicting the definition of $\cu$. \end{proof}

\begin{Proposition}[Periodic-like attractors for contracting Lorenz maps]\label{PropositionPerLikLor}
Let $f:[0,1]\setminus\{c\}\to[0,1]$ be a $C^3$ contracting Lorenz map with negative Schwarzian derivative. If $f$ has a periodic-like attractor $A_1$, then either $\leb(\beta_f(A_1))=1$ or else there is a second periodic-like attractor $A_2$ such that $\leb(\beta_f(A_1)\cup\beta_f(A_2))=1$.
\end{Proposition}
\begin{proof}
Let $A_1$ be a periodic-like attractor for $f$. By Singer's Theorem \cite{Si}, there is $\varepsilon>0$ such that $(c-\varepsilon,c)$ or $(c,c+\varepsilon)\subset\beta_f(A_1)$. Suppose for instance that $(c,c+\varepsilon)\subset\beta_f(A_1)$, the other case being similar. Let $(c,b)$ be the maximal open interval containing $(c,c+\varepsilon)$ and contained in $\beta_f(A_1)$.

Firstly suppose that $\exists\,\delta>0$ such that $(c-\delta,c)$ belongs to the basin of attraction of a periodic-like attractor $A_2$ ($A_2$ may be equal to $A_1$). We then apply Lemma~\ref{LemmaDuasBacias} and conclude the proof.

Thus, we may assume that, for every $\delta>0$, $(c-\delta,c)$ is not contained in the basin of attraction of a periodic attractor. In particular, this implies, by Singer's Theorem, that $A_1$ is the unique periodic-like attractor of $f$.

\begin{claim}
If $(a,c)$ is a homterval for some $a<c$, then $\leb(\beta_f(A_1))=1$.
\end{claim}
\begin{proof}[Proof of the claim]
Let $I=(a',c)$ be the maximal homterval containing $(a,c)$. As $I$ is a homterval, $\omega_f(x)=\omega_f(y)$, $\forall\,x,y\in I$. Thus,  if $I\cap\beta_f(A_1)\ne\emptyset$, then $I\subset\beta_f(A_1)$ and by Lemma~\ref{LemmaDuasBacias} we get that $\leb(\beta_f(A_1))=1$.

Now, let us show that the assumption $I\cap\beta_f(A_1)=\emptyset$ leads to a contradiction. Indeed, if $I\cap\beta_f(A_1)=\emptyset$, then $f^n(J)\cap(c,b)=\emptyset$ $\forall\,n\ge0$.
By the maximality of $I$, either $a'=0$ or $f^{\ell}(a')=c$ for some $\ell\ge1$. As $f^n(J)\cap(c,b)=\emptyset$ $\forall\,n\ge0$, we necessarily have $a'=0$. As $f^2|_I$ is a diffeomorphism, $c\notin f(I)=(0,\lim_{\varepsilon\downarrow0}f(c-\varepsilon))$. This implies that $f(I)\subset I$ and therefore, $(0,c)$ is contained in the basin of attraction of some fixed-like point $p\in[0,c]$, and as we are assuming that $A_1$ is the unique periodic-like attractor, $A_1=p$ and $I=(0,c)\subset\beta_f(p)$, contradicting the assumption.

%if $f^n(I)\cap I\ne\emptyset$ for some $n\ge1$ then $f^n(I)\subset I$ and so, $I$ has to be contained in the basin of attraction of a periodic-like attractor, contradiction our assumption.
%Therefore, $f^n(I)\cap I=\emptyset$ $\forall\,n\ge1$. As a consequence, $\leb(\{x\,;\,\co_f^+(x)\cap(a,b)=\emptyset\})\ge\leb(f(I))>0$. Now, applying the same argument of the proof of Lemma~\ref{LemmaDuasBacias} with $g$ given by the same formula (\ref{EgGLorCont}), we get that the orbit of almost every point of $f(I)$ intersects $(a,b)$, a contradiction.
\end{proof}

Let us suppose, by contradiction, that $\leb([0,1]\setminus\beta_f(A_1))>0$. Thus, it follows from the claim above that $f$ can not have a wandering interval $J$ with $c\in\partial J$. Also, by Lemma~\ref{lematres}, $f$ does not admit a cycle of intervals. Thus, it follows from Corollary~\ref{CorollaryTirandoMane} that
\begin{equation}\label{Eqokd5gna0kmn}c_-\in\omega_f(x)\text{ for almost all }x\in[0,1]\setminus\beta_f(A_1).
\end{equation}

Furthermore, it follows from Corollary~\ref{CorollaryCritnoPoco} that
	$\omega_f(x)\subset\overline{\co_f^+(c_-)}$ for almost all $x\in[0,1]\setminus\beta_f(A_1)$.
Thus, we have that
\begin{equation}
	\omega_f(x)=\overline{\co_f^+(c_-)},\text{ for almost all }x\in[0,1]\setminus\beta_f(A_1).
\end{equation}
In particular,
$$c_-\in\overline{\co_f^+(c_-)}\,\text{ and }\,\leb(V(a))>0,\,\,\forall a\in[0,c),$$
where $V(a)=\{x\in(a,c)\,;\,\#(\co_f^+(x)\cap(a,c))=\infty\}$.
Furthermore,  as $c_-=f^n(c_-)$ some $n\ge1$ would implies that $A_2=\{c,f(c_-),\cdots,f^{n-1}(c_-)$ is a second attracting periodic-like orbit with contradicts our assumption, we get that\begin{equation}
c_-\in\omega_f(c_-)
\end{equation}

%As $\omega_f(c_+)=A_1$, if $c_-\notin\omega_f(c_-)$ then $\exists\,0<a<c$ such that $\co_f^+(\cv_f)\cap(a,c)=\emptyset$, where $\cv_f=\{f(c_-),f(c_+)\}$. Furthermore, in this case, $c$ is an isolated point of $\overline{\co_f(c_\pm)}$, because $(c,c+\varepsilon)\subset\beta_f(A_1)$. As a consequence, it follows from Corollary~\ref{CorollaryOmegaIsolado} that $\leb(\{x\in[0,1]\setminus\beta_f(A_1)\,;\,c\in\omega_f(x)\})=0$, which is a contradiction with (\ref{Eqokd5gna0kmn}).
% As a conclusion, we have $c_-\in\omega_f(c_-)$.

\begin{claim}
If $\#A_1\ge2$, then $c\notin A_1$.
\end{claim}
\begin{proof}[Proof of the claim]
Suppose that $c\in A_1$ and let $n\ge2$ be the period of $c_+$.
Let $T:=(c,t)$ the be the maximal interval such that $f^n|_T$ in a diffeomorphism.
We claim that $f^n(t_-)>t$. Indeed, by the maximality of $T$ either $t=1$ or $\exists\,1\le\ell<n$ such that $f^n(t)=c$.
As $t=1$ implies that $n=1$, we conclude that $f^{\ell}(t)=c$, for some $1\le\ell<n$.
As $f$ is an orientation preserving map, we get that $f^{\ell}(T)=f^{\ell}((c,t))=(f^\ell(c_+),c))$ and also that $f^n(T)=(f^n(c_+),f^n(t_-))=f^{n-\ell}((f^{\ell}(c_+),c))=(f^{n}(c_+),f^{n-\ell}(c_-))$.
Thus, if $f^n(t_-)\le t$, then either $t_-$ is a periodic-like point with $c_-\in\co_f^+(t_-)$ or $c<f^{n-\ell}(c_-)=f^n(t_-)<t$.
The first case is impossible because it implies the existence of a second attracting periodic-like orbit, and we are assuming that $A_1$ is the unique one.
On the other hand, $c<f^{n-\ell}(c_-)=f^n(t_-)<t$ implies that $c<f^{k\,n}(x)<f^{(k-1)n}(x)<\cdots<f^n(x)<x$ for every $x\in(c,t)$ and this means that $T=(c,t)\subset\beta_f(A_1)$. As a consequence, $T'=f^{\ell}(T)\subset\beta_f(A_1)$ which contradicts the assumption that $(c-\delta,c)$ is not contained in the basin of attraction of a periodic-like attractor, $\forall\,\delta>0$.
\end{proof}

Notice that $A_1\cap(0,c)=\emptyset$ if and only if $A_1$ is an attracting fixed-like point $q\in[c,1]$. As we are assume that $p$ is the unique attracting periodic-like point it is easy to see that if $A_1$ is a fixed-like point, then $f(c_-)>c$ and that $(c,1)\supset\beta_f(q)$. In this case one can conclude easily that $\beta_f(q)=(0,1)$.

So, we may assume that $A_1\cap(0,c)\ne\emptyset$ and, by the claim just above, we get that $c\notin A_1$.

Let $J:=(p,q)$ be the connected component of $(0,1)\setminus A_1$ containing $c$. Thus, $J$ is a nice interval containing $c$ and $p\in\per(f)$.

By Singer's Theorem \cite{Si}, $(c,q)\subset\beta_f(A_1)$ and so, $\co_f^+(c_-)\cap(c,q)=\emptyset$. As we also have that $q\in\per(f)$,  
$c_-$ does not belong to the basin of attraction of $A_1$ (the unique periodic-like attractor of $f$) and $c\in\overline{\co_f^+(c_-)\cap(0,c)}$, we can consider $U\subset(p,c)$, $F:U\to(p,q)$ and $R:U\to\NN$ as in Lemma~\ref{Lemma90102555}.
Note that $V(p)\subset U$. Let $\cc_0\subset(c,q)$ be a finite set and $g:(c,q)\setminus\cc_0\to(c,q)$ be any $C^3$ orientation preserving local diffeomorphism with $S g<0$.
Set $\cu=U\cup(c,q)\setminus\cc_0$ and $G:\cu\to(p,q)$ by
$$G(x)=\begin{cases}
F(x) & \text{ if }x\in U\\
g(x) & \text{ if }x\in(c,q)\setminus\cc_0
\end{cases}
$$

Because $G^n(x)=F^n(x)$ $\forall\,x\in V(p)$ and $\forall\,n\in\NN$, we get $G(V(p))\subset V(p)$. Let $\cp$ be the collection of connected components of $\cu$. As, $G(P)=(p,q)$ $\forall\,P\in\cp$, $SG<0$ and $Leb\big(\bigcap_{n\ge0}G^{-n}(\cu)\big)\ge Leb(V(p))>0$, it follows from Lemma~\ref{Lemma375945194558} that $\omega_F(x)=\omega_G(x)=[p,q]$ for almost every $x\in V(p)$. In particular, the $\omega$-limit set of almost every $x\in V(p)$ is a cycle of intervals. This is a contradiction, as $f$ can not admit a cycle of intervals. Thus, we necessarily have $\leb([0,1]\setminus\beta_f(A_1))=0$, which concludes the proof.

\end{proof}

\subsubsection*{Proof of Theorem~\ref{mainTheoremMTheoLORB}} 
If $f$ has attracting periodic-like orbits, then the proof follows from Proposition~\ref{PropositionPerLikLor} above. 
If $c$ belongs to the boundary of a wandering interval, then theorem is a consequence of Lemma~\ref{LemmaWCritico}.

If $f$ does not admit attracting periodic-like orbits, $\leb(\BB_1(f))=0$ and $c$ is not in the boundary of a wandering interval, then it follows from Corollary~\ref{CorollaryMassa}. In particular, this is the case if $f$ is non-flat (see Definition~\ref{Non-flat} in the Appendix), it does not admit attracting periodic-like orbits and $\leb(\BB_1(f))=0$. Indeed, if $f$ is non-flat, it follows from Lemma~\ref{LemmaStP} in the Appendix that $c$ cannot belongs to the boundary of a wandering interval.

On the other hand, if $\leb(\BB_1(f))>0$ the Theorem is a corollary of Lemma~\ref{LemmaBB1(f)}.

\subsubsection*{Further comments} In the present paper we have dealt with metrical attractors. In this kind of attractor, the basin of attraction is required to have positive Lebesgue measure. Another viewpoint is to consider {\em topological attractors}, namely the ones whose basins of attraction are residual in some open set of the domain, see Milnor \cite{Milnor:1985ut}. Notice that the topological attractors and the metrical ones may not be the same, as in the case of wild attractors  \cite{BKNvS}.

For non-flat $C^3$ maps of the interval $[0,1]$ with negative Schwarzian derivative it is known that the number of topological attractors is bounded by the number of critical points, as shown by Guckenheimer \cite{G79} and Blokh and Lyubich \cite{BL89eg}. For maps with discontinuities, Brand\~ao showed in \cite{Br} the topological version of Theorem \ref{mainTheoremMTheoLORB}, that is, contracting Lorenz maps have either one single topological attractor or two attracting periodic orbits whose union of basins of attraction is a residual subset of the whole interval. In the context of topological attractors, the question of finiteness of the number of attractors for discontinuous maps with more than one critical point remains open.

%As an additional comment, we recall that Palis conjecture also encompasses parametrized families of dynamics, focusing in particular on   the dichotomy between deterministic and stochastic dynamics. There are important results concerning such a  dichotomy, including 
% \cite{ALM,AM,GS97,Hubbard:1993tq,J,Ly02,McM,Su,vSV2010}.

\section{Appendix}

\begin{Definition}[Non-flat]\label{Non-flat}
	A $C^3$ contracting Lorenz map $f:[0,1]\setminus\{c\}\to\RR$ is called {\em non-flat} if there exist $\varepsilon>0$, constants $\alpha, \beta\ge1$  and $C^3$  diffeomorphisms $\phi_{0}:[c-\varepsilon,c]\to \text{Im}(\phi_{0})$ and $\phi_{1}:[c,c+\varepsilon]\to\text{Im}(\phi_{1})$ such that $$f(x)=\begin{cases}f(c_-)+\big(\phi_{0}(x-c)\big)^\alpha&\text{ if }x\in(c-\varepsilon,c)\cap(0,1)\\
f(c_+)+\big(\phi_{1}(x-c)\big)^\beta&\text{ if }x\in(c,c+\varepsilon)\cap(0,1)\end{cases}$$
\end{Definition}

\begin{Lemma}[Lemma 3.36 of \cite{StP}]\label{LemmaStP} If $f:[0,1]\setminus\{c\}\to[0,1]$ is a $C^{2}$ non-flat contracting Lorenz map, then $f$ does not admit a wandering interval $J$ with $c\in\partial J$.
\end{Lemma}

\begin{proof}
	Suppose for instance that, $f$ has a wandering interval $J=(c,c+\varepsilon)$, with $\varepsilon>0$. So, we can modify $f$, see Figure~\ref{errantepng}, to coincide with the original map out of this interval, but being $C^2$ and non-flat (as defined in the introduction of the paper). This way, the modified map is a $C^2$ non-flat map with $f(J)$ being a wandering interval for it, but it can't happen, according to Theorem A of Chapter IV of \cite{MvS}.

%%%%%%%%%%%%%%%%%%%%%%%%%%%%%%%%%%%%%%%%%%%%
\begin{figure}
\begin{center}
\includegraphics[scale=.18]{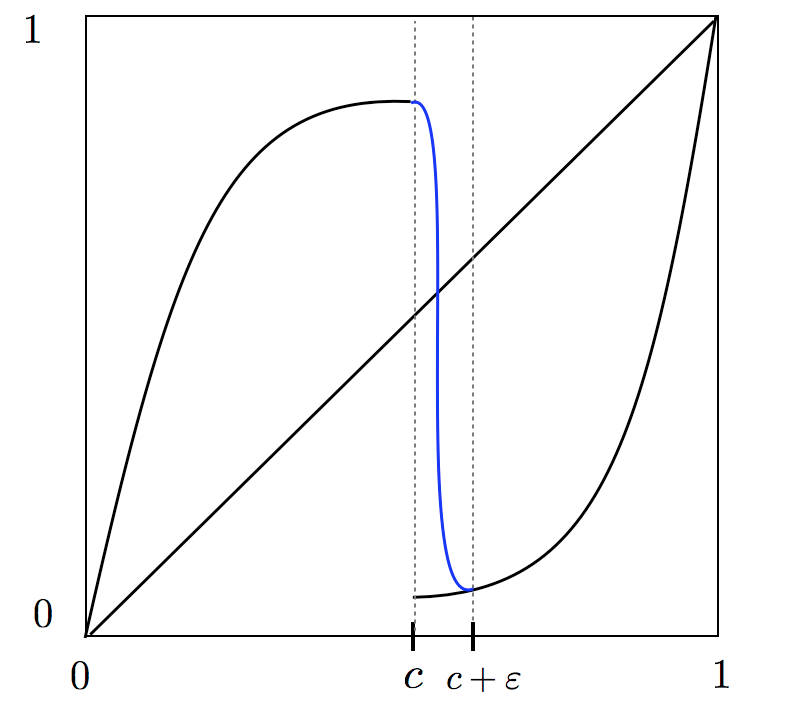}\\
\caption{}\label{errantepng}
\end{center}
\end{figure}
%%%%%%%%%%%%%%%%%%%%%%%%%%%%%%%%%%%%%%%%%%%%%%

\end{proof}

\begin{Lemma}\label{LemmaRecorrente}
	Let $\XX$ be a compact metric space and $f:U\to\XX$ a continuous map define in an open subset $U\subset\XX$. If $x\in\bigcap_{j\ge0}f^{-j}(U)$ is recurrent, $x\in\omega_f(x)$, then either $x\in\per(f)$ or $\omega_f(x)=\overline{\co_f^+(x)}$ is a perfect set.
\end{Lemma}
\begin{proof}
	Suppose that there are $q\in\omega_f(x)$ and $\varepsilon>0$ such that $\omega_f(x)\cap B_{\varepsilon}(q)=\{q\}$. Let $n_j\nearrow\infty$ be a sequence such that $q=\lim_{j}f^{n_j}(x)$.
	Because of $x\in\omega_f(x)$, we have that $\co_f^+(x)\subset\omega_f(x)$.
	Thus, $\exists\,j_0\ge1$ so that $f^{n_j}(x)=q$ $\forall\,j\ge j_0$. In particular, $q\in \bigcap_{j\ge0}f^{-j}(U)$.
	Therefore, $q=f^{n_{j_0+1}}(x)=f^{n_{j_0+1}-n_{j_0}}(f^{n_{j_0}}(x))=f^{n_{j_0+1}-n_{j_0}}(q)$.
	That is, $q\in\per(f)$. As $x$ is pre-periodic, because $f^{n_{j_0}}(p)=q$, and recurrent, it follows that $p$ is periodic.
\end{proof}

\begin{Lemma}[\cite{Br}]
\label{LemmaFUGAdoCherry}
Let $f:[0,1]\setminus\{c\}\to[0,1]$ be a contracting Lorenz map without periodic-like attractors. If $c\in \omega_f(x)$ for every $x\in(0,1)\setminus\co_f^-(c)$, then $f|_{(f(c_+),f(c_-))}$ is injective.
\end{Lemma}

\begin{proof}
Suppose that $c\in\omega_f(x)$ for every $x\in(0,1)$ and that there are $p_0\in(f(c_+),c)$ and $p_1\in(c,f(c_-))$ such that $f(p_0)=f(p_1)$. Let $p=f(p_0)=f(p_1)$ and $g:(f(c_+),f(c_-))\setminus\{c\}\to(f(c_+),f(c_-))$ given by
$g(x)=
\begin{cases}
p & \text{ if }x\notin(p_0,p_1)\\
f(x) & \text{ if }x\in(p_0,p_1)
\end{cases}
$, see Figure~\ref{CherryeCherry-like.png}.
%%%%%%%%%%%%%%%%%%%%%%%%%%%%%%%%%%%%%%%%%%%%%
\begin{figure}
  \begin{center}\includegraphics[scale=.2]{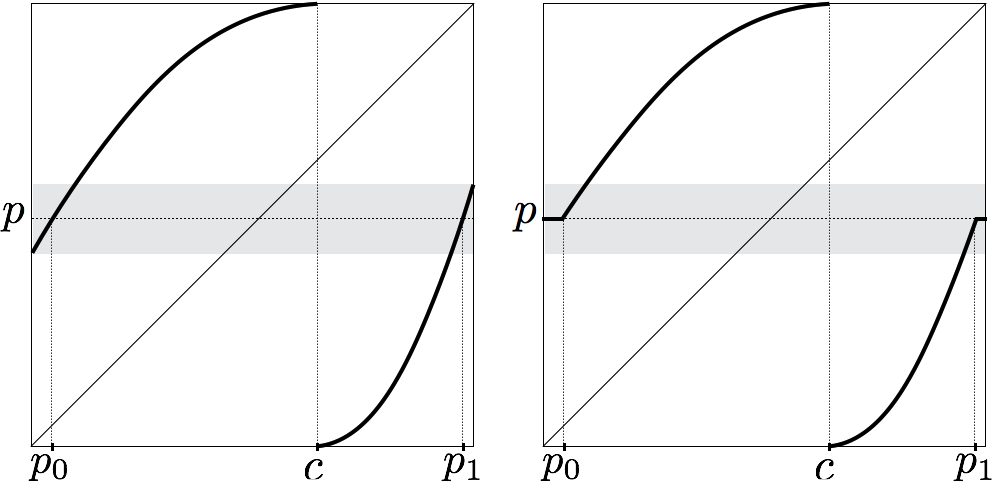}\\
 \caption{}\label{CherryeCherry-like.png}
  \end{center}
\end{figure}
%%%%%%%%%%%%%%%%%%%%%%%%%%%%%%%%%%%%%%%%%%%%%
We claim that there exists $n\ge0$ such that $g^n(p)\notin(p_0,p_1)$. Indeed, $g^n(p)=f^n(p)$ for every $n\ge0$ such that $g^j(p)\in(p_0,p_1)$ $\forall0\le j<n$. As $c\in\omega_f(x)$ $\forall\,x\in(f(c_+),f(c_-))$, eventually $f^n(p)$ is close to $c$ and so, $f^{n+1}(p)\notin(p_0,p_1)$. Therefore, $g^n(p)\notin(p_0,p_1)$ for some $n\ge0$, as claimed. This implies that $g^{n+1}(p)=p$. That is, $g$ has a periodic point and this implies (see \cite{GT85}) that $f$ also has a periodic point in $(f(c_+),f(c_-))$, which is impossible as $c\in\omega_f(x)$ for every $x\in(0,1)$.

\end{proof}

\begin{Lemma}\label{LemmaSinglePerAtt}
Let $f:[0,1]\to[0,1]$ be a $S$-unimodal map having an attracting periodic point $p$. Then $\leb(\beta_f(\co_f^+(p)))=1$.
\end{Lemma}
\begin{proof}
Let $0<c<1$ be the critical point of $f$ and let $\hat{p}=\max\co_f^+(p)$. As $Sf<0$ and $f$ has at most two fixed points, if $\hat{p}=0$, then $\omega_f(x)=0$ for every $x\in[0,1]$. Thus, we may assume that $\hat{p}\ne0$. In particular, $0$ cannot be an attracting periodic point.

Writing $(a,b)=J:=f^{-1}((\hat{p},1])$, we have that $a$ or $b\in\co_f^+(p)$. Furthermore, it follows from Singer's Theorem \cite{Si} that $p$ is the unique attracting periodic orbit of $f$ and that $\omega_f(x)=\co_f^+(p)$, for every $x\in J$ because $[a,c]$ or $[c,b]$ is contained in the local basin of $\co_f^+(p)$.

We claim that $\leb(\cu)=0$, where $\cu:=[0,1]\setminus\BB_0(f)=[0,1]\setminus\beta_f(\co_f^+(p))$. To show this, 
let's consider $g:[0,1]\setminus\cc_g\to[0,1]$ defined by 
$$g(x)=
\begin{cases}
f(x) & \text{ if }x\notin{J}\\
\lambda (f(x)-f(p)) +f(p) & \text{ if }x\in{J}
\end{cases}
$$
where $\lambda=(1-f(p))/(f(c)-f(p))$ and  $\cc_g=\{c\}\cup\partial J$.%, see Figure~\ref{paodeacucar.png}.

%%%%%%%%%%%%%%%%%%%%%%%%%%%%%%%%%%%%%%%%%%%%%%%
%\begin{figure}
%\begin{center}\includegraphics[scale=.2]{paodeacucar.png}\\
%\caption{The map $g$ associated to the unimodal map $f$ in the proof of Lemma~\ref{LemmaSinglePerAtt}.}
%\label{paodeacucar.png}
%\end{center}
%\end{figure}
%%%%%%%%%%%%%%%%%%%%%%%%%%%%%%%%%%%%%%%%%%%%%%%

Assume, by contradiction, that $\leb(\cu)>0$. Notice that $\co_f^+(x)=\co_g^+(x)$, $\forall\,x\in\cu$. In particular, $\BB_0(g)\cap\cu=\emptyset$.

Let $I_1,\cdots, I_t$ be the connected components of $([0,1]\setminus J)\setminus\co_f^+(p)$ and observe that $\cu\subset I_1\cup\cdots\cup I_t$. In particular, there is $t_0$ so that $\leb(\{x\in\cu\,;\,\#(\co_g^+(x)\cap I_{t_0})=\infty\}\big\})>0$. As $\cu\cap\BB_0(f)=\emptyset$, it follows from Lemma~\ref{Lemmaiuywe4jj} that $$\leb(\{x\in\cu\,;\,\omega_g(x)\cap I_{t_0}\ne\emptyset\})>0.$$

As $\co_g^+(g(a_\pm))=\co_g^+(g(b_\pm))=\co_f^+(p)$ and $\co_g^+(g(c_\pm))=\{0,1\}$, we get that $I_{t_0}\cap\co_g^+(\cv_g)=\emptyset$.
Thus, by the Interval Dichotomy Lemma (Lemma \ref{tudoounadadershchw}), $\omega_g(x)\supset I_{t_0}$ for almost every $x\in I_{t_0}$.
This implies that $I_{t_0}$ does not intersect the basin of attraction of a periodic attractor nor it is a wandering interval.
Thus, by the Homterval Lemma (Lemma \ref{LemmaHomterval}), there is some $n\ge0$ such that $g^n|_{I_{t_0}}$ is a diffeomorphism and $\{a,b,c\}\cap g^n(I_{t_0})\ne\emptyset$. But, this implies that $g^n(I_{t_0})\cap J\ne\emptyset$ and as a consequence the orbit (with respect to $g$ and also $f$) of almost every point of $\cu\cap I_{t_0}$ intersects $J$, contradicting the definition of $\cu$. Therefore, $\cu$ must be a zero measure set, proving the lemma.

\end{proof}

\begin{proof} of Theorem \ref{sunimodal}

By Singer's Theorem \cite{Si}, a $S$-unimodal map $f$ has at most one attracting periodic orbit. Moreover, by lemma \ref{LemmaSinglePerAtt}, the basin of attraction of this periodic orbit has full Lebesgue measure.
Thus, we may assume that $f$ does not have a periodic attractor. That is, $\BB_0(f)=\emptyset$.

As $\co_f^+(c_-)=\co_f^+(c_+)=\co_f^+(c)$, it follows from Corollary~\ref{Cor09876111} that $\omega_f(x)\subset\overline{\co_f^+(c)}$ for almost every $x\in[0,1]\setminus\BB_1(f)$. As we are assuming that $f$ does not have a periodic attractor, in particular it does not have a saddle-node, all periodic orbits are hyperbolic repellers. Applying  Mañe's Theorem \cite{Man85} (or Corollary \ref{CorollaryTirandoMane}), it follows that
\begin{equation}\label{Eqks9hvf}
	c\in\omega_f(x)\text{ for almost every }x\in[0,1].
\end{equation}

\begin{Claim}\label{Claim76ghj45}
	$c\in\overline{\per(f)}$.
\end{Claim}
\begin{proof}[Proof of the claim] If not, let $t=\max\overline{\per(f)}$ and set $I=(a,b)=f^{-1}((t,f(c)])$.
Note that $I$ is a nice interval containing $c$.
From (\ref{Eqks9hvf}), $\leb(I^*)=|a-b|$, where $I^*=\{x\in I\,;\,\co_f^+(f(x))\cap I\ne\emptyset\}$.
Letting $R(x)=min\{j\ge1\,;\,f^j(x)\in I\}$ for $x\in I^*$, and $F(x)=f^{R(x)}(x)$ being the first return map to $I$, it easy so check that $F(J)=I$ for every connected component $J$ of $I^*$ such that $c\notin\overline{J}$.
Furthermore, either $I^*=(a,c)\cup(c,b)$ or $I^*$ has infinitely many connected components.
In the first case, one can show that $t\in\per(f)$, $F=f^n$, where $n$ is the period of $t$, and that $f^n|_{[a,b]}$ is conjugate with an unimodal map without periodic attractors.
In particular, $f^n$ has a fixed point in $(a,b)$, contradicting the definition of $t$.
On the other hand, if $I^*$ has infinitely many connected components, then taking any connected component $J$ of $I^*$ such that $\overline{J}\subset I$, we get that $F(J)=I$ and so,
$\fix(F)\cap I\ne\emptyset$ which implies that $\per(f)\cap(t,1]\ne\emptyset$, contradicting again the definition of $t$.
\end{proof}

\begin{Claim}\label{Claimoiuygvhu8765}
If $c\notin\omega_f(c)$, then $\leb(\BB_1(f))=1$.	
\end{Claim}
\begin{proof}[Proof of the claim]
Suppose that $\overline{\co_f^+(f(c))}\cap(c-\delta,c+\delta)=\emptyset$ for some $\delta>0$. By Claim~\ref{Claim76ghj45} above, we can choose $p_0\in\per(f)\cap(c-\delta,c+\delta)$. Let $I=(p,p')$ be the connected component of $[0,1]\setminus\co_f^+(p_0)$. Thus, $I$ is a nice interval with $c\in I\subset(c-\delta/2,c+\delta/2)$ and $p,p'\in\co_f^+(p_0)$. Let $I^*=\{x\in I\,;\,\co_f^+(f(x))\cap I\ne\emptyset\}$, $R(x)=min\{j\ge1\,;\,f^j(x)\in I\}$ for $x\in I^*$, and $F(x)=f^{R(x)}(x)$ being the first return map to $I$, it easy so check that $F(J)=I$ for every connected component $J$ of $I^*$ and that $\leb(I^*)=\leb(I)=\leb(\bigcap_{j\ge0}F^{-j}(I))$. As $c\in\omega_F(x)$ for almost every $x\in I$, it follows from Lemma~\ref{Lemma375945194558} that $\omega_F(x)=\overline{I}$ for almost every $x\in I$. As $\leb\circ f^{-1}\ll\leb$ and as $\leb(\bigcup_{j\ge0}f^{-j}(I))=1$, it follows that $\leb(\BB_1(f))=1$.
\end{proof}

As a consequence of (\ref{Eqks9hvf}), $\overline{\co_f^+(c)}\subset\omega_f(x)$ almost surely. On the other hand, Corollary~\ref{Cor09876111} implies that $\omega_f(x)\subset\overline{\co_f^+(c)}$ for almost every $x\in[0,1]\setminus\BB_1(f)$. Thus, 
\begin{equation}\label{Eqyyy654edfghj}\omega_f(x)=\overline{\co_f^+(c)}\text{ for almost every }x\in[0,1]\setminus\BB_1(f).
\end{equation}

If $\leb(\BB_1(f))=0$, then  the theorem is proved by taking $A=\overline{\co_f^+(c)}$. Indeed, by (\ref{Eqyyy654edfghj}), we need only to verify that $\overline{\co_f^+(c)}$ is a perfect set, and this follows from  Claim~\ref{Claimoiuygvhu8765}, Lemma~\ref{LemmaRecorrente} and from that fact that  $c\notin\per(f)$, as $f$ does not have periodic attractors.

Therefore, we may suppose that $f$ has a cycle of intervals $\cu=I_1\cup\cdots\cup I_n$, $I_j=[a_j,b_j]$, and also that $\leb(\BB_1(f))>0$. It follows from Corollary~\ref{CorollaryFinitenessOfCycleOfIntervals} that this cycle of intervals is unique. Thus, $\BB_1(f)=\{x\in[0,1]\,;\,\omega_f(x)=\cu\}$. Furthermore, $c\in(a_k,b_k)$ for some $1\le k\le n$ (Lemma~\ref{lematres}). 
\begin{claim}
$\omega_f(x)=\cu$ for almost every $x\in[0,1]$.
\end{claim}
\begin{proof}[Proof of the claim]
Suppose by contradiction that $\leb([0,1]\setminus\BB_1(f))>0$. In this case, we have that $\interior(\omega_f(c))$ $=$ $\emptyset$. Otherwise, by the Lemma~\ref{lematres}, $\omega_f(c)$ is a cycle of intervals, and so $\omega_f(c)=\cu$. As a consequence, $c\in\omega_f(c)$ and, then, $\overline{\co_f^+(c)}=\omega_f(c)=\cu$. By (\ref{Eqyyy654edfghj}) and the definition of $\BB_1(f)$, we get $\omega_f(x)=\cu$ for almost every $x\in[0,1]$, i.e., $\leb([0,1]\setminus\BB_1(f))=0$, contradicting the assumption. 

As $\interior(\omega_f(c))=\emptyset$, we can choose an open interval $I=(a,b)$ contained in $\cu\setminus\omega_f(c)$. As $\co_f^+(c)\cap I=\emptyset$ and we know that $\leb(\{x\,;\,\omega_f(x)\supset\cu\supset I\})>0$, it follows from the interval dichotomy lemma (Lemma~\ref{tudoounadadershchw}) that $\omega_f(x)\supset I$ for almost every $x\in I$. By the homterval lemma, either $c\in I$ or there is some $\ell\ge1$ such that $f^{\ell}|_I$ is a diffeomorphism and $c\in f^{\ell}(I)$. Thus, $\omega_f(x)\supset f^{\ell}(I)$ for almost every $x\in I$ and also for almost every $x\in T:=f^{\ell}(I)$. As a consequence, $\interior(\omega_f(x))\cap\interior(\cu)\ne\emptyset$ for almost every $x\in T$. Therefore, it follows from Lemma~\ref{LemmaFinitenessOfCycleOfIntervals} that 
$\omega_f(x)=\cu$ for almost every $x\in T$. Let $T'=\{x\in T\,;\,\omega_f(x)=\cu\}$. As $T$ is an open interval containing $c$, it follows from (\ref{Eqks9hvf}) that $\leb(\bigcup_{n\ge0}f^{-n}(T))=1$ and so, $\leb(\bigcup_{n\ge0}f^{-n}(T'))=1$, since $f_*\leb$ is absolutely continuous. That is, $\leb([0,1]\setminus\BB_1(f))=0$, contradicting our assumption.
\end{proof}
It follows from the claim that $A:=\cu$ is a minimal attractor and $\beta_f(A)$ contains almost every point of $[0,1]$, concluding the proof of the theorem.
\end{proof}

\end{document}